\documentclass[12pt]{amsart}

\usepackage{amssymb}
\usepackage{hyperref} 
\usepackage{amsmath}
\usepackage{latexsym}
\usepackage{extarrows}
\usepackage{enumerate}
\usepackage{mathtools}
\usepackage{bm}
\usepackage[all]{xy}
\usepackage{tikz-cd}
\usepackage{xcolor}
\usepackage{mathrsfs}
\usepackage{euscript}
\usepackage{csquotes}
\usepackage{comment}
\usepackage{bbm}
\usepackage{cleveref}

\makeatletter
\@namedef{subjclassname@2020}{%
  \textup{2020} Mathematics Subject Classification}
\makeatother

\usepackage[american]{babel}

\usepackage[T1]{fontenc}
\calclayout

\theoremstyle{plain}

\newtheorem{theorem}{Theorem}[section]
\newtheorem{proposition}[theorem]{Proposition}

\newtheorem{lemma}[theorem]{Lemma}
\newtheorem{corollary}[theorem]{Corollary}

\newtheorem*{theorem*}{Theorem}

\theoremstyle{definition}

\newtheorem{definition}[theorem]{Definition}

\newtheorem{remark}[theorem]{Remark}
\newtheorem{remarks}[theorem]{Remarks}
\newtheorem{example}[theorem]{Example}
\newtheorem{examples}[theorem]{Examples}

\newcommand\Z{\mathbb{Z}}

\newcommand\T{\mathbb{T}}
\newcommand\C{\mathbb{C}}

\renewcommand\subset{\subseteq}
\renewcommand\phi{\varphi}

\newcommand\fix{\operatorname{fix}}
\newcommand\lin{\operatorname{lin}}

\newcommand\supp{\operatorname{supp}}

\renewcommand{\1}{\text{\usefont{U}{bbold}{m}{n}1}}
\MakeRobust{\1}

\begin{document}

%%%%% To ease editing, for IMPAN journals add:

\baselineskip=17pt
\setlength\parindent{0pt}

%%%%%%%%%%%%%%%%

\title[Compactifications of Abelian Group Bundles]{Compactifications of Abelian Group Bundles and a topological Halmos--von Neumann-type Theorem}

\author[D. Anglewitz]{Dustin Anglewitz}
\address{Dustin Anglewitz, Fachgruppe Mathematik und Informatik, Bergische Universität Wuppertal, Gaußstraße 20, 
42119 Wuppertal}
\email{dustin.anglewitz@uni-wuppertal.de}

\author[P. Hermle]{Patrick Hermle}
\address{Patrick Hermle, Fachgruppe Mathematik und Informatik, Bergische Universität Wuppertal, Gaußstraße 20,
42119 Wuppertal}
\email{patrick.hermle@uni-wuppertal.de}

\author[H. Kreidler]{Henrik Kreidler}
\address{Henrik Kreidler, Fakultät für Mathematik und Informatik, Universität Leipzig, Augustusplatz 10, 04109 Leipzig}
\email{henrik.kreidler@uni-leipzig.de}


\begin{abstract}
    The topological Halmos--von Neumann theorem is a fundamental result of topological dynamics that completely classifies continuous actions of abelian topological groups which are \emph{ergodic} and have \emph{discrete spectrum}. It rests on a correspondence between such actions (up to the choice of a distinguished point), compactifications of the acting group and subgroups of its dual group.\medskip
    
    \noindent We generalize both the classical statement and the underlying categorical equivalences to the framework of open group \emph{bundles} over locally compact base spaces. As a byproduct, we establish a Pontryagin-type duality between \'{e}tale Hausdorff abelian group bundles and open proper Hausdorff abelian group bundles, extending the classical duality between discrete and compact abelian groups.
\end{abstract}

\maketitle

\section*{Introduction}

Compactifying topological spaces is a classical and useful concept with many applications, e.g., to analysis, additive combinatorics, and ergodic theory. If we start from a topological space with additional algebraic structure, it is often desirable to construct compactifications which still admit the same structure. Concretely, given a topological group $G$, it is natural to ask for its \emph{group compactifications}, i.e., pairs $(H,c)$, where $H$ is a compact\footnote{All compact spaces in this article are assumed to be Hausdorff.} topological group and $c \colon G \rightarrow H$ is a continuous group homomorphism with dense range\footnote{We do not require $c$ to be injective here.}.\medskip

Such group compactifications are also interesting from a dynamical point of view. Let us call a pair $(K,\varphi)$ a \emph{topological dynamical system} if $\varphi \colon G \times K \rightarrow K, \, (t,x) \mapsto \varphi_t(x)$ is a continuous $G$-action on a compact space $K$. Then every group compactification $(H,c)$ of $G$ gives rise to a \emph{rotation system} $(H,\varphi_c)$ via $(\varphi_c)_t(x) = c(t)x$ for $t \in G$ and $x \in H$.\medskip

In fact, if $G$ is abelian, these rotation systems turn out to be the prototypical examples of topological dynamical systems which are \enquote{irreducible} and \enquote{structured}. This is part of the topological version of the famous Halmos--von Neumann theorem from ergodic theory (see \cite{halmosvonneumann}).\medskip

To make this precise, recall that to each topological dynamical system $(K,\varphi)$ we can associate its \emph{Koopman representation} $T^\varphi \colon G \rightarrow \mathscr{L}(\mathrm{C}(K)),\, t \mapsto T_t^{\varphi}$ on the Banach space of continuous functions $\mathrm{C}(K)$, where $T_t^{\varphi}f = f \circ \varphi_t^{-1}$ for $t \in G$, see \cite{EFHN2015} for this approach. The system $(K,\varphi)$ then
    \begin{enumerate}[(i)]
        \item is \emph{(topologically) ergodic} (\enquote{is irreducible}) if the fixed space 
            \begin{align*}
                \fix(T^{\varphi}) \coloneqq \{f \in \mathrm{C}(K) \mid T^{\varphi}_tf = f \textrm{ for every } t \in G\}
            \end{align*}
            only contains the constant functions.
        \item has \emph{discrete spectrum} (\enquote{is structured}) if the Banach space $\mathrm{C}(K)$ is generated by the eigenspaces
            \begin{align*}
                \ker(\chi - T^{\varphi}) \coloneqq \{f \in \mathrm{C}(K) \mid T^{\varphi}_tf = \chi(t)f \textrm{ for every } t \in G\},
            \end{align*}
        where $\chi \colon G \rightarrow \T$ is a continuous group homomorphism, i.e., an element of the dual group $G^*$. 
    \end{enumerate}
Further denoting by $\sigma_{\mathrm{p}}(T^{\varphi}) \coloneqq \{\chi \in G^*\mid  \ker(\chi -T^{\varphi}) \neq \{0\}\}$ the \emph{point spectrum} of the Koopman representation $T^{\varphi}$, the topological Halmos--von Neumann theorem can be stated as follows (see, e.g., \cite[Section 5.5]{Walters1975}, \cite[Chapter VIII]{DNP}, \cite[Theorem 3.1 and Section 4]{edeko} for $G= \Z$, and  \cite[Theorem 4.4]{kreidler-hermle} in the general case):
    \begin{theorem*}[Topological Halmos--von Neumann Theorem]
        Let $G$ be an abelian topological group. For ergodic topological dynamical systems with discrete spectrum with respect to $G$ the following assertions hold.
            \begin{enumerate}[(i)]
                \item Two such systems $(K_1,\varphi_1)$ and $(K_2,\varphi_2)$ are isomorphic precisely when the equality $\sigma_{\mathrm{p}}(T^{\varphi_1}) = \sigma_{\mathrm{p}}(T^{\varphi_2})$ holds.
                \item The point spectra $\sigma_{\mathrm{p}}(T^{\varphi})$ arising from such systems $(K,\varphi)$ are precisely the subgroups of the dual group $G^*$.
                \item Every such system $(K,\varphi)$ is isomorphic to a rotation system $(H,\varphi_c)$ defined by some group compactification $(H,c)$ of $G$.
            \end{enumerate}
    \end{theorem*}
Parts (i), (ii) and (iii) can be seen as the \emph{uniqueness}, \emph{realization} and \emph{representation} aspects of the result, respectively (cf. \cite[Introduction]{edeko}). In this form, the Halmos--von Neumann theorem is a classification result for irreducible and structured topological dynamical systems, and, with the help of topological models, can be used to deduce the ergodic theoretic counterpart for measure-preserving transformations (see, e.g.,   \cite[Theorems 3.8 and 4.5]{kreidler-hermle}). However, as outlined in \cite{kreidler-hermle}, the key coherences behind the topological Halmos--von Neumann theorem can also be expressed as an equivalence between three categories for a fixed abelian topological group $G$:
    \begin{enumerate}[(i)]
        \item the category of group compactifications $(H,c)$ of $G$,
        \item the category of pointed ergodic systems with discrete spectrum $(K,\varphi,x)$ over $G$, and
        \item the opposite category of subgroups $\sigma$ of the dual group $G^*$.
    \end{enumerate}
The complete result can be found in \cite[Theorem 4.25]{kreidler-hermle}. Thus, loosely speaking, both the compactifications of $G$ and the (pointed) ergodic systems with discrete spectrum with respect to $G$ are classified by the subgroups of $G^*$.\medskip

The goal of this article is to extend the above results from abelian topological groups to abelian topological group bundles. Here, an \emph{abelian topological group bundle} is a continuous surjection $p \colon G \rightarrow M$ from a topological space $G$ to a topological space $M$ such that the fibers $G_m \coloneqq p^{-1}(\{m\})$ for $m \in M$ carry the structure of an abelian group, multiplication and inversion are continuous, and the neutral elements depend continuously on the base point (see Subsection \ref{sectopgrpbundles}  below for the precise definition). We further restrict to the case that the base space $M$ is locally compact and that $p$ is also an open map.\medskip

%, given any open topological abelian group bundle $p \colon G \rightarrow M$ over a locally compact space $M$,
In this framework, we derive that the following categories are equivalent (see Theorem \ref{hvncat}):
    \begin{enumerate}[(i)]
        \item the category of group bundle compactifications $(q,c)$ of $p$,
        \item the category of relatively ergodic systems with relative discrete spectrum and distinguished section $(q,\varphi,\tau)$ over $p$, and
        \item the opposite category of subsheaves $F$ of the dual sheaf $\mathcal{D}_p$.
    \end{enumerate}
The involved concepts and the functors establishing these equivalences will be introduced and discussed in detail throughout this article. As a consequence of these correspondences we obtain a generalization of the topological Halmos--von Neumann theorem to actions of open topological abelian group bundles over locally compact spaces (see Theorem \ref{hvnmain}).\medskip

We also remark that, as a tool for our main results, Theorem \ref{equivproper} establishes a duality between \'{e}tale Hausdorff abelian group bundles and open proper Hausdorff abelian group bundles over a locally compact base space. This extends the classical and well-known duality between discrete and compact abelian groups to group bundles, and complements the Pontryagin duality theorem from \cite{Goeh2008} for second countable locally compact abelian group bundles with a continuous Haar system. 

\subsection*{Related Results and Context} Topological group bundles feature in the theory of operator algebras in the setting of groupoid C*-algebras (see \cite{IKRSW2021}, \cite{IKRSW2021b}, and \cite{VaWi2022} for some recent examples). Smooth group bundles also appear in the framework of differential geometry (see \cite{BMS2022} and \cite{CaRo2023}). It therefore is natural and useful to study and classify compactifications of group bundles.\medskip

A more concrete motivation comes from classical topological dynamics and ergodic theory of group actions. Many structural results in these disciplines rely on understanding \enquote{relative behavior} of a system with respect to one of its factors, see, e.g., Furstenberg's structure theorem for minimal distal systems (see  \cite{Furs1963}, \cite[Section V.3]{deVr1993}) as well as Furstenberg--Zimmer structure theory (see, e.g., \cite{Furs1977}, \cite{zimmer1976ergodic}, \cite{zimmer1976extension}, and more recently \cite{jamneshan2019fz}, \cite{EHK2021}, \cite{HaKr2026})  and Host--Kra structure theory (see \cite{HoKr2018}) for measure-preserving systems.\medskip

In ergodic theory the notion of discrete spectrum has a natural generalization to morphisms (known as \enquote{extensions}) of measure-preserving systems (see, e.g., \cite[Section 4]{jamneshan2019fz},  \cite[Section 8.2]{EHK2021}). Moreover, under certain ergodicity assumptions one can generalize the Halmos--von Neumann theorem to such extensions (see \cite[Sections 5 and 6]{zimmer1976extension} for extensions of ergodic systems, and \cite{austin-fund} as well as \cite{EJK2023} for the representation aspect in the larger class of relatively ergodic extensions).\medskip

However, in topological dynamics the \enquote{right} relative concept of discrete spectrum for factor maps $q \colon (K,\varphi) \rightarrow (M,\psi)$ is still missing, at least if one does not impose any irreducibility conditions. For so-called minimal and distal systems, a notion of relative discrete spectrum based on the structure of $\mathrm{C}(K)$ as a $\mathrm{C}(M)$-module is implicit in \cite{Knapp1967}. This has been adapted and extended in \cite{EdKr2021} to cover open factor maps with a topologically ergodic base system $(M,\psi)$.\medskip

As is also observed in \cite[Example 1.12]{EdKr2021}, factor maps $q \colon (K,\varphi) \rightarrow (M,\psi)$ over a fixed system $(M,\psi)$ can equivalently be described as actions of the \emph{action groupoid} $M \rtimes G$. Here, \emph{groupoids} generalize groups in the sense that multiplication is only partially defined (see, e.g., \cite{2013MitreaMonniaux} or \cite{IbRo2019} for an introduction). Concretely, $M \rtimes G$ is defined by the topological product space $M \times G$ with multiplication and inversion given by
    \begin{align*}
        &(\psi_t(m),s) \cdot (m,t) \coloneqq (m,st) \quad \textrm{ for }\, s,t \in G \, \textrm{ and }\, m \in M,\\
        &(m,t)^{-1} \coloneqq (\psi_t(m),t^{-1}) \quad \textrm{ for } t \in G \, \textrm{ and } m \in M.
    \end{align*}

In particular, in the special case that the action on $M$ is trivial, i.e., $\psi_t(x) = x$ for all $x \in M$ and $t \in G$, the action groupoid $M \rtimes G$ is simply given by the \emph{trivial group bundle} $\mathrm{triv}_M^G$ over $M$ with fiber $G$ (see Example \ref{examplesgrpbundles} (ii) below). Thus in this article -- \linebreak by introducing and studying group bundle actions with relative discrete spectrum -- we also cover this special case of factor maps of continuous group actions. See also \cite[Subsection 5.6]{ADKPL2026} for a recent Halmos--von Neumann-type theorem in the ergodic theoretic analogue of this setting. \medskip

Let us mention two more works closely related to our article. First, Edeko in \cite{edeko} proves a version of the Halmos--von Neumann theorem for not necessarily ergodic continuous $\Z$-actions with discrete spectrum. His results also involve the consideration of bundles over the \emph{fixed factor} (see Subsection \ref{ergodicitysec} below). The key difference to his setting is that in this article the notion of discrete spectrum is seen relative to this factor (i.e., with respect to a group bundle action), whereas it is the classical one (i.e., for group actions) in \cite{edeko}.\medskip
%The relation is discussed in more detail below. Nonetheless, both articles are related, as every topological dynamical system $(K,\varphi)$ over a topological abelian group $G$ having discrete spectrum in the classical case, canonically also gives rise to a continuous group bundle action  with relative discrete spectrum: If we denote by $M$ the fixed factor of $(K,\varphi)$, then $M$ with the trivial dynamics is a factor of $(K,\varphi)$, and hence defines an action of the trivial group bundle $\mathrm{triv}_M^G$. We note however, that the notion of relative discrete spectrum is  \medskip

Second, our article can also be seen as the counterpart\footnote{With the significant caveat that we only deal with \emph{abelian} group bundles here.} to the upcoming second author's PhD thesis \cite{Herm2026}, where a similar theory is developed for actions of topologically ergodic groupoids.

\subsection*{Organization of the Article} We now briefly outline the structure and contents of this article.\medskip

In Section \ref{prelimsec} we first set up some notation and terminology and recall basic notions from category theory. We then discuss topological bundles and some elementary, but useful results for proper and \'{e}tale bundles. Spaces of continuous fiber maps on topological bundles and the compact-open topology thereon are also recalled in this preliminary part of the article (see Definitions \ref{defcontfiber} and \ref{defcompactopen}).\medskip

Section \ref{secharmonic} is then devoted to topological group bundles. After the discussion of elementary concepts and examples in Subsection \ref{sectopgrpbundles}, we introduce the dual bundle of an open locally compact group bundle in Definition \ref{defdualbundle} and show the compatibility of our definition with the one used in the literature on groupoid C*-algebras (see Proposition \ref{agreetop}). Our considerations culminate in a duality between \'{e}tale Hausdorff abelian group bundles and open proper Hausdorff abelian group bundles over a locally compact base space (see Theorem \ref{equivproper}). In the last part of this section, we recall the well-known correspondence between \'{e}tale abelian group bundles and  sheaves (of abelian groups), see Theorem \ref{bunseciso}. We also introduce the notion of well-supported sheaves (see Definition \ref{definitionsupport}) as the counterpart to \emph{Hausdorff} \'{e}tale bundles (see Theorem \ref{wsuphausdorff}).\medskip

The two remaining sections are the main parts of this article. In Section \ref{sec3} we introduce group bundle compactifications (see Definition \ref{defgrpcomp}). To classify these, we define the dual sheaf of an open topological abelian group bundle in Definition \ref{defdualsheaf}. Every group bundle compactification then gives rise to a well-supported subsheaf of the dual sheaf in a natural and functorial way (see Definition \ref{defdiscertespectrum}). In Subsection \ref{subsectclasscomp} we then treat the converse direction and build group bundle compactifications from well-supported subsheaves of the dual sheaf (see Definition \ref{defcompactdualfunctor}). These functors then establish a categorical equivalence (see Theorem \ref{maintheoremcomp}).\medskip

Finally, we add a \enquote{dynamical perspective} in Section \ref{discspsec} by introducing the notion of relative discrete spectrum. The concept is first defined in a purely operator theoretic way in the framework of bounded strongly continuous representations on Banach bundles (see Definition \ref{defdiscretespectrum}), and then applied to topological dynamics of group bundles (see Definition \ref{defdiscretespectop}). We show that representations of open and proper Hausdorff abelian group bundles always have relative discrete spectrum (see Proposition \ref{comprepdiscsp}). As a consequence, rotation systems defined by group bundle compactifications have relative discrete spectrum (see Proposition \ref{compactifdiscr}). In Subsection \ref{ergodicitysec} we then introduce the notion of relative ergodicity for actions of group bundles, and discuss several equivalent descriptions in Proposition \ref{chartoperg}. As a tool, we also define the uniform enveloping semigroup bundle (see Definition \ref{defunifenvsem}) as a relative version of the classical concept of enveloping semigroups in Subsection \ref{univenvsec}. In particular, this leads to a characterization of relative discrete spectrum for relatively ergodic systems in terms of the uniform enveloping semigroup bundle (see Theorem \ref{unifenvcomp2}). In Subsection \ref{hvnsec} we introduce a suitable generalization of the category of pointed relatively ergodic systems with relative discrete spectrum to the group bundle case (see Definition \ref{defpointed}). We then finally extend the Halmos--von Neumann theorem to the group bundle case in the categorical and the classical version (see Theorems \ref{hvncat} and \ref{hvnmain}).

\subsection*{Acknowledgements} The authors thank Noa Bihlmaier, Bálint Farkas, Rainer Nagel, Jean Renault and Nick Ruoff for feedback and fruitful discussions. They are also grateful towards Alexander Furlong for bringing the preprint \cite{ADKPL2026} to their attention.
%Nick/Noa?

\subsection*{AI Disclosure Statement} No AI was involved in the research leading to this article. In an early version of this article, a built-in AI tool of the \LaTeX\, editor \emph{Overleaf}  was used for some language suggestions.

\section{Preliminaries}\label{prelimsec}

\subsection{Notation and Terminology}
In the following, all vector spaces are complex. The continuous linear maps on a normed space $E$ are denoted by $\mathscr{L}(E)$. Moreover, the notations $\1$ and $\1_X$ are used for the constant one-function on a set $X$. The neutral element of a group $G$ is written as $1_G$. We write $\T \coloneqq \{z \in \C \mid |z| = 1\}$ for the complex unit circle, and equip it with the natural topology and group structure.\medskip

We call a topological space quasi-compact if every open cover has a finite subcover, and compact if it is, in addition, a Hausdorff space. Moreover, all locally compact spaces are assumed to be Hausdorff.\medskip

For any topological space $\Omega$ we write $\mathrm{C}(\Omega)$ for the vector space of all continuous complex-valued functions on $\Omega$, and $\mathrm{C}_{\mathrm{b}}(\Omega)$ for the linear subspace of bounded continuous functions. If $\Omega$ is locally compact, then $\mathrm{C}_{\mathrm{c}}(\Omega)$ and $\mathrm{C}_0(\Omega)$ denote the linear subspaces of compactly supported functions and functions vanishing at infinity, respectively.\medskip

\subsection{Some Category Theoretical Language}

As seen in \cite{kreidler-hermle}, the Halmos--von Neumann theorem for groups can be stated elegantly using category theoretical language. We use the same philosophy in the group bundle case considered here. This short section reviews the relevant category theoretical concepts. Our exposition is based on \cite[Section 2]{kreidler-hermle}; more detailed introductions can be found, e.g., in \cite{lane1998categories} and \cite{Agor2023}.

\begin{definition}
    A \textbf{(locally small) category} $\mathscr{C}$ is given by
        \begin{enumerate}[(1)]
            \item a class $\mathrm{Obj}_{\mathscr{C}}$ of \textbf{objects},
            \item a set $\mathrm{Mor}_{\mathscr{C}}(X,Y)$ of \textbf{morphisms} for each pair $(X,Y)$ of objects,
            \item a \textbf{composition map}
            \begin{align*}
                \circ = \circ_{X,Y,Z} \colon \mathrm{Mor}_{\mathscr{C}}(Y,Z) \times \mathrm{Mor}_{\mathscr{C}}(X,Y) \rightarrow \mathrm{Mor}_{\mathscr{C}}(X,Z), \quad (\alpha, \beta) \mapsto \alpha \circ \beta
            \end{align*}
                for each triple $(X,Y,Z)$ of objects, and
            \item an \textbf{identity morphism} $\mathrm{id}_X \in \mathrm{Mor}_{\mathscr{C}}(X,X)$ for each object $X$,
        \end{enumerate}
    subject to the following conditions.
        \begin{enumerate}[(i)]
            \item For any objects $W,X,Y,Z$ the sets $\mathrm{Mor}_{\mathscr{C}}(W,X)$ and $\mathrm{Mor}_{\mathscr{C}}(Y,Z)$ are disjoint if $W \neq Y$ or $X \neq Z$.
            \item If $\alpha \in \mathrm{Mor}_{\mathscr{C}}(W,X)$, $\beta \in \mathrm{Mor}_{\mathscr{C}}(X,Y)$ and $\gamma \in \mathrm{Mor}_{\mathscr{C}}(Y,Z)$ for objects $W,X,Y,Z$, then
                \begin{align*}
                    \gamma \circ (\beta \circ \alpha) = (\gamma \circ \beta) \circ \alpha.
                \end{align*}
            \item If $\alpha \in \mathrm{Mor}_{\mathscr{C}}(X,Y)$ for objects $X$ and $Y$, then $\mathrm{id}_{Y} \circ \alpha = \alpha \circ \mathrm{id}_X$. 
            \end{enumerate}
\end{definition}

We discuss some important examples.

    \begin{examples}\label{examplescat}
        \begin{enumerate}[(i)]
            \item Sets with maps as morphisms form a category $\mathbf{Set}$.
            \item Abelian groups with group homomorphisms as morphisms define a category $\mathbf{Ab}$. By equipping these groups with the discrete topology, we obtain the category of discrete abelian groups $\mathbf{DA}$.
            \item Compact abelian groups with continuous group homomorphisms as morphisms give rise to a category $\mathbf{CA}$.
            \item Locally compact abelian groups with continuous group homomorphisms as morphisms give rise to a category $\mathbf{LCA}$.
            \item For a partially ordered set $A$, we define a category $\mathscr{C}_A$ by letting $\mathrm{Obj}_{\mathscr{C}_A} \coloneqq A$ and 
                \begin{align*}
                    \mathrm{Mor}_{\mathscr{C}_A}(a,b) \coloneqq \begin{cases} \{(a,b)\} &\textrm{ if } a \leq b,\\
                \emptyset & \textrm{ else},
                    \end{cases}
                \end{align*}
            for $a,b \in A$. With the obvious composition maps and identity morphisms, this yields a category $\mathscr{C}_A$.
            \item For a given topological space $M$, a special case of (ii) is given by taking the set of open subsets of $M$ ordered by \emph{reverse} set inclusion. We denote the resulting category by $\mathscr{O}_M$.
            \item If $\mathscr{C}$ is any category, then the \textbf{opposite category} $\mathscr{C}^{\mathrm{op}}$ has the same class of objects as $\mathscr{C}$, but $\mathrm{Mor}_{\mathscr{C}^{\mathrm{op}}}(X,Y) \coloneqq \mathrm{Mor}_{\mathscr{C}}(Y,X)$ for all objects $X,Y$. The order in the composition maps is also reversed, while the identity morphisms are the same as in $\mathscr{C}$. Clearly, $(\mathscr{C}^{\mathrm{op}})^{\mathrm{op}} = \mathscr{C}$.
        \end{enumerate}
    \end{examples}

It is common and convenient to write a morphism $\alpha \in \mathrm{Mor}_{\mathscr{C}}(X,Y)$ between two objects $X$ and $Y$ of a category $\mathscr{C}$ as an arrow $\alpha \colon X \rightarrow Y$. Moreover, one says that the morphism $\alpha$ is an \textbf{isomorphism} if there is a morphism $\beta \colon Y \rightarrow X$ with 
    \begin{align*}
        \beta \circ \alpha = \mathrm{id}_X \quad \textrm{ and } \quad \alpha \circ \beta = \mathrm{id}_Y.
    \end{align*}
In this case, the morphism $\beta$ is uniquely determined, denoted by $\alpha^{-1}$ and called the \textbf{inverse} of $\alpha$. Two objects $X$ and $Y$ in a category $\mathscr{C}$ are \textbf{isomorphic} if there exists an isomorphism $\alpha \colon X \rightarrow Y$.\medskip

The concept of a functor allows to formulate relations between different categories:

\begin{definition}
    For categories $\mathscr{C}$ and $\mathscr{D}$, a \textbf{functor} $F \colon \mathscr{C} \rightarrow \mathscr{D}$ is given by
        \begin{enumerate}[(1)]
            \item an assignment $F \colon \mathrm{Ob}_{\mathscr{C}} \rightarrow \mathrm{Ob}_{\mathscr{D}}, \, X \mapsto F(X)$, and
            \item a map $F = F_{X,Y} \colon \mathrm{Mor}_{\mathscr{C}}(X,Y) \rightarrow \mathrm{Mor}_{\mathscr{D}}(F(X),F(Y)), \, \alpha \mapsto F(\alpha)$ for any two objects $X,Y$ in $\mathscr{C}$, 
        \end{enumerate}
    such that the following conditions are satisfied:
        \begin{enumerate}[(i)]
            \item $F(\beta \circ \alpha) = F(\beta) \circ F(\alpha)$ for all morphisms $\alpha \colon X \rightarrow Y$ and $\beta \colon Y \rightarrow Z$  in $\mathscr{C}$.
            \item $F(\mathrm{id}_X) = \mathrm{id}_{F(X)}$ for each object $X$ in $\mathscr{C}$.
        \end{enumerate}
\end{definition}

For any category $\mathscr{C}$, the \textbf{identity functor} $\mathrm{Id}_{\mathscr{C}}$ assigns to each object and each morphism of $\mathscr{C}$ itself. Moreover, if $F_1 \colon \mathscr{B} \rightarrow \mathscr{C}$ and $F_2 \colon \mathscr{C} \rightarrow \mathscr{D}$ are two functors (between categories $\mathscr{B}, \mathscr{C}$ and $\mathscr{D})$, their \textbf{composition} $F_2 \circ F_1 \colon \mathscr{B} \rightarrow \mathscr{D}$, defined in the obvious way on objects and morphisms, is again a functor. Moreover, each functor $F \colon \mathscr{C} \rightarrow \mathscr{D}$ induces a functor $F^{\mathrm{op}} \colon \mathscr{C}^{\mathrm{op}} \rightarrow \mathscr{D}^{\mathrm{op}}$ between the induced opposite categories.\medskip

Here is a concrete example:

\begin{example}\label{basicdual}
    We obtain a functor $\mathrm{D} \colon \mathbf{LCA} \rightarrow \mathbf{LCA}^{\mathrm{op}}$ from the category of locally compact abelian groups to its opposite category by 
        \begin{enumerate}[(i)]
            \item assigning to each locally compact abelian group $H$ its Pontryagin dual 
                \begin{align*}
                    H^* \coloneqq \mathrm{Mor}(H,\T) = \{\varrho \colon H \rightarrow \T \mid \varrho \textrm{ is a continuous group homomorphism}\},
                \end{align*}
                equipped with pointwise defined multiplication as well as the compact-open topology, and
            \item mapping each continuous group homomorphism $\Phi \colon H_1 \rightarrow H_2$ to the pullback 
                \begin{align*}
                    \Phi^* \colon H_2^* \rightarrow H_1^*, \quad \varrho \mapsto \varrho \circ \Phi.
                \end{align*}
        \end{enumerate}
\end{example}

To compare different functors, one uses the concept of natural transformations.

\begin{definition}
    Let $F_1,F_2 \colon \mathscr{C} \rightarrow \mathscr{D}$ be two functors between categories $\mathscr{C}$ and $\mathscr{D}$. A \textbf{natural transformation} $\eta \colon F_1 \rightarrow F_2$ assigns to each object $X$ of $\mathscr{C}$ a morphism $\eta_X \colon F_1(X) \rightarrow F_2(X)$ in $\mathscr{D}$ in such a way that the diagram
            \[
		\xymatrix{
			F_1(X) \ar[r]^{\eta_X} \ar[d]_{F_1(\alpha)} & F_2(X)\ar[d]^{F_2(\alpha)}\\
			 F_1(Y) \ar[r]^{\eta_Y} & F_2(Y)
		}
	       \]	
    is commutative for any morphism $\alpha \colon X \rightarrow Y$ in $\mathscr{C}$. It is a \textbf{natural isomorphism} if, in addition, $\eta_X$ is an isomorphism for each object $X$ in $\mathscr{C}$. If such a natural isomorphism exists, then the functors $F_1$ and $F_2$ are \textbf{naturally isomorphic} (in symbols: $F_1 \simeq F_2$).
\end{definition}

Via natural isomorphisms we obtain the notion of categorical equivalence.

\begin{definition}
    Let $\mathscr{C}$ and $\mathscr{D}$ be categories. Two functors $F_1 \colon \mathscr{C} \rightarrow \mathscr{D}$ and $F_2 \colon \mathscr{D} \rightarrow \mathscr{C}$  are \textbf{essentially inverse} to each other if $F_2 \circ F_1 \simeq \mathrm{Id}_{\mathscr{C}}$ and $F_1 \circ F_2 \simeq \mathrm{Id}_{\mathscr{D}}$. If such functors exist, then the categories $\mathscr{C}$ and $\mathscr{D}$ are said to be \textbf{equivalent}.
\end{definition}

\begin{example}\label{exampleequiv}
    Consider the functor $\mathrm{D} \colon \mathbf{LCA} \rightarrow \mathbf{LCA}^{\mathrm{op}}$ from Example \ref{basicdual}. By Pontryagin duality (see, e.g., \cite[Section 4.3]{Foll2016}), assigning to each locally compact abelian group $H$ the evaluation map
        \begin{align*}
            \iota^H \colon H \rightarrow H^{**}, \quad x \mapsto [\varrho \mapsto \varrho(x)] 
        \end{align*}
    defines natural isomorphisms $\mathrm{Id}_{\mathbf{LCA}} \rightarrow \mathrm{D}^{\mathrm{op}} \circ \mathrm{D}$ and $\mathrm{D} \circ \mathrm{D}^{\mathrm{op}} \rightarrow \mathrm{Id}_{\mathbf{LCA}^{\mathrm{op}}}$. Therefore, the functors $\mathrm{D} \colon \mathbf{LCA} \rightarrow \mathbf{LCA}^{\mathrm{op}}$ and $\mathrm{D}^{\mathrm{op}} \colon \mathbf{LCA}^{\mathrm{op}}\rightarrow \mathbf{LCA}$ establish an equivalence between the categories $\mathbf{LCA}$ and $\mathbf{LCA}^{\mathrm{op}}$.\medskip

    Since the dual of a discrete group is compact and vice versa, the functors
        \begin{align*}
            \mathrm{D}|_{\mathbf{DA}} \colon \mathbf{DA} \rightarrow \mathbf{CA}^{\mathrm{op}} \quad \textrm{ and } \quad (\mathrm{D}|_{\mathbf{CA}})^{\mathrm{op}} \colon \mathbf{CA}^{\mathrm{op}}\rightarrow \mathbf{DA}
        \end{align*}
    define an equivalence of the category $\mathbf{DA}$ of discrete abelian groups and the opposite category of compact abelian groups $\mathbf{CA}^{\mathrm{op}}$. 
\end{example}

\subsection{Topological Bundles}\label{subsectopbun} Based on \cite[Subsection 1.1]{EJK2023} (see also \cite{Jame1989}) we now briefly recall the most fundamental concepts for this work:

\begin{definition}\label{deftopbundle}
A \textbf{topological bundle} over a topological space $M$ is a continuous surjection $q \colon \Omega \rightarrow M$ from a topological space $\Omega$ to $M$. We call $\Omega$ the  \textbf{total space} and $M$ the \textbf{base space} of the bundle $q$. Moreover, the subspace $\Omega_m \coloneqq q^{-1}(\{m\}) \subseteq \Omega$ is the \textbf{fiber} over $m \in M$.\medskip

A \textbf{morphism} $\Phi \colon q \rightarrow q'$ of topological bundles $q\colon \Omega \rightarrow M$ and $q' \colon \Omega' \rightarrow M$ over the same topological space $M$ is a continuous map $\Phi \colon \Omega \rightarrow \Omega'$ such that the diagram
\[
		\xymatrix{
			\Omega \ar[rr]^{\Phi}  \ar[dr]_q &  &\Omega' \ar[dl]^{q'}\\
			 & M & 
		}
	\]	
is commutative. We write $\Phi_m \coloneqq \Phi|_{\Omega_m} \colon \Omega_m \rightarrow \Omega_m'$ for the induced \textbf{fiber map} over $m \in M$.\medskip

    A topological bundle $q\colon \Omega \rightarrow M$ is called 
        \begin{enumerate}[(i)]
            \item \textbf{Hausdorff} if $\Omega$ and $M$ are Hausdorff spaces.
            \item \textbf{compact} if $\Omega$ and $M$ are compact spaces.
            \item \textbf{locally compact} if $\Omega$ and $M$ are locally compact spaces.
            \item \textbf{open} if $q$ is an open map.
            \item \textbf{proper} if $q$ is a proper map\footnote{This means that $q^{-1}(C)$ is quasi-compact for each quasi-compact subset $C \subseteq M$.}.
            \item \textbf{\'{e}tale} if $q$ is a local homeomorphism.\footnote{Recall that this means that each point $x \in \Omega$ has an open neighborhood $V \subseteq \Omega$ such that $q(V)$ is an open subset of $M$ and the restriction $q|_V \colon V \rightarrow q(V)$ is a homeomorphism.}
        \end{enumerate}
\end{definition}

Usually, the spaces $\Omega$ and $M$ will be locally compact. We also mostly work with open bundles, which can be characterized via the following approximation property (see \cite[Theorems 17.7 and 17.19]{AliBor}).

\begin{lemma}\label{charopen}
    For a topological bundle $q \colon \Omega \rightarrow M$ the following assertions are equivalent.
        \begin{enumerate}[(a)]
            \item The bundle $q$ is open.
            \item Whenever $(m_{\alpha})_{\alpha}$ is a net in $M$ converging to an element $m \in M$ and $x \in \Omega_m$, then there are a subnet $(m_{\beta})_{\beta}$ of $(m_{\alpha})_{\alpha}$ and elements $x_\beta \in \Omega_{m_{\beta}}$ for each $\beta$, such that $(x_{\beta})_{\beta}$ converges to $x$.
        \end{enumerate}
\end{lemma}

The subsequent characterization of proper bundles over a locally compact base space via net convergence is easy to check.

\begin{lemma}\label{charproper}
    For a topological bundle $q \colon \Omega \rightarrow M$ over a locally compact space $M$ the following assertions are equivalent.
        \begin{enumerate}[(a)]
            \item The bundle $q$ is proper.
            \item Whenever $(x_{\alpha})_{\alpha}$ is a net in $\Omega$ such that the net $(q(x_{\alpha}))_{\alpha}$ converges in $M$, the net $(x_{\alpha})_{\alpha}$ has a subnet converging in $\Omega$.
        \end{enumerate}
\end{lemma}

The following can be seen as the most basic example of a topological bundle.
\begin{example}\label{exampletriv}
    Let $M$ and $\Omega$ be topological spaces. Then the \textbf{trivial bundle} over $M$ with fiber $\Omega$ is the topological bundle
            \begin{align*}
             \mathrm{triv}_M^{\Omega} \colon M \times \Omega \rightarrow M, \quad (m, x) \mapsto m,
            \end{align*}
    where $M \times \Omega$ is equipped with the product topology. 
\end{example}

We now introduce some basic constructions for topological bundles.
\begin{definition}\label{constructiontopbund}
    Let $M$ be a topological space.
    \begin{enumerate}[(i)]
        \item For a topological bundle $q \colon \Omega \rightarrow M$ call $q_L \coloneqq q|_{q^{-1}(L)} \colon q^{-1}(L) \rightarrow L$ the \textbf{restriction} of $p$ to a subspace $L \subseteq M$.
        %, and abbreviate $\Omega_L \coloneqq q^{-1}(L)$
        \item If $q \colon \Omega \rightarrow M$ is a topological bundle  and $\Omega' \subseteq \Omega$ is a subset with $q(\Omega') = M$, then $q|_{\Omega'} \colon \Omega' \rightarrow M$ is called a \textbf{subbundle} of $q$.
        \item If $q \colon \Omega \rightarrow M$ and $q' \colon \Omega' \rightarrow M$ are topological bundles over $M$, then 
            \begin{align*}
                \Omega \times_{q,q'} \Omega' \coloneqq \{(x,x') \in \Omega \times \Omega' \mid q(x) = q'(x')\} \subseteq \Omega \times \Omega'
            \end{align*}
        defines a topological bundle over $M$ called the \textbf{fiber product} of $q$ and $q'$. We sometimes write $\Omega \times_M \Omega'$ instead of $\Omega \times_{q,q'} \Omega'$ if the maps $q$ and $q'$ are clear from the context. Notice that all the properties of bundles (i) -- (vi) listed in Definition \ref{deftopbundle} are compatible with fiber products, i.e., if $q$ and $q'$ have one of these properties, then so does their fiber product.
    \end{enumerate}
\end{definition}

The following basic \enquote{closed mapping theorem} has some important consequences. 
\begin{proposition}\label{properclosed}
    Let $q\colon \Omega \rightarrow M$ be a proper and $q' \colon \Omega' \rightarrow M$ a Hausdorff topological bundle over a locally compact space $M$. If $\Phi \colon q \rightarrow q'$ is a morphism of topological bundles, then $\Phi$ is a closed map.
\end{proposition}

\begin{proof}
    Let $D \subseteq \Omega$ be a closed subset and $(x_\alpha)_{\alpha}$ a net in $D$ such that $(\Phi(x_{\alpha}))_{\alpha}$ converges to some $y \in \Omega'$. Then the net $(q(x_{\alpha}))_{\alpha} = (q'(\Phi(x_{\alpha})))_{\alpha}$ converges to $q'(y)$. Since $q$ is proper, we can apply Lemma \ref{charproper} and pass to a subnet to assume that $(x_\alpha)_{\alpha}$ converges to some $x \in D$. By continuity of $\Phi$, the net $(\Phi(x_{\alpha}))_{\alpha}$ converges to $\Phi(x)$. Since $\Omega'$ is a Hausdorff space, we obtain that $y = \Phi(x) \in \Phi(D)$.
\end{proof}

The following is now obvious.

\begin{corollary}\label{automaticcontprop}
     Let $q\colon \Omega \rightarrow M$ be a proper and $q' \colon \Omega' \rightarrow M$ a Hausdorff topological bundle over a locally compact space $M$. If $\Phi \colon q \rightarrow q'$ is a bijective morphism of topological bundles, then it is a homeomorphism and hence an isomorphism of topological bundles.
\end{corollary}

We also obtain a \enquote{closed graph theorem} for proper bundles over locally compact spaces.

\begin{proposition}\label{lemmaclosedgraph}
    Let $q\colon \Omega \rightarrow M$ be  topological bundle and $q' \colon \Omega' \rightarrow M$ a proper Hausdorff topological bundle over a locally compact space $M$. If $\Phi \colon \Omega \rightarrow \Omega'$ satisfies $q' \circ \Phi = q$, then the following assertions are equivalent:
        \begin{enumerate}[(a)]
            \item $\Phi$ is continuous, i.e., a morphism of topological bundles.
            \item The graph of $\Phi$ is closed in $\Omega \times_M \Omega'$.
        \end{enumerate}
\end{proposition}
\begin{proof}
    The implication \enquote{(a) $\Rightarrow$ (b)} is trivial. Conversely, assume that the graph of $\Phi$ is closed in $\Omega \times_M \Omega'$. Since continuity can be checked locally and $M$ is locally compact, it suffices to prove that the restrictions $\Phi|_{q^{-1}(L)} \colon q^{-1}(L) \rightarrow (q')^{-1}(L)$ are continuous for each compact subset $L \subseteq M$. But this follows from the usual closed graph theorem for compact spaces (see, e.g., \cite[Exercise 5.1.22]{Sing2019}) since $(q')^{-1}(L)$ is a compact (Hausdorff) space for every compact subset $L \subseteq M$.
\end{proof}

For \'{e}tale bundles we have the following counterpart to Proposition \ref{properclosed}.
\begin{proposition}\label{lemmaopenmap}
    Let $q\colon \Omega \rightarrow M$ be an open and $q' \colon \Omega' \rightarrow M$ an \'{e}tale topological bundle over a topological space $M$. If $\Phi \colon q \rightarrow q'$ is a morphism of topological bundles, then $\Phi$ is an open map.
\end{proposition}
\begin{proof}
    It suffices to show that $\Phi$ is open locally. Take $x \in \Omega$. Since $q'$ is \'{e}tale, we find an open neighborhood $O$ of  $\Phi(x)$ such that $q'(O)$ is open in $M$ and $q'|_O \colon O \rightarrow q'(O)$ is a homeomorphism. Since $\Phi$ is continuous, the preimage $\Phi^{-1}(O)$ is an open neighborhood of $x$ in $\Omega$. Clearly, the diagram
        \[
		\xymatrix{
			\Phi^{-1}(O) \ar[rr]^{\Phi|_{\Phi^{-1}(O)}}  \ar[dr]_{q|_{\Phi^{-1}(O)}} &  & O \ar[dl]^{q'|_{O}}\\
			 & q'(O) & 
		}
	\]	
    is commutative. Since  $q'|_O$ is a homeomorphism and $q|_{\Phi^{-1}(O)}$ is open,  this implies that the restriction $\Phi|_{\Phi^{-1}(O)} \colon \Phi^{-1}(O) \rightarrow O$ is open as well.
\end{proof}

\begin{corollary}\label{lemmaiso}
     Let $q\colon \Omega \rightarrow M$ be an open and $q' \colon \Omega' \rightarrow M$ an \'{e}tale topological bundle over a topological space $M$. If $\Phi \colon q \rightarrow q'$ is a bijective morphism of topological bundles, then it is a homeomorphism and hence an isomorphism of topological bundles.
\end{corollary}

At various points throughout the article we use the concept of (local) continuous sections of a topological bundle. Here is the definition. 
\begin{definition}
    A \textbf{(continuous) section} of a topological bundle $q \colon \Omega \rightarrow M$ is a continuous map $\tau \colon M \rightarrow \Omega$ such that $q \circ \tau = \mathrm{id}_M$. For an open subset $U \subseteq M$, a section $\tau \colon U \rightarrow q^{-1}(U)$ of the restriction $q_U \colon q^{-1}(U) \rightarrow U$ is called a \textbf{local section} of $q$ (over $U$). Write $\Gamma_q(U)$ for the set of all local sections of $q$ over $U$.
\end{definition}

For \'{e}tale bundles there always exist \enquote{many} local sections as demonstrated by the following observation (see \cite[Proposition II.6.1]{MaMo1992}).

\begin{lemma}\label{etaleneighborhoodbase}
    Let $q \colon \Omega \rightarrow M$ be an \'{e}tale bundle. For every $x \in \Omega$ the collection 
        \begin{align*}
            \{\tau(U)\mid U \textrm{ is open neighborhood of } q(x), \tau \in \Gamma_q(U) \textrm{ with } \tau(q(x)) = x\}
        \end{align*}
    is an open neighborhood basis of $x$. In particular, every element of $\Omega$ is contained in the image of some local section.
\end{lemma}

\subsection{Spaces of Continuous Fiber Maps}
\enquote{Relativizing} the space of continuous maps between topological spaces to the framework of topological bundles, we introduce the following definition (see \cite[Section 2]{EdKr2021} and \cite[Subsection 2.2]{EJK2023}).  

\begin{definition}\label{defcontfiber}
For topological bundles $q \colon \Omega \rightarrow M$ and $q' \colon \Omega' \rightarrow M'$ define the \textbf{space of continuous fiber maps}
        \begin{align*}
            \mathrm{C}_q^{q'}(\Omega,\Omega') \coloneqq \bigsqcup_{\substack{m \in M\\m' \in M'}} \mathrm{C}(\Omega_m,\Omega'_{m'}).
        \end{align*}
    If $\vartheta \in \mathrm{C}(\Omega_m,\Omega'_{m'}) \subseteq \mathrm{C}_q^{q'}(\Omega,\Omega')$ for some $m \in M$ and $m' \in M'$, we set $\mathrm{s}(\vartheta) \coloneqq m$ and $\mathrm{r}(\vartheta) \coloneqq m'$. The mappings
        \begin{align*}
            \mathrm{s} &\colon \mathrm{C}_q^{q'}(\Omega,\Omega') \rightarrow M, \quad \vartheta \mapsto \mathrm{s}(\vartheta), \textrm{ and}\\
            \mathrm{r} &\colon \mathrm{C}_q^{q'}(\Omega,\Omega') \rightarrow M', \quad \vartheta \mapsto \mathrm{r}(\vartheta)
        \end{align*}
    are called the \textbf{source map} and \textbf{range map}, respectively.\medskip\\
    In the special case of $\Omega' = \C$ and $M' = \{\mathrm{pt}\}$ being a singleton set,  we use the abbreviation $\mathrm{C}_q(\Omega) \coloneqq  \mathrm{C}_q^{q'}(\Omega,\Omega')$.
\end{definition}
We equip the space of continuous fiber maps with a suitable topology (cf. \cite[Definition 2.3]{EdKr2021} and \cite[Definition 2.3]{EJK2023}).

\begin{definition}\label{defcompactopen}
    For topological bundles $q \colon \Omega \rightarrow M$ and $q' \colon \Omega' \rightarrow M'$ the topology on   $\mathrm{C}_q^{q'}(\Omega,\Omega')$
    generated by the sets 
        \begin{align*}
            \mathrm{W}(U,C,O) \coloneqq \{\vartheta \in \mathrm{s}^{-1}(U) \mid \vartheta(\Omega_{\mathrm{s}(\vartheta)} \cap C) \subseteq \Omega_{\mathrm{r}(\vartheta)} \cap O\}
        \end{align*}
    for $U \subseteq M$ open, $C \subseteq \Omega$ quasi-compact, and $O \subseteq \Omega'$ open, is called the \textbf{compact-open topology}.
\end{definition}

From now on spaces of continuous fiber maps are always equipped with the compact-open topology. The following characterization of convergence with respect to this topology (cf. \cite[Proposition 2.4]{EdKr2021} and \cite[Proposition 2.4]{EJK2023}) will be used on many occasions.

\begin{proposition}\label{netconv}
    Let $q \colon \Omega \rightarrow M$ and $q' \colon \Omega' \rightarrow M'$ be topological bundles, and assume that $q$ is locally compact. For a net $(\vartheta_{\alpha})_{\alpha}$ in $\mathrm{C}_q^{q'}(\Omega,\Omega')$ and $\vartheta \in \mathrm{C}_q^{q'}(\Omega,\Omega')$ the following assertions are equivalent.
        \begin{enumerate}[(a)]
            \item $\vartheta_{\alpha} \rightarrow \vartheta$ in $\mathrm{C}_q^{q'}(\Omega,\Omega')$.
            \item The subsequent two conditions are satisfied:
                \begin{enumerate}[(i)]
                    \item $\lim_{\alpha} \mathrm{s}(\vartheta_{\alpha}) = \mathrm{s}(\vartheta)$ in $M$.
                    \item Whenever $(\vartheta_{\beta})_{\beta}$ is a subnet of $(\vartheta_{\alpha})_{\alpha}$ and $(x_{\beta})_{\beta}$ is a net in $\Omega$ with $x_\beta \in \Omega_{\mathrm{s}(\vartheta_{\beta})}$ for each $\beta$, we have $\vartheta_{\beta}(x_{\beta}) \rightarrow \vartheta(x)$.
                \end{enumerate}
        \end{enumerate}
\end{proposition}

\begin{proof}
    Assume that (a) holds. It is clear that the source map $\mathrm{s}$ is continuous with respect to the compact-open topology, hence $\lim_{\alpha} \mathrm{s}(\vartheta_{\alpha}) = \mathrm{s}(\vartheta)$ in $M$. Now, let $(\vartheta_{\beta})_{\beta}$ be a subnet of $(\vartheta_{\alpha})_{\alpha}$ and $(x_{\beta})_{\beta}$ a net in $\Omega$ converging to $x \in \Omega$  as in (ii). Let further $O$ be an open neighborhood of $\vartheta(x)$ in $\Omega'$. Since $\Omega$ is locally compact and $\vartheta$ is continuous, we find a compact neighborhood $C$ of $x$ such that $\vartheta(\Omega_{\mathrm{s}(\vartheta)} \cap C) \subseteq \Omega_{\mathrm{r}(\vartheta)} \cap O$. But then there is an index $\beta_0$ such that $\vartheta_{\beta} \in \mathrm{W}(M,C,O)$ for all $\beta \geq \beta_0$. We may further assume that $x_{\beta} \in C$ for all $\beta \geq \beta_0$. But then $\vartheta_{\beta}(x_{\beta}) \in O$ for all $\beta \geq \beta_0$. This shows $\vartheta_{\beta}(x_{\beta}) \rightarrow \vartheta(x)$ as desired. \medskip

    Suppose now that (b) holds, but (a) does not. We then find $U \subseteq M$ open, $C \subseteq \Omega$ compact, and $O \subseteq \Omega'$ open, as well as a subnet $(\vartheta_{\beta})_{\beta}$ of $(\vartheta_{\alpha})_{\alpha}$ such that $\vartheta \in \mathrm{W}(U,C,O)$, but $\vartheta_{\beta} \notin  \mathrm{W}(U,C,O)$ for every $\beta$. Since $\lim_{\alpha} \mathrm{s}(\vartheta_{\alpha}) = \mathrm{s}(\vartheta)$ by (i), we may assume that $\mathrm{s}(\vartheta_{\beta}) \in U$ for all $\beta$. But then we find for each $\beta$ some $x_{\beta} \in  \Omega_{\mathrm{s}(\vartheta_{\beta})} \cap C$ with $\vartheta_{\beta}(x_{\beta}) \notin O$. Since $C$ is compact, we may further assume  -- by again passing to a subnet -- that $x \coloneqq \lim_{\beta} x_{\beta}$ exists in $C$. But then $\vartheta_{\beta}(x_{\beta}) \rightarrow \vartheta(x)$, hence $\vartheta(x) \notin O$. However, this contradicts $\vartheta \in \mathrm{W}(U,C,O)$.
\end{proof}

\begin{remark}
    In \cite[Section 2]{EdKr2021} (and then also \cite[Subsection 2.2]{EJK2023}) the compact-open topology is defined slightly differently, and then the characterization of net convergence from Proposition \ref{netconv} is only established for compact bundles $q$ (albeit \cite[Proposition 2.4]{EdKr2021} being erroneously stated for locally compact bundles). In particular, our definition here is equivalent to the one of \cite[Section 2]{EdKr2021} in the case of compact bundles, but also yields the \enquote{correct} notion of convergence in the locally compact case.
\end{remark}

In view of Lemma \ref{charopen}, the following is an easy consequence of Proposition \ref{netconv} (cf. \cite[Remark 2.5]{EdKr2021} and \cite[Corollary 2.6]{EJK2023}).

\begin{corollary}\label{corhausdorff}
    Let $q \colon \Omega \rightarrow M$ be an open locally compact bundle, and $q' \colon \Omega' \rightarrow M'$ a Hausdorff topological bundle. Then $\mathrm{C}_q^{q'}(\Omega,\Omega')$ is a Hausdorff space.
\end{corollary}

For \'{e}tale Hausdorff bundles over locally compact spaces we obtain the following description of the compact-open topology. This will be useful later.

\begin{lemma}\label{compopenetale}
    Let $q \colon \Omega \rightarrow M$ be an \'{e}tale Hausdorff bundle over a locally compact space $M$. For $f \in \mathrm{C}_q(\Omega)$ the sets
        \begin{align*}
            V(f,U,D,\varepsilon) \coloneqq \{g \in \mathrm{s}^{-1}(U) \mid |g(\tau(\mathrm{s}(g))) - f(\tau(\mathrm{s}(f)))| < \varepsilon \textrm{ for every } \tau \in D\} 
        \end{align*}
    for an open neighborhood $U$ of $\mathrm{s}(f)$,  $D \subseteq \Gamma_q(U)$ finite and $\varepsilon > 0$, define an open neighborhood basis of $f$.
\end{lemma}
\begin{proof}
    For fixed $f \in \mathrm{C}_q(\Omega)$ denote by $\mathcal{V}_f$ the collection  of subsets $V(f,U,D,\varepsilon)$ for open neighborhoods $U$ of $\mathrm{s}(f)$, finite subsets  $D \subseteq \Gamma_q(U)$ and numbers $\varepsilon > 0$. Observe that for $f \in \mathrm{C}_q(\Omega)$  any finite intersection of elements of $\mathcal{V}_f$ again contains an element of $\mathcal{V}_f$. To prove the claim, it therefore suffices to check that the topology $\mathcal{T}$ generated by $\bigcup_{f \in \mathrm{C}_q(\Omega)} \mathcal{V}_f$ coincides with the compact-open topology. Noting that $q$ is a locally compact bundle, we can apply Proposition \ref{netconv}.\medskip
    
    So let $(f_{\alpha})_{\alpha}$ be a net in $\mathrm{C}_q(\Omega)$ converging to some $f \in \mathrm{C}_q(\Omega)$ with respect to the compact-open topology. Let further $U$ be an open neighborhood of $\mathrm{s}(f)$,  $D \subseteq \Gamma_q(U)$ finite and $\varepsilon > 0$. Since $\lim_{\alpha} \mathrm{s}(f_{\alpha}) = \mathrm{s}(f)$, we find some index $\alpha_0$ such that $\mathrm{s}(f_{\alpha}) \in U$ for each $\alpha \geq \alpha_0$. For $\tau \in D$ the net $(\tau(\mathrm{s}(f_{\alpha})))_{\alpha \geq \alpha_0}$ converges to $\tau(\mathrm{s}(f))$, and hence $\lim_{\alpha} f_{\alpha}(\tau(\mathrm{s}(f_{\alpha})))= f(\tau(\mathrm{s}(f)))$. This shows $f_{\alpha} \in  V(f,U,D,\varepsilon)$ for large enough $\alpha$.\medskip

    Conversely, assume that a net $(f_{\alpha})_{\alpha}$ in $\mathrm{C}_q(\Omega)$ converges to some $f \in \mathrm{C}_q(\Omega)$ with respect to $\mathcal{T}$. Clearly, $\lim_{\alpha} \mathrm{s}(f_{\alpha})_{\alpha} = \mathrm{s}(f)$. To verify condition (ii) of Proposition \ref{netconv} (b), take a subnet $(f_{\beta})_{\beta}$ of  $(f_{\alpha})_{\alpha}$ and $x_{\beta} \in \Omega_{\mathrm{s}(f_{\beta})}$ for each $\beta$ such that $\lim_{\beta} x_{\beta} = x \in \Omega$. Then $x \in \Omega_{\mathrm{s}(f)}$. By Lemma \ref{etaleneighborhoodbase} we find an open neighborhood $U$ of $\mathrm{s}(f)$ and a local section $\tau \in \Gamma_q(U)$ with $\tau(\mathrm{s}(f)) = x$. Since (by the same lemma) the set $\tau(U)$ is an open neighborhood of $x$, we find some index $\beta_0$ with $x_{\beta} = \tau(\mathrm{s}(f_{\beta}))$ for all $\beta \geq \beta_0$. But then it is clear from the definition of $\mathcal{T}$ that
        \begin{align*}
            \lim_{\beta} f_{\beta}(x_{\beta}) =  \lim_{\beta} f_{\beta} (\tau(\mathrm{s}(f_{\beta}))) = f(\tau(\mathrm{s}(f))) = f(x),
        \end{align*}
    as desired.
\end{proof}

 For compact bundles we also quote an Arzel\'a--Ascoli-type theorem from \cite{EdKr2021}. Here, for continuous surjections $q \colon \Omega \rightarrow L$ and $q' \colon \Omega' \rightarrow L'$ from uniform to compact spaces, a subset $E \subseteq \mathrm{C}_{q}^{q'}(\Omega,\Omega')$ is \textbf{uniformly equicontinuous} if for each entourage $W$ of $\Omega'$ there exists an entourage $V$ of $\Omega$ such that $(\vartheta(x),\vartheta(y)) \in W$ for each $\vartheta \in E$ and all $(x,y) \in V$ with $q(x) = q(y) = \mathrm{s}(\vartheta)$. Recalling that any compact space carries a unique uniform structure, we have the following result (see   \cite[Theorem 3.13]{EdKr2021}).

\begin{theorem}\label{arzasc}
    Let $q \colon \Omega \rightarrow M$ be an open compact bundle, and $q' \colon \Omega' \rightarrow M'$ a continuous surjection from a Hausdorff uniform space to a compact space $M'$. Then a subset $E \subseteq \mathrm{C}_{q}^{q'}(\Omega,\Omega')$ is precompact if and only if both of the subsequent two assertions are satisfied.
        \begin{enumerate}[(i)]
            \item $\bigcup \{\mathrm{im}(\vartheta)\mid \vartheta \in E\} \subseteq \Omega'$ is precompact.
            \item $E$ is uniformly equicontinuous.
        \end{enumerate}
\end{theorem}

\section{Some Relative Harmonic Analysis}\label{secharmonic}

\subsection{Topological Semigroup and Group Bundles}\label{sectopgrpbundles}
We now recall the concept of topological (semi)group bundles (see, e.g., \cite[Section 3.1]{EJK2023}).

\begin{definition}\label{defsgrpbundle}
    A topological bundle $q \colon H \rightarrow M$ over a topological space $M$ is a
    \textbf{topological semigroup bundle} if each fiber $H_m$ for $m \in M$ is equipped with the structure of a semigroup such that the mapping
    \begin{align*}
        \cdot \colon H \times_{M} H \rightarrow H, \quad (x, y) \mapsto x \cdot y
    \end{align*}
    is continuous. Is is a \textbf{group bundle} if, in addition, each fiber $H_m$ for $m \in M$ is a group, and the maps
        \begin{align*}
        \,^{-1} &\colon H \rightarrow H, \quad x \mapsto x^{-1} \quad \textrm{ and } \quad   1\colon M \rightarrow H, \quad m \mapsto 1_{H_m}
        \end{align*}
    are also continuous. We call a topological (semi)group bundle $q \colon H \rightarrow M$ \textbf{abelian} if all the fiber (semi)groups $H_m$ for $m \in M$ are abelian. \medskip

A \textbf{morphism} $\Phi \colon q \rightarrow q'$ between topological (semi)group bundles $q\colon H \rightarrow M$ and $q' \colon H' \rightarrow M$ over the same topological space $M$ is a morphism of topological bundles having the additional property that the fiber maps $\Phi_m \colon H_m \rightarrow H_{m}'$ are (semi)group homomorphisms for each $m \in M$.
\end{definition}

Here are a few simple examples (cf. \cite[Examples 3.2]{EJK2023}).
\begin{examples}\label{examplesgrpbundles}
    \begin{enumerate}[(i)]
        \item A topological group defines a topological group bundle over a singleton space $\{\mathrm{pt}\}$.
        \item Let $M$ be a topological space and $H$ be a topological group. Then the  trivial bundle over $M$ with fiber $H$
            \begin{align*}
                \mathrm{triv}_M^H \colon M \times H \rightarrow M, \quad (m, x) \mapsto m
            \end{align*}	
        from Example \ref{exampletriv} defines a topological group bundle over $M$.
        \item Take a Hausdorff topological space $M$, a topological group $H$ and a point $m_0 \in M$. The quotient space of $M \times H$ defined by identifying all points $(m_0,x)$ for $x \in H$ yields a topological group bundle over $M$. The fiber over $m_0$ is a trivial group, whereas each fiber over $m \in M \setminus \{m_0\}$ can be identified with $H$.
    \end{enumerate}
\end{examples}

The next proposition introduces further important examples of (semi)group bundles. Here we denote by $\mathrm{Homeo}(\Omega)$ the homeomorphism group of a topological space $\Omega$. 

\begin{proposition}\label{semigrpbundlefibermaps}
     Consider any locally compact bundle $q \colon \Omega \rightarrow M$. Then the subspace
        \begin{align*}
            \mathrm{C}(q) \coloneqq \bigsqcup_{m \in M} \mathrm{C}(\Omega_m, \Omega_m) \subseteq \mathrm{C}_q^q(\Omega, \Omega),
        \end{align*}
        equipped with the fiberwise composition of maps, defines a topological semigroup bundle over $M$ with respect to the restriction of the source map $\mathrm{s}$. If $q$ is proper, then
        \begin{align*}
            \mathrm{Homeo}(q) \coloneqq \bigsqcup_{m \in M} \mathrm{Homeo}(\Omega_m) \subseteq \mathrm{C}(q)
        \end{align*}
        defines a topological group bundle over $M$. 
\end{proposition}
\begin{proof}
    In view of Proposition \ref{netconv} the only difficult part is to show that, in the case of a proper bundle $q$, the inversion map $\mathrm{Homeo}(q) \rightarrow \mathrm{Homeo}(q), \, \vartheta \mapsto \vartheta^{-1}$ is continuous. Since this can be checked locally on $M$, we may restrict to the case of a compact bundle. But then, with the subbase for the compact open topology from Definition \ref{defcompactopen}, we have
        \begin{align*}
            W(U,C,O)^{-1} &= \{\vartheta \in \mathrm{s}^{-1}(U) \mid \Omega_{\mathrm{s}(\vartheta)} \cap C \subseteq \vartheta(\Omega_{\mathrm{s}(\vartheta)} \cap O)\} \\
            &= \{\vartheta \in \mathrm{s}^{-1}(U) \mid  \vartheta(\Omega_{\mathrm{s}(\vartheta)} \cap \Omega \setminus O)  \subseteq \Omega_{\mathrm{s}(\vartheta)} \cap (\Omega \setminus C)\} \\
            &= W(U,\Omega \setminus O,\Omega \setminus C),
        \end{align*}
    which yields the claim.
\end{proof}
Next, recall the concept of (sub)semigroup bundles.

\begin{definition}
    Let $q \colon H \rightarrow M$ be a topological semigroup bundle over a topological space $M$. The restriction $q|_{H'} \colon H' \rightarrow M$ to a subset $H' \subseteq H$ is called a \textbf{subsemigroup bundle} of $q$ if $H'_m \coloneqq H_m \cap H'$ is a non-empty subsemigroup of $H_m$ for each $m \in M$. If $q$ is a topological group bundle, and $H_m'$ is even a subgroup of $H_m$ for each $m \in M$, then $q|_{H'} \colon H' \rightarrow M$ is called a \textbf{subgroup bundle} of $q$.
\end{definition}
Clearly, every sub(semi)group bundle defines a topological (semi)group bundle in its own right.\medskip

We highlight that -- in contrast to the situation of topological semigroups -- the closure of the total space of a subsemigroup bundle does generally \emph{not} define a subsemigroup bundle:

\begin{example}
    Consider the Klein four-group $H \coloneqq \Z/2\Z \times \Z/2\Z$ as a discrete topological group, and then the trivial bundle $q \coloneqq \mathrm{triv}_M^H$ over $M= [-1,1]$ with fiber $H$ (cf. Example \ref{examplesgrpbundles}). Setting
        \begin{align*}
            H_m' \coloneqq \begin{cases} \{m\} \times (\Z/2\Z \times \{[0]_{2\Z}\}) & \textrm{ for } m \in [-1,0),\\
            \{m\} \times (\{[0]_{2\Z}\} \times \{ [0]_{2\Z}\})  & \textrm{ for } m = 0,\\
            \{m\} \times  (\{[0]_{2\Z}\} \times \Z/2\Z) & \textrm{ for } m\in  (0,1]\\
            \end{cases}
        \end{align*}
    defines a subgroup bundle $q|_{H'} \colon H' \rightarrow M$. The closure of $H'$ is 
        \begin{align*}
            ([-1,0] \times  (\Z/2\Z \times \{[0]_{2\Z}\})) \cup ([0,1] \times (\{[0]_{2\Z}\} \times \Z/2\Z)),
        \end{align*}
    and the restriction $q|_{H'} \colon H' \rightarrow M$ is not a subsemigroup bundle (and, in particular, not a subgroup bundle) of $q$. 
\end{example}

This observation motivates the following notions (see also \cite[Construction 4.1.41]{Herm2026} in the framework of topological \emph{poloids}).

\begin{definition}\label{generatedbundle}
    For a topological semigroup bundle $q \colon H \rightarrow M$ and a subbundle $q|_C \colon C \rightarrow M$ (in the sense of Definition \ref{constructiontopbund} (ii)) we write 
        \begin{align*}
            \langle C\rangle \coloneqq \bigcap \{H' \subseteq H \mid H' \subset H \textrm{ closed}, q|_{H'} \textrm{ is a subsemigroup bundle of } q, C \subseteq H'\}.
        \end{align*}
    We say that $C$ \textbf{generates} $q$ if $\langle C \rangle = H$.
\end{definition}

It is clear that the restriction $q|_{\langle C\rangle}$ is then indeed a subsemigroup bundle of $q$ and that $\langle C\rangle$ is a closed subset of $H$. The following basic result will be of use later.

\begin{lemma}\label{abeliangenerator}
    Let $q \colon H \rightarrow M$ be an open topological semigroup bundle. If $q|_{H'} \colon H' \rightarrow M$ is an open abelian subsemigroup bundle and $H'$ generates $q$, then $q$ is also abelian. 
\end{lemma}
\begin{proof}
    Since $q|_{H'} \colon H' \rightarrow M$ is open, we can use Lemma \ref{charopen} to deduce that 
        \begin{align*}
            S_1 \coloneqq \{x \in H \mid xy = yx \textrm{ for every } y \in H'_{q(x)}\}
        \end{align*}
    is closed in $H$. Since $S_1$ also contains $H'$ and defines a subsemigroup bundle, it must be $H$ itself. Now consider 
        \begin{align*}
            S_2 \coloneqq \{x \in H \mid xy = yx \textrm{ for every } y \in H_{q(x)}\}.
        \end{align*}
    Since $q$ is open, the set $S_2$ is also closed in $H$. By the above, it contains $H'$. Since $S_2$ also defines a subsemigroup bundle, it again has to be $H$.
\end{proof}

For subsets of topological group bundles which are closed with respect to inversion we can also characterize the generated subsemigroup bundle in terms of sub\emph{group} bundles.

\begin{lemma}
    Let $q \colon H \rightarrow M$ be a topological group bundle and $q|_C \colon C \rightarrow M$ a subbundle with $C^{-1} = C$. Then 
         \begin{align*}
            \langle C\rangle = \bigcap \{H' \subseteq H \mid H' \subset H \textrm{ closed}, q'|_{H'} \textrm{ is a subgroup bundle of } q, C \subseteq H'\}.
        \end{align*}
\end{lemma}
\begin{proof}
    It suffices to observe that the restriction $q'|_{H'} \colon H' \rightarrow M$ to
        \begin{align*}
            H' \coloneqq \{x \in \langle C \rangle \mid x^{-1} \in \langle C \rangle\}
        \end{align*}
    is a subsemigroup bundle of $q$, and that $H'$ is closed in $H$ with $C \subseteq H'$.
\end{proof}

In this article, quotients of topological group bundles are only considered in the following special case:

\begin{proposition}\label{quotient}
    Let $q \colon H \rightarrow M$ be a proper Hausdorff abelian group bundle over a locally compact space $M$. Let further $q|_{H'} \colon H' \rightarrow M$ be a subgroup bundle such that $H'$ is closed in $H$. If we equip
        \begin{align*}
            H/H' \coloneqq \bigsqcup_{m \in M} H_m/H_m'
        \end{align*}
    with the quotient topology, then the canonical map $q_{/H'} \colon H/H' \rightarrow M$ again defines a proper Hausdorff abelian group bundle. Moreover, if $q$ is open, then so is  $q_{/H'}$.
\end{proposition}

For compact base spaces the result is an easy consequence of \cite[Lemma 3.4]{EJK2023}, and one can reduce the general case to this situation. However, for the reader's convenience we provide a full proof here. 
\begin{proof}[Proof of Proposition \ref{quotient}]
    First show that $H/H'$ is a Hausdorff space. Since $H$ is locally compact, it suffices by \cite[Proposition I.10.4.9]{Bour1995} to check that saturations of compact subsets with respect to the quotient map $\pi \colon H \rightarrow H/H'$ are again compact (indeed, this even shows that $H/H'$ is a locally compact space). So take a compact subset $C \subseteq H$. Then $q(C)$ is compact, and hence $H' \cap q^{-1}(q(C))$ is also compact since $q$ is proper. But then
        \begin{align*}
            \pi^{-1}(\pi(C)) = \{xy\mid (x,y) \in H \times_M H \textrm{ with } x \in C \textrm{ and } y \in H' \cap q^{-1}(q(C))\}
        \end{align*}
    is compact as the image of the compact subset
        \begin{align*}
           (H \times_M H) \cap (C \times (H' \cap q^{-1}(q(C)))) \subseteq H \times_M H
        \end{align*}
    with respect to the multiplication map of $q$.\medskip
    
    Since $q$ is continuous and $H/H'$ carries the final topology with respect to $\pi$, it is clear that $q_{/H'}$ is continuous. Since $q$ proper, the map $\pi$ is a continuous surjection and the diagram 
    \[
		\xymatrix{
			H  \ar[dd]_q \ar[dr]^{\pi} & \\
             & H/H'\ar[dl]^{q_{/H'}}  \\
			M &\\
		}
	\]	
    commutes, we immediately obtain that $q_{/H'}$ is proper as well. If $q$ is open, then we similarly obtain that also $q_{/H'}$ is open.\medskip
    
    To check that the multiplication map of $q_{/H'}$ is continuous, we apply the implication \enquote{(b) $\Rightarrow$ (a)} from Proposition \ref{lemmaclosedgraph} and then only need to show that it has a closed graph. However, the graph of the multiplication map of $p_{/H'}$ is the image of the graph of the multiplication map of $q$ with respect to the morphism
        \begin{align*}
            \pi \times_M \pi \times_M \pi \colon (H \times_M H) \times_M H &\rightarrow (H/H' \times_M H/H') \times_M H/H'\\
            \quad (x,y,z) &\mapsto (\pi(x),\pi(y),\pi(z)).
        \end{align*}
    So the claim follows from Proposition \ref{properclosed} combined with \enquote{(a) $\Rightarrow$ (b)} from Proposition \ref{lemmaclosedgraph} for the multiplication map of $q$. Continuity of the inverse mapping follows similarly. Finally, continuity of the map $M \rightarrow H/H', \, m \mapsto 1_{H_m/H_{m}'}$ is a direct consequence of the continuity of $M \rightarrow H, \, m \mapsto 1_{H_m}$ and of the map $\pi$.
\end{proof}

In the last part of this subsection we briefly discuss a relative version of the Haar measure (see, e.g., \cite[Section I.2]{Rena1980} in the more general framework of groupoids).

\begin{definition}
    Let $q \colon H \rightarrow M$ be a locally compact group bundle. A \textbf{continuous (left) Haar system} $(\lambda_m)_{m \in M}$ for $q$ consists of a left Haar measure $\lambda_m$ of $H_m$ for each $m \in M$ such that the map
        \begin{align*}
            M \rightarrow \C, \quad m \mapsto \int_{H_m} f \, \mathrm{d}\lambda_m
        \end{align*}
    is continuous for each $f \in \mathrm{C}_{\mathrm{c}}(H)$.
\end{definition}

The subsequent criterion for the existence of a continuous Haar system basically follows from \cite[Lemma 1.3]{Rena1991} (noticing that the proof of the easy implication \enquote{(a) $\Rightarrow$ (b)} works for any locally compact group bundle), see also \cite[Theorem 6.9 and its proof]{Will2019}.

\begin{proposition}\label{existenceconthaar}
    For a locally compact group bundle $q \colon H \rightarrow M$ consider the following assertions.
        \begin{enumerate}[(a)]
            \item $q$ has a continuous Haar system.
            \item $q$ is open.
        \end{enumerate}
    Then \enquote{(a) $\Rightarrow$ (b)}. If $M$ is compact, then (a) and (b) are equivalent.
\end{proposition}

\subsection{Pontryagin Duality}
A crucial tool for the proof of our main results is a \enquote{relative version} of the classical Pontryagin duality. Recall from Example \ref{basicdual} that for any locally compact abelian group $H$ its Pontryagin dual is given by
    \begin{align*}
        H^* = \{\varrho \colon H \rightarrow \T\mid \varrho \textrm{ is a continuous group homomorphism}\}.
    \end{align*}
The following is the generalization of this concept to open locally compact abelian group bundles.

\begin{definition}\label{defdualbundle}
    Let $q \colon H \rightarrow M$ be an open locally compact abelian group bundle. Equip $H^* \coloneqq \bigsqcup_{m \in M} H_m^*$ with the compact-open topology inherited from $\mathrm{C}_q(H)$. The canonical map $q^* \colon H^* \rightarrow M$, sending $\varrho \in H_m^*$ to $m \in M$, is called the \textbf{dual group bundle} of $q$.
\end{definition}

For a locally compact abelian group bundle $q \colon H \rightarrow M$ with a continuous Haar system $(\lambda_m)_{m \in M}$, the dual group bundle is examined in \cite{MRW1996}, \cite{Goeh2008}, and \cite{Rena2021} from a C*-algebra perspective. The dual is identified with the spectrum of the group bundle C*-algebra\footnote{This is usually done under the assumption of second countability, but the relevant parts of  \cite{MRW1996} for this identification still work without this assumption, see \cite[Footnote 3]{MRW1996}.} $\mathrm{C}^*(H)$. In particular, in these works $H^*$ is equipped with the initial topology with respect to the maps
    \begin{align*}
           \xi_f \colon H^* \rightarrow \C, \quad \varrho \mapsto \int_{H_{q^*(\varrho)}} \varrho(x) f(x) \, \mathrm{d}\lambda_{q^*(\varrho)}(x)
    \end{align*}
for $f \in \mathrm{C}_{\mathrm{c}}(H)$. The following result shows the compatibility with our definition and is basically a version of \cite[Proposition 3.3]{MRW1996} for bundles without the assumption of second countability (see also \cite[Lemma 3.14]{EJK2023} for the case of compact group bundles). It also uses an idea from the proof of \cite[Theorem 6.9]{Will2019}. Notice here that by Proposition \ref{existenceconthaar} every locally compact abelian group bundle with continuous Haar system is open.

\begin{proposition}\label{agreetop}
    Let $q \colon H \rightarrow M$ be a locally compact abelian group bundle with continuous Haar system $(\lambda_m)_{m \in M}$. Then the compact-open topology on $H^*$ agrees with the initial topology defined by the maps $\xi_f$ for $f \in \mathrm{C}_{\mathrm{c}}(H)$.
\end{proposition}

We first prove an auxiliary result (cf. \cite[Proof of Proposition 2.9]{EJK2023}). Here, we denote by $\|\cdot\|_{\infty}$
the supremum norm.
\begin{lemma}\label{convergencehaar}
    Consider a locally compact group bundle $q \colon \Omega \rightarrow M$, some function $f \in \mathrm{C}_{\mathrm{c}}(\Omega)$ and $m \in M$. Let further $(f_\alpha)_\alpha$ be a net in $\mathrm{C}_{q}(\Omega)$ converging to $f|_{\Omega_m}$. If $\bigcup_{\alpha} \supp(f_{\alpha})$ is precompact in $\Omega$, then $\lim_{\alpha} \|f_{\alpha} - f|_{\Omega_{\mathrm{s}(f_{\alpha})}}\|_{\infty} = 0$.
\end{lemma}
\begin{proof}
    For each $\alpha$ we choose $x_{\alpha} \in \Omega_{\mathrm{s}(f_{\alpha})}$ with 
        \begin{align*}
            \|f_{\alpha} - f|_{\Omega_{\mathrm{s}(f_{\alpha})}}\|_{\infty} = |f_{\alpha}(x_{\alpha}) - f(x_{\alpha})|.
        \end{align*}
    It suffices to check that there is no subnet $(f_\beta)_{\beta}$ of $(f_\alpha)_{\alpha}$ such that $(|f_{\beta}(x_{\beta}) - f(x_{\beta})|)_{\beta}$ is bounded from below by some constant $c > 0$. So pick any subnet $(f_{\beta})_{\beta}$ of $(f_{\alpha})_{\alpha}$. If $x_{\beta} \in \Omega \setminus(\supp(f_{\beta}) \cup \supp(f))$ for all large enough $\beta$, clearly $\lim_{\beta} |f_{\beta}(x_{\beta}) - f(x_{\beta})| = 0$. Otherwise, we find a subnet  $(f_{\gamma})_{\gamma}$ of  $(f_{\beta})_{\beta}$ such that $x_{\gamma} \in \supp(f_{\gamma}) \cup \supp(f)$ for each $\gamma$. Using that $\bigcup_{\alpha} \supp(f_{\alpha})$ is precompact, we may pass to a subnet once more to assume that $(x_{\gamma})_{\gamma}$ converges to some $x \in \Omega$. We clearly have $x \in \Omega_m$. But then
        \begin{align*}
            \lim_{\gamma} |f_{\gamma}(x_{\gamma}) - f(x_{\gamma})| = |f(x) - f(x)| = 0.
        \end{align*}
\end{proof}

\begin{proof}[Proof of Proposition \ref{agreetop}]
    Take a net $(\varrho_{\alpha})_{\alpha}$ in $H^*$ and $\varrho \in H^*$. We write $m_{\alpha} \coloneqq q^*(\varrho_{\alpha})$ for each $\alpha$, and $m \coloneqq q^*(\varrho)$.\medskip
    
    Assume first that $(\varrho_{\alpha})_{\alpha}$ converges to $\varrho$ with respect to the compact-open topology and take $f \in \mathrm{C}_{\mathrm{c}}(H)$. Then $(f|_{H_{m_{\alpha}}})_{\alpha}$ converges to $f|_{H_m}$ with respect to the compact-open topology, and hence $(\varrho_{\alpha}f|_{H_{m_{\alpha}}})_{\alpha}$ converges to $\varrho f|_{H_m}$. Let $g \in \mathrm{C}_{\mathrm{c}}(H)$ be a continuous and compactly supported extension of $\varrho f|_{H_m}$. Since $\supp (\varrho_{\alpha}f|_{H_{m_{\alpha}}}) \subseteq \supp(f)$ for each $\alpha$ and $\supp(f)$ is compact, Lemma \ref{convergencehaar} yields $\lim_{\alpha} \|\varrho_{\alpha}f|_{H_{m_{\alpha}}} - g|_{H_{m_{\alpha}}}\|_{\infty} = 0$. Choose $h \in \mathrm{C}_{\mathrm{c}}(H)$ with $h \geq 0$ and $h(x) = 1$ for all $x \in \supp(f) \cup \supp(g)$. Then 
        \begin{align*}
            \biggl|\int_{H_{m_{\alpha}}} \varrho_{\alpha}f|_{H_{m_{\alpha}}} - g|_{H_{m_{\alpha}}} \, \mathrm{d}\lambda_{m_{\alpha}}\biggr| \leq  \left(\int_{H_{m_{\alpha}}} h\, \mathrm{d}\lambda_{m_{\alpha}}  \right)  \cdot  \|\varrho_{\alpha} f|_{H_{m_{\alpha}}} - g|_{H_{m_{\alpha}}}\|_{\infty}
        \end{align*}
    for each $\alpha$. Continuity of the Haar system implies that the right-hand side tends to zero. Again using continuity of the the Haar system, we conclude that  
        \begin{align*}
            \lim_{\alpha} \xi_f(\varrho_{\alpha}) = \lim_{\alpha} \int_{H_{m_{\alpha}}} \varrho_{\alpha}f|_{H_{m_{\alpha}}}  \, \mathrm{d}\lambda_{m_{\alpha}} = \lim_{\alpha}  \int_{H_{m_\alpha}} g \, \mathrm{d}\lambda_{m_{\alpha}} =  \int_{H_m} g \, \mathrm{d}\lambda_m = \xi_f(\varrho),
        \end{align*}
    as desired.\medskip
    
    Assume conversely that  $(\varrho_{\alpha})_{\alpha}$ converges to $\varrho$ with respect to the initial topology defined by the maps  $\xi_f$ for $f \in \mathrm{C}_{\mathrm{c}}(H)$. Suppose first that $(m_{\alpha})_{\alpha}$ does not converge to $m$. We then find a subnet $(m_{\beta})_{\beta}$ of $(m_{\alpha})_{\alpha}$ and an open neighborhood $U \subseteq M$ of $m$ with $m_{\beta} \notin U$ for each $\beta$. Choose $g \in \mathrm{C}_{\mathrm{c}}(H_m)$ with $\int_{H_m} \varrho g\, \mathrm{d}\lambda_m = 1$, and let $f \in \mathrm{C}_{\mathrm{c}}(H)$ be an extension of $g$ with $\supp(f) \subseteq q^{-1}(U)$. Then 
        \begin{align*}
            \lim_{\beta} \xi_f(\varrho_{\beta}) = 0 \neq 1 = \xi_f(\varrho),
        \end{align*}
    a contradiction. Thus, the net $(m_{\alpha})_{\alpha}$ converges to $m$ in $M$.\medskip
    
    Now take a subnet $(\varrho_{\beta})_{\beta}$ of $(\varrho_{\alpha})_{\alpha}$ and $x_{\beta} \in H_{m_{\beta}}$ for each $\beta$ such that $x \coloneqq \lim_{\beta} x_{\beta}$ exists in $\Omega$. We have to show that $\lim_{\beta} \varrho_{\beta}(x_{\beta}) = \varrho(x)$. Let $f \in \mathrm{C}_{\mathrm{c}}(H)$ with $\xi_f(\varrho) = 1$, and find some $g \in \mathrm{C}_{\mathrm{c}}(H)$ with $g(y) = f(x^{-1}y)$ for each $y \in H_m$.  For each $\beta$ define $g_{\beta} \in \mathrm{C}_{\mathrm{c}}(H_{m_\beta})$ by $g_{\beta}(y) \coloneqq f(x_{\beta}^{-1}y)$ for $y \in H_{m_{\beta}}$. Then $\lim_{\beta} g_{\beta} = g|_{H_m}$ with respect to the compact-open topology. Moreover, if $C$ is a compact neighborhood of $x$, then 
        \begin{align*}
            \supp(g_{\beta}) \subseteq x_{\beta}\cdot (\supp(f)\cap H_{m_{\beta}}) \subseteq \{yz\mid (y,z) \in (H \times_M H) \cap (C \times \supp(f))\} 
        \end{align*}
    for all large enough $\beta$. Since the set on the right hand side is compact by continuity of multiplication, we can apply Lemma  \ref{convergencehaar} to obtain
        \begin{align*}
            \lim_{\beta} \|g_{\beta} - g|_{H_{m_{\beta}}}\|_{\infty} = 0.
        \end{align*}
    Let $h \in \mathrm{C}_{\mathrm{c}}(H)$ with $h \geq 0$ and $h(x) = 1$ for all $x \in \supp(f) \cup \supp(g)$. Then
        \begin{align*}
            \biggl|\int_{H_{m_{\beta}}} \varrho_{\beta}(y) (f(x_{\beta}^{-1}y)  - g(y) )\, \mathrm{d}\lambda_{m_{\beta}}(y)\biggr| &\leq \lambda_{m_{\beta}}(x_{\beta}\supp(f) \cup \supp(g)) \cdot \|g_{\beta} - g|_{H_{m_{\beta}}}\|_{\infty}
        \end{align*}
    for each $\beta$. But the right-hand side converges to zero since 
        \begin{align*}
             \lambda_{m_{\beta}}(x_{\beta}\supp(f) \cup \supp(g)) \leq \lambda_{m_{\beta}}(\supp(f)) + \lambda_{m_{\beta}}(\supp(g)) \leq 2 \int_{H_{m_{\beta}}} h\, \mathrm{d}\lambda_{m_{\beta}}
        \end{align*}
    for all $\beta$ and the Haar system is continuous. This shows that
        \begin{align*}
            \lim_{\beta} \int_{H_{m_{\beta}}} \varrho_{\beta}(y) f(x_{\beta}^{-1}y) \, \mathrm{d}\lambda_{m_{\beta}}(y) = \lim_{\beta} \xi_g(\varrho_{\beta}) = \xi_g(\varrho) = \int_{H_m} \varrho(y) f(x^{-1}y)\, \mathrm{d}\lambda_m(y).
        \end{align*}
    But, for each $\beta$, 
        \begin{align*}
            \int_{H_{m_{\beta}}} \varrho_{\beta}(y) f(x_{\beta}^{-1}y) \, \mathrm{d}\lambda_{m_{\beta}}(y) = \varrho_{\beta}(x_{\beta}) \int_{H_{m_{\beta}}} \varrho_{\beta}(y) f(y) \,  \mathrm{d}\lambda_{m_{\beta}}(y)  = \varrho(x_{\beta}) \xi_f(\varrho_{\beta})
        \end{align*}
    and, similarly, 
        \begin{align*}
             \int_{H_{m}} \varrho(y) f(x^{-1}y) \, \mathrm{d}\lambda_{m}(y) = \varrho(x) \xi_f(\varrho) = \varrho(x).
        \end{align*}
    Since $\lim_{\beta} \xi_f(\varrho_{\beta}) = \xi_f(\varrho) = 1$, we must have $\lim_{\beta} \varrho(x_{\beta}) = \varrho(x)$.
\end{proof}

Proposition \ref{agreetop}, combined with the results from \cite{MRW1996} and \cite{Rena2021}, allows us to establish the following crucial statement.

\begin{proposition}
    If $q \colon H \rightarrow M$ is an open locally compact abelian group bundle, then $q^* \colon H^* \rightarrow M$ is also an open locally compact abelian group bundle.
\end{proposition}
\begin{proof}
    Using Proposition \ref{netconv} it is straightforward to check that $q^*$ is a topological group bundle. We also know from Corollary \ref{corhausdorff} that $H^*$ is a Hausdorff space. Since local compactness and openness of $q^*$ can be checked locally, we may assume that the base space $M$ is compact and hence that $q$ has a continuous Haar system $(\lambda_m)_{m \in M}$ by Proposition \ref{existenceconthaar}. But then $H^*$ is locally compact since it is homeomorphic to the spectrum of the commutative C*-algebra $\mathrm{C}^*(H)$ by Proposition \ref{agreetop}. Moreover, by \cite[Corollary 5]{Rena2021} the dual Haar system of $(\lambda_m)_{m \in M}$ is a continuous Haar system for $q^*$. In particular, $q^*$ has a continuous Haar system. Therefore, $q^*$ is open by Proposition \ref{existenceconthaar}.
\end{proof}

The following result for morphisms of locally compact abelian group bundles is straightforward to verify. Recall here the notation for induced fiber maps from Definition \ref{deftopbundle}.
\begin{proposition}\label{dualmorph}
    Let $q \colon H \rightarrow M$ and $q' \colon H' \rightarrow M$ be two open locally compact abelian group bundles. If $\Phi\colon q \rightarrow q'$ is a morphism, then 
        \begin{align*}
            \Phi^* \colon (H')^* \rightarrow H^*, \quad \varrho \mapsto \varrho \circ \Phi_{(q')^*(\varrho)}
        \end{align*}
     defines a morphism $\Phi^* \colon (q')^* \rightarrow q^*$.
\end{proposition}

For any given locally compact space $M$ the construction of the dual bundle therefore gives rise to a functor from the category $\mathbf{LCA}(M)$ of open locally compact abelian group bundles to its opposite category.
\begin{definition}\label{dualfunctor}
    Let $M$ be a locally compact space. Define the functor 
        \begin{align*}
            \mathrm{D} \colon \mathbf{LCA}(M) \rightarrow \mathbf{LCA}(M)^{\mathrm{op}}
        \end{align*}
    by setting
        \begin{enumerate}[(i)]
            \item $\mathrm{D}(q) \coloneqq q^*$ for each open locally compact abelian group bundle $q \colon H \rightarrow M$, and
            \item $\mathrm{D}(\Phi) \coloneqq \Phi^*$ for each morphism $\Phi\colon q \rightarrow q'$ of open locally compact abelian group bundles.
        \end{enumerate}
\end{definition}

Here, we are only interested in the following two special cases of open locally compact abelian group bundles $q \colon H \rightarrow M$:
\begin{enumerate}[(i)]
    \item $q$ is an open and proper Hausdorff bundle, or
    \item $q$ is an \'{e}tale Hausdorff bundle. 
\end{enumerate}

We establish that these two types of bundles are - exactly as compact and discrete groups - dual to each other.

\begin{proposition}\label{dualpropetal}
    Let $M$ be a locally compact space. 
        \begin{enumerate}[(i)]
            \item If $q \colon H \rightarrow M$ is an open and proper Hausdorff abelian group bundle, then the dual bundle $q^* \colon H^* \rightarrow M$ is an \'{e}tale Hausdorff abelian group bundle.
            \item If $q \colon H \rightarrow M$ is an \'{e}tale Hausdorff abelian group bundle, then the dual bundle $q^* \colon H^* \rightarrow M$ is an open and proper Hausdorff abelian group bundle.
        \end{enumerate}
\end{proposition}

\begin{proof}
    For part (i) assume that $q \colon H \rightarrow M$ is an open and proper Hausdorff abelian group bundle. We can check locally on $M$ that $q$ is \'{e}tale. We may therefore assume that $M$ is compact and hence that $q$ is an open compact abelian group bundle. But then the result is a consequence of \cite[Theorem 3.18 and Lemma 3.19]{EJK2023} (for the special case of an \emph{abelian} group bundle).\medskip
    
    For part (ii) let  $q \colon H \rightarrow M$ be an \'{e}tale Hausdorff abelian group bundle. We may once again assume that $M$ is compact and then have to check that $H^*$ is compact as well. To do so, let $(\varrho_{\alpha})_{\alpha}$ be a net in $H^*$ and show that it has a convergent subnet. By compactness of $M$, we find a subnet $(\varrho_{\beta})_{\beta}$ of $(\varrho_{\alpha})_{\alpha}$ such that the limit $m \coloneqq \lim_{\beta} q^*(\varrho_{\beta})$ exists in $M$. Applying Tychonoff's theorem to a power of $\T$ we may assume, by passing to a subnet once again, that $d_{U,\tau} \coloneqq \lim_{\beta} \varrho_{\beta}(\tau(q^*(\varrho_{\beta})))$ exists in $\T$ for each open neighborhood $U$ of $m$ and each local section $\tau \in \Gamma_q(U)$.\medskip
    
    Now if $x \in H_m$, then (since $q$ is \'{e}tale, see Lemma \ref{etaleneighborhoodbase}) there is a local continuous section $\tau \in \Gamma_q(U)$ on an open neighborhood $U$ of $m$ with $\tau(m) = x$. Moreover, any two such sections $\tau_1 \in \Gamma_q(U_1)$ and $\tau_2 \in \Gamma_q(U_2)$ have to agree on an open neighborhood of $m$ which is contained in $U_1 \cap U_2$, hence the nets $(\varrho_{\beta}(\tau_1(q^*(\varrho_{\beta}))))_{\beta}$ and  $(\varrho_{\beta}(\tau_2(q^*(\varrho_{\beta}))))_{\beta}$ have to eventually coincide, and this implies $d_{U_1,\tau_1} = d_{U_2,\tau_2}$. As a consequence there is a unique map $\varrho \colon H_m \rightarrow \T$ with the property that for every $x \in H_m$ we have
        \begin{align*}
            \varrho(x) = \lim_{\beta} \varrho_{\beta}(\tau(q^*(\varrho_{\beta})))
        \end{align*}
    for every local continuous section $\tau \in \Gamma_q(U)$ on an open neighborhood $U$ of $m$ with $\tau(m) = x$. One can readily check that $\varrho$ is a group homomorphism. Since $H_m$ is discrete, this implies $\varrho \in H_m^*$. By Lemma \ref{compopenetale} the net $(\varrho_{\beta})_{\beta}$ converges to $\varrho$ in $H^*$.
\end{proof}

It is a well known fact that a subset $C \subseteq H$ of a locally compact abelian group $H$ generates a dense subgroup if and only if the only character $\varrho \in H^*$ satisfying $\varrho(x) =1$ for all $x \in C$ is the trivial character $\varrho = \1_H$.
We show the following version for generating subsets of open and proper Hausdorff abelian group bundles over a locally compact space.

\begin{corollary}\label{chargenerating}
    Let $q \colon H \rightarrow M$ be an open and proper Hausdorff abelian group bundle over a locally compact space $M$. For a subbundle $q|_C \colon C \rightarrow M$ of $q$ the following assertions are equivalent.
        \begin{enumerate}[(a)]
            \item $C$ is generating for $q$.
            \item Whenever $\tau \in \Gamma_{q^*}(U)$ is a local section of $q^*$ on some open subset $U \subseteq M$ with 
                \begin{align*}
                    (\tau(m))(x) = 1 \quad \textrm{ for all } \, m \in U \, \textrm{ and } \, x \in C_m, 
                \end{align*}
            then $\tau(m) = \1_{H_m}$ for each $m \in M$.
        \end{enumerate}
\end{corollary}
\begin{proof}
    The implication \enquote{(a) $\Rightarrow$ (b)} follows from the fact that for any given local section $\tau \in \Gamma_{q^*}(U)$ as in (b) the set  
        \begin{align*}
            H' \coloneqq \{x \in H \mid \tau(q(x))(x) = 1 \textrm{ if } q(x) \in U\}
        \end{align*}
    is closed in $H$, contains $C$ and defines a subsemigroup bundle of $q$.\medskip
    
    Now assume that (a) does not hold and let $H' \coloneqq \langle C \rangle \subsetneq H$. Then the quotient bundle $q_{/H'} \colon H/H' \rightarrow M$ from Lemma \ref{quotient} has a non-trivial fiber. Thus, since characters separate points, we find some $m \in M$ and some non-trivial character $\varrho \in (H_m/H_m')^*$. The bundle $(q_{/H'})^* \colon (H/H')^* \rightarrow M$ is \'{e}tale by Proposition \ref{dualpropetal}, hence there is a local section $\tau \in \Gamma_{(q_{/H'})^*}(U)$ defined on an open neighborhood $U$ of $m$ with $\tau(m) = \varrho$. Write $\pi \colon H \rightarrow H/H'$ for the quotient map, and set $\tau(m') \coloneqq \tau(m') \circ \pi|_{H_{m'}}$ for each $m' \in U$. One can readily check that the map 
        \begin{align*}
            \tau \colon U \rightarrow H^*, \quad m' \mapsto \tau(m')
        \end{align*}
    is a local section for $q^*$. Moreover, $\tau(q(x))(x) = 1$ for each $x \in H'$ (and in particular each $x \in C$) with $q(x) \in U$. Since $\tau(m) = \varrho \circ \pi|_{H_m} \neq \1_{H_m}$, condition (b) is not fulfilled.
\end{proof}

We now establish Pontryagin duality for the above two classes of open locally compact abelian group bundles.

\begin{theorem}[Pontryagin Duality]\label{pontryagin}
    Let $q \colon H \rightarrow M$ be an open and proper Hausdorff abelian group bundle or an \'{e}tale Hausdorff abelian group bundle over a locally compact space $M$. Then the canonical map
        \begin{align*}
            \iota^q \colon H \rightarrow H^{**}, \quad x \mapsto \iota^q_x
        \end{align*}
    with $\iota^q_x(\varrho) \coloneqq \varrho(x)$ for $(x,\varrho) \in H \times_{M} H^*$ defines an isomorphism of topological group bundles.
\end{theorem}

Under the assumption of second countability, Theorem \ref{pontryagin} holds for general open locally compact group bundles by \cite{Goeh2008} (notice that these have a continuous Haar system by \cite[Lemma 1.3]{Rena1991} or  \cite[Theorem 6.9]{Will2019}). In our situation the proof is an easy consequence of the basic results from Subsection \ref{subsectopbun}.

\begin{proof}[Proof of Theorem \ref{pontryagin}]
    It is clear from Proposition \ref{netconv} that $\iota^q$ is continuous and then clearly a morphism of topological group bundles. It follows from classical Pontryagin duality (see \cite[Section 4.3]{Foll2016}, or \cite[Section 14.3]{EFHN2015} for the simpler case of compact or discrete groups) that $\iota^q$ is bijective. Thus, Corollary \ref{automaticcontprop} (in the case of an open and proper bundle) or Corollary \ref{lemmaiso} (in the case of an \'{e}tale bundle) yields the claim.
\end{proof}

For a locally compact space $M$ write $\mathbf{\acute{E}H}(M)$ for the category of \'{e}tale Hausdorff abelian group bundles over $M$, and $\mathbf{PH}(M)$ for the category of open and proper Hausdorff abelian group bundles over $M$. Then the duality functor from Definition \ref{dualfunctor} induces functors
    \begin{align*}
        \mathrm{D}|_{\mathbf{\acute{E}H}(M)} &\colon \mathbf{\acute{E}H}(M) \rightarrow \mathbf{PH}(M)^{\mathrm{op}},\\
        (\mathrm{D}|_{\mathbf{PH}(M)})^{\mathrm{op}} &\colon  \mathbf{PH}(M)^{\mathrm{op}} \rightarrow \mathbf{\acute{E}H}(M).
    \end{align*}

Theorem \ref{pontryagin} shows that these define a categorical equivalence extending the classical duality between discrete and compact abelian groups from Example \ref{exampleequiv}.

\begin{theorem}\label{equivproper}
    Let $M$ be a locally compact space. 
        \begin{enumerate}[(i)]
            \item The isomorphisms $\iota_q \colon q \rightarrow q^{**}$ for \'{e}tale Hausdorff abelian group bundles $q \colon H \rightarrow M$ define a natural isomorphism $\mathrm{Id}_{\mathbf{\acute{E}H}(M)} \rightarrow  (\mathrm{D}|_{\mathbf{PH}(M)})^{\mathrm{op}} \circ \mathrm{D}|_{\mathbf{\acute{E}H}(M)}$.
            \item The isomorphisms $\iota_q \colon q \rightarrow q^{**}$ for open and proper Hausdorff abelian group bundles $q \colon H \rightarrow M$ define a natural isomorphism $\mathrm{D}|_{\mathbf{\acute{E}H}(M)} \circ (\mathrm{D}|_{\mathbf{PH}(M)})^{\mathrm{op}} \rightarrow \mathrm{Id}_{\mathbf{PH}(M)^{\mathrm{op}}}$.
        \end{enumerate}
    In particular, the functors $\mathrm{D}|_{\mathbf{\acute{E}H}(M)}$ and $(\mathrm{D}|_{\mathbf{PH}(M)})^{\mathrm{op}}$ are essentially inverse to each other.
\end{theorem}

\subsection{Sheaves of Abelian Groups}
As an additional preparation for our main results, we now recall the concept of sheaves (of abelian groups) (see, e.g. \cite[Section II.1]{MaMo1992} or \cite[Chapter 3]{Wedh2016}). The remaining part of this subsection then briefly discusses the well-known correspondence  between  \'{e}tale bundles and sheaves (see, e.g., \cite[Sections II.5, II.6 and II.7]{MaMo1992}).

\begin{definition}
   For a given topological space $M$, a  \textbf{presheaf (of abelian groups)} over $M$ is a functor $\mathcal{F} \colon \mathcal{O}_M \rightarrow \mathbf{Ab}$ from the category of open subsets of $M$ to the category of abelian groups (see Examples \ref{examplescat} (ii) and (vi)), i.e., $\mathcal{F}$ assigns 
    \begin{enumerate}[(1)]
        \item to every open subset $U \subseteq M$ an abelian group $\mathcal{F}(U)$, and 
        \item to each pair $(V,U)$ of open subsets of $M$ with $U \subseteq V$ a group homomorphism $\mathcal{F}_U^V \colon F(V) \rightarrow F(U)$,
    \end{enumerate}
    such that 
    \begin{enumerate}[(i)]
        \item $\mathcal{F}_U^W = \mathcal{F}_U^V \circ \mathcal{F}_V^W$ for all open subsets $U,V,W \subseteq M$ with $U \subseteq V \subseteq W$, and
        \item $\mathcal{F}_U^U = \mathrm{id}_{\mathcal{F}(U)}$ for each open subset $U \subseteq M$.
    \end{enumerate}
    It is convenient to then abbreviate $\tau|_U \coloneqq \mathcal{F}_U^V(\tau)$ for $\tau \in \mathcal{F}(V)$ and open subsets $U,V \subseteq M$ with $U \subseteq V$.\medskip
    
    We say that $F$ is a \textbf{sheaf} if, in addition, the following condition is satisfied.
    \begin{enumerate}[(S)]
        \item Given an open cover $U = \bigcup_{i \in I} U_i$  of an open subset $U \subseteq M$ and elements $\tau_{i} \in \mathcal{F}(U_i)$ for each $i \in I$ with $(\tau_i)|_{U_i \cap U_j} = (\tau_j)|_{U_i \cap U_j}$ for all $i,j \in I$, there is a unique $\tau \in \mathcal{F}(U)$ with $\tau|_{U_i} = \tau_i$ for each $i \in I$.
    \end{enumerate}

    A \textbf{morphism} $\eta \colon \mathcal{F}_1 \rightarrow \mathcal{F}_2$ of (pre)sheaves over a topological space $M$ is a natural transformation from the functor $\mathcal{F}_1 \colon \mathcal{O}_M \rightarrow \mathbf{Ab}$ to the functor $\mathcal{F}_2 \colon \mathcal{O}_M \rightarrow \mathbf{Ab}$, i.e., $\eta$ maps each open set $U \subseteq M$ to a group homomorphism $\eta_U \colon \mathcal{F}_1(U) \rightarrow \mathcal{F}_2(U)$ such that 
        \begin{align*}
            \eta_U(\tau|_U) = \eta_V(\tau)|_U \quad \textrm{ for all } \tau \in \mathcal{F}_1(V) \textrm{ and } U,V \subseteq M \textrm{ open with } U \subseteq V. 
        \end{align*}

    For a given topological space $M$ write $\mathbf{PSh}(M)$ and $\mathbf{Sh}(M)$ for the categories of presheaves and sheaves over $M$, respectively.
\end{definition}
Here are simple examples.
\begin{examples}\label{examplessheaves}
    Let $M$ be a topological space and $H$ a topological group. The \textbf{sheaf of continuous maps} $\mathrm{C}(\,\cdot\,,H)$ assigns to each open subset $U \subseteq M$ the group $\mathrm{C}(U,H)$ of continuous maps from $U$ to $H$, and to each pair $(U,V)$ of open subsets of $M$ with $U \subseteq V$ the natural restriction map 
        \begin{align*}
            \mathrm{C}(V,H) \rightarrow \mathrm{C}(U,H), \quad \vartheta \mapsto \vartheta|_U.
        \end{align*}
    In particular, we can consider these sheaves for $H= \C$ with addition or $H= \T$ with multiplication. If $H$ is a discrete group, then $\mathrm{C}(\,\cdot\,,H)$ is the sheaf of locally constant maps.
\end{examples}

We also recall the following notions.
\begin{definition}
    Given a presheaf $\mathcal{F}$ over a topological space $M$, a \textbf{subpresheaf} $\mathcal{F}'$ of $\mathcal{F}$ assigns to each open subset $U \subseteq M$ a subgroup $\mathcal{F}'(U)$ of $\mathcal{F}(U)$ in such a way that $\tau|_U \in \mathcal{F}'(U)$ for each $\tau \in \mathcal{F}'(V)$ and all open subsets $U,V \subseteq M$ with $U \subseteq V$.\medskip

    If $\mathcal{F}$ is even a sheaf, then a subpresheaf $\mathcal{F}'$ of $\mathcal{F}$ is a \textbf{subsheaf} of $\mathcal{F}$ if for every open cover $U = \bigcup_{i \in I} U_i$ of an open subset $U \subseteq M$ and each $\tau \in \mathcal{F}(U)$ the implication
        \begin{align*}
            \tau|_{U_i} \in \mathcal{F}'(U_i) \quad \textrm{ for all } i \in I \quad \quad \Rightarrow \quad \quad \tau \in \mathcal{F}'(U)
        \end{align*}
    holds.
\end{definition}
Clearly, sub(pre)sheaves of a (pre)sheaf naturally define (pre)sheaves in their own right.\medskip

We now turn to the connection between bundles and (pre)sheaves. Every topological abelian group bundle gives rise to a sheaf in the following way.
\begin{definition}
    Let $q \colon H\rightarrow M$ be a topological abelian group bundle over a topological space $M$. The \textbf{section sheaf} $\Gamma_q$ assigns
        \begin{enumerate}[(i)]
            \item to every open subset $U \subseteq M$ the abelian group $\Gamma_q(U)$ of local sections over $U$ (with multiplication of sections defined pointwise), and
            \item to every pair $(U,V)$ of open subsets of $M$ with $U \subseteq V$ the natural restriction mapping $\Gamma_p(V) \rightarrow \Gamma_q(U)$.
        \end{enumerate}
\end{definition}

On the level of morphisms we obtain the following. The proof is straightforward.

\begin{proposition}\label{propsectfunctor}
    Let $\Phi \colon q \rightarrow q'$ be a morphism of topological abelian group bundles over a topological space $M$. Then the group homomorphisms
        \begin{align*}
            (\Phi_*)|_U \colon \Gamma_q(U) \rightarrow \Gamma_{q'}(U), \quad \tau \mapsto \Phi \circ \tau
        \end{align*}
    for open subsets $U \subseteq M$ define a morphism $\Phi_* \colon \Gamma_q \rightarrow \Gamma_{q'}$ of sheaves. 
\end{proposition}

Writing  $\mathbf{Ab}(M)$ for the category of abelian topological group bundles over a topological space $M$, we introduce the following functor.

\begin{definition}\label{defsecfunctor}
    Let $M$ be a topological space. The functor $\Gamma \colon \mathbf{Ab}(M) \rightarrow \mathbf{Sh}(M)$ defined by
        \begin{enumerate}[(i)]
            \item assigning the section sheaf $\Gamma_q$ to each topological abelian group bundle $q$ over $M$, and
            \item mapping every morphism $\Phi \colon q \rightarrow q'$ of topological abelian group bundles over $M$ to $\Phi_* \colon \Gamma_q \rightarrow \Gamma_{q'}$,
        \end{enumerate}
    is called the \textbf{section functor}.
\end{definition}
On the other hand, every presheaf gives rise to an \'{e}tale abelian group bundle via a well-known construction (see, e.g., \cite[Section II.5]{MaMo1992} or \cite[Construction 3.36]{Wedh2016}):

\begin{definition}\label{defstalk}
Let $\mathcal{F}$ be a presheaf over a topological space $M$. For $m \in M$ set 
    \begin{align*}
        \mathcal{F}_m \coloneqq \bigsqcup \{\mathcal{F}(U) \mid U \subseteq M \textrm{ is an open neighborhood of } m\}/\sim
    \end{align*}
where $\tau \sim \tau'$ for $\tau \in \mathcal{F}(U)$ and $\tau' \in \mathcal{F}(U')$ if there is an open neighborhood $V \subseteq U \cap U'$ with $\tau|_{V} = \tau'|_{V}$. The abelian group $\mathcal{F}_m$ (with respect to the natural multiplication) is called the \textbf{stalk} of $\mathcal{F}$ at $m$. Its elements $[\tau]_m$ for $\tau \in \mathcal{F}(U)$, where $U \subseteq M$ is an open neighborhood of $m$, are called the \textbf{germs} of $\mathcal{F}$ at $m$.
\end{definition}
One can then prove the following assertions (see \cite[Section II.5]{MaMo1992} or \cite[Section 3.4]{Wedh2016}).

\begin{proposition}\label{prepbun}
    Let $M$ be a topological space.
    \begin{enumerate}[(i)]
        \item For a presheaf $\mathcal{F}$ over $M$ equip the disjoint union $H_{\mathcal{F}} \coloneqq \bigsqcup_{m \in M} \mathcal{F}_m$ with the topology generated by the sets 
        \begin{align*}
            \{[\tau]_m \in  \mathcal{F}_m \mid m \in U\} \quad \textrm{ for }\,  \tau \in \mathcal{F}(U) \, \textrm{ and } \,  U \subseteq M \textrm{ open}.
        \end{align*}
    Then the natural map $\pi_{\mathcal{F}} \colon H_{\mathcal{F}} \rightarrow M, \, [\tau]_m \mapsto m$ is an \'{e}tale abelian group bundle.
        \item If $\eta \colon \mathcal{F}  \rightarrow \mathcal{F}'$ is a morphism of presheaves over $M$, then the map
            \begin{align*}
                \Phi_{\eta} \colon H_{\mathcal{F}} \rightarrow H_{\mathcal{F}'}, \quad [\tau]_m \mapsto [\eta(\tau)]_m
            \end{align*}
        is a (well-defined) morphism of topological group bundles $\Phi_{\eta}  \colon \pi_{\mathcal{F}} \rightarrow \pi_{\mathcal{F}'}$.
    \end{enumerate}
\end{proposition}

Denote the category of all \'{e}tale abelian group bundles over a topological space $M$ by $\mathbf{\acute{E}}(M)$. Proposition \ref{prepbun} gives rise to the following functor.
\begin{definition} \label{bunfunc}
    Let $M$ be a topological space. The functor $\mathrm{Bun} \colon \mathbf{PSh}(M) \rightarrow \mathbf{\acute{E}}(M)$ defined by assigning 
        \begin{enumerate}[(i)]
            \item to every presheaf $\mathcal{F}$ over $M$ the \'{e}tale abelian group bundle $\pi_{\mathcal{F}} \colon H_{\mathcal{F}} \rightarrow M$, and
            \item to every morphism $\eta \colon \mathcal{F}\rightarrow \mathcal{F}'$ of presheaves over $M$ the morphism $\Phi_{\eta} \colon \pi_{\mathcal{F}} \rightarrow \pi_{\mathcal{F}'}$
        \end{enumerate}
    is called the \textbf{bundle functor}.    
\end{definition}

Restricting the section functor from Definition \ref{defsecfunctor} and the bundle functor of Definition \ref{bunfunc} to the categories of \'{e}tale abelian group bundles and sheaves, respectively, yields the following categorical equivalence (see  \cite[Section II.6]{MaMo1992} or \cite[Section 3.4]{Wedh2016}).

\begin{theorem}\label{bunseciso}
    Let $M$ be a topological space. 
        \begin{enumerate}[(i)]
            \item There is a natural isomorphism $\gamma \colon \mathrm{Id}_{\mathbf{Sh}(M)} \rightarrow \Gamma \circ \mathrm{Bun}|_{\mathbf{Sh}(M)}$ given on a sheaf $\mathcal{F}$ over $M$ by
                \begin{align*}
                    (\gamma_{\mathcal{F}})_U \colon \mathcal{F}(U) \rightarrow \Gamma_{\pi_{\mathcal{F}}}(U), \quad \tau \mapsto [m \mapsto [\tau]_m]  
                \end{align*}
            for every open subset $U \subseteq M$.
            \item There is a natural isomorphism $\delta \colon \mathrm{Bun}  \circ \Gamma|_{\mathbf{\acute{E}}(M)} \rightarrow \mathrm{Id}_{\mathbf{\acute{E}}(M)}$ given on an \'{e}tale abelian group bundle $q$ over $M$ by
                \begin{align*}
                    \delta_p \colon \pi_{\Gamma_q} \rightarrow q, \quad [\tau]_m \rightarrow \tau(m).
                \end{align*}
        \end{enumerate}
    In particular, the functors $\mathrm{Bun}|_{\mathbf{Sh}(M)} \colon \mathbf{Sh}(M) \rightarrow \mathbf{\acute{E}}(M)$ and $\Gamma|_{\mathbf{\acute{E}}(M)} \colon \mathbf{\acute{E}}(M) \rightarrow \mathbf{Sh}(M)$ establish an equivalence between the categories of sheaves over $M$ and \'{e}tale abelian group bundles over $M$.
\end{theorem}

The \'{e}tale bundle associated with a sheaf is generally not a Hausdorff space (see \cite[Section page 85]{MaMo1992}). To characterize the sheaves corresponding to \'{e}tale Hausdorff bundles, we need the following concepts (cf. \cite[Definition 3.24]{Wedh2016}).

\begin{definition}\label{definitionsupport}
    Let $M$ be a topological space and consider a sheaf $\mathcal{F}$ over $M$. Given an element $\tau \in \mathcal{F}(U)$, where $U \subseteq M$ is open, the set
        \begin{align*}
            \supp(\tau) \coloneqq \{m \in U \mid [\tau]_m \neq [1]_m\}
        \end{align*}
    is the \textbf{support} of $\tau$, and we call $\tau$ \textbf{well-supported} if $\supp(\tau)$ is open. The sheaf $\mathcal{F}$ is \textbf{well-supported} if every element $\tau \in \mathcal{F}(U)$, for $U \subseteq M$ open, is well-supported. Denote by $\mathbf{WSh}(M)$ the category of all well-supported sheaves over $M$.
\end{definition}

The following is \cite[Lemma 3.25]{Wedh2016}.

\begin{lemma}
     Let $\mathcal{F}$ be a sheaf over a topological space $M$ and $U \subseteq M$ open. For $\tau \in \mathcal{F}(U)$ the support $\supp(\tau)$ is a closed subset of $U$. In particular, $\tau$ is well-supported precisely when $\supp(\tau)$ is clopen in $U$.
\end{lemma}

As demonstrated by the following example, the notion of support introduced above generalizes the classical support of continuous functions.

\begin{example}\label{examplesupportsheaves}
    For a topological space $M$ and a topological group $H$ consider the sheaf  $\mathrm{C}(\,\cdot\,, H)$ of continuous maps to $H$ from Example \ref{examplessheaves}. A moment's thought reveals that for an open subset $U \subseteq M$ and $\vartheta \in \mathrm{C}(U,H)$ we have
        \begin{align*}
            \supp( \vartheta) = \overline{\{m \in U \mid \vartheta(m) \neq 1_H\}},
        \end{align*}
    where the closure is taken in $U$. In particular, if $H = \C$ with addition, then
        \begin{align*}
            \supp(\vartheta) = \overline{\{m \in U \mid  \vartheta(m) \neq 0\}},
        \end{align*}
    cf. \cite[Example 3.26]{Wedh2016}.
\end{example}

For Hausdorff base spaces we obtain the following (see \cite[Corollary 1.27 as well as Problems 3.12 and 3.13]{Wedh2016}).

\begin{proposition}
    Let $\mathcal{F}$ be a sheaf over a Hausdorff space $M$. Then $\mathcal{F}$ is well-supported if and only if the \'{e}tale abelian group bundle $\pi_{\mathcal{F}}$ is Hausdorff.
\end{proposition}

\begin{corollary}\label{wsuphausdorff}
    For any Hausdorff space $M$ the functors 
        \begin{align*}
            \mathrm{Bun}|_{\mathbf{WSh}(M)} \colon \mathbf{WSh}(M) \rightarrow \mathbf{\acute{E}H}(M) \quad \textrm{and} \quad \Gamma|_{\mathbf{\acute{E}H}(M)} \colon \mathbf{\acute{E}H}(M) \rightarrow \mathbf{WSh}(M)
        \end{align*}
    establish and equivalence between the categories of well-supported sheaves over $M$ and \'{e}tale Hausdorff abelian group bundles over $M$.
\end{corollary}

In combination with Theorem \ref{equivproper} we therefore conclude that for any locally compact space $M$ the following categories are pairwise equivalent:
\begin{enumerate}[(i)]
    \item the category $\mathbf{WSh}(M)$ of well-supported sheaves over $M$,
    \item the category $\mathbf{\acute{E}H}(M)$ of \'{e}tale Hausdorff abelian group bundles over $M$, and
    \item the opposite category $\mathbf{HP}(M)^{\mathrm{op}}$ of open and proper Hausdorff abelian group bundles over $M$.
\end{enumerate}

We will use these correspondences in the following sections to classify compactifications of open group bundles over locally compact spaces.

\section{Classification of Group Bundle Compactifications}\label{sec3}
\subsection{Group Bundle Compactifications}

Throughout this section, let $p \colon G \rightarrow M$ be an arbitrary, but fixed open topological abelian group bundle over a locally compact space $M$. Analogously to the case of a group compactification (see \cite[Example 1.6]{kreidler-hermle}), we now define the notion of group \emph{bundle}  
compactifications. We also refer to \cite[Section II.8]{Jame1989} and \cite[Appendix A]{Anan2026} for compactification concepts for topological bundles, and to \cite[Section 4.1.4]{Herm2026} for compactifications in the framework of groupoids. Recall the definition of generating subsets of a topological group bundle from Definition \ref{generatedbundle}.

\begin{definition}\label{defgrpcomp}
    Call a pair $(q, c)$, consisting of an open and proper Hausdorff group bundle $q \colon H \rightarrow M$ and a morphism $c \colon p \rightarrow q$ of group bundles,  a \textbf{group bundle compactification} of $p$ if the image $c(G)$ generates $q$.\medskip

    A \textbf{morphism} $\Phi \colon (q_1, c_1) \rightarrow (q_2, c_2)$ between group bundle compactifications $(q_1, c_1)$ and $(q_2, c_2)$ of $p$ is a 
morphism of group bundles $\Phi \colon q_1 \rightarrow q_2$ such that the diagram
	\[
		\xymatrix{
			p \ar[rr]^{c_1}  \ar[dr]_{c_2} &  &q_1 \ar[dl]^{\Phi}\\
			 & q_2 & 
		}
	\]
is commutative.\medskip

We denote the category of group compactifications of $p$ by $\mathbf{Comp}(p)$.
\end{definition}

\begin{remarks}
    \begin{enumerate}[(i)]
        \item Note that for a group bundle compactification $(q, c)$ we do \emph{not} require the morphism $c$ to be injective.
        \item We highlight that by Lemma \ref{abeliangenerator} the open and proper Hausdorff group bundle $q$ of a group bundle bundle compactification $(q,c)$ is automatically abelian.
    \end{enumerate}
\end{remarks}

 The following are some simple examples of group bundle compactifications.

\begin{examples}\label{examplescomp}
    \begin{enumerate}[(i)]
        \item Consider the trivial discrete group $\{1\}$ and the corresponding trivial group bundle $\mathrm{triv}_M^{\{1\}}$ over $M$ from Example \ref{examplesgrpbundles} (ii). Then the map
            \begin{align*}
               c \colon G \rightarrow M \times \{1\}, \quad t \mapsto (p(t),1)
            \end{align*}
        defines a group compactification $(\mathrm{triv}_M^{\{1\}},c)$, called the \textbf{trivial group compactification}, of $p$.
        %can form a \textbf{trivial group bundle compactification} $(\mathrm{id}_M, c_M)$, where $c_M \colon p \rightarrow \mathrm{id}_M$ is the trivial group bundle 
%morphism $p$.
        \item Let $(c',H')$ be a compactification of a topological group $G'$ in the sense of  \cite[Example 1.6]{kreidler-hermle}, i.e., $H'$ is a compact group and $c' \colon G' \rightarrow H'$ is a continuous group homomorphism with dense range. If $p = \mathrm{triv}_M^{G'}$ is the trivial bundle with fiber $G'$ over the locally compact space $M$ (see Example \ref{examplesgrpbundles} (ii)), then the trivial group bundle $q \coloneqq \mathrm{triv}_M^{H'}$ together with the morphism of group bundles
            \begin{align*}
                c \colon \mathrm{triv}_M^{G'} \rightarrow \mathrm{triv}_M^{H'}, \quad (m, t) \mapsto (m, c'(t))
            \end{align*}
        defines a group bundle compactification $(q,c)$ of $p$.
        \item Assume that  $M = [-1, 1]$ and the open topological group bundle $p \colon G \rightarrow M$ is given as the subgroup bundle of the trivial bundle $\mathrm{triv}_{[-1,1]}^{\Z/2\Z}$ defined by the fibers
    \begin{align*}
        G_m \coloneqq \begin{cases}\{ m \} \times (\mathbb{Z} / 2 \mathbb{Z}) & \textrm{ for
} m \in [-1,0),\\
\{ m \} \times \{ [0]_{2\Z}\} & \textrm{ for
} m \in [0,1].
    \end{cases}
    \end{align*}
The \enquote{larger} subgroup bundle   $q' \colon H' \rightarrow M$ of $\mathrm{triv}_{[-1,1]}^{\Z/2\Z}$ given by the fibers 
    \begin{align*}
        H_m' \coloneqq \begin{cases}\{ m \} \times (\mathbb{Z} / 2 \mathbb{Z}) & \textrm{ for
} m \in [-1,0],\\
\{ m \} \times \{ [0]_{2\Z}\} & \textrm{ for
} m \in (0,1].
    \end{cases}
    \end{align*}
is compact, but not open. However, identifying all the points of the fiber $H_{\frac{1}{2}}'$ yields an open and compact abelian group bundle $q \colon H \rightarrow M$. The composition $q \coloneqq c \circ \iota$ of the inclusion map $\iota \colon G \hookrightarrow H'$ with the quotient map $\pi \colon H' \rightarrow H$ then yields a group bundle compactification $(q,c)$ of $p$.
    \end{enumerate}
\end{examples}

As the only result of this subsection we prove the following.

\begin{lemma} \label{morgbcsurj}
    If $\Phi \colon (q_1, c_1) \rightarrow (q_2, c_2)$ is a morphism of group bundle compactifications of $p$, then the map $\Phi$ is surjective.
\end{lemma}

\begin{proof}
    We write $q_i \colon H_i \rightarrow M$ for $i = 1,2$. The 
bundle $q_2|_{\Phi(H_1)}$ is a subgroup bundle of $q_2$ with
    \begin{align*}
        c_2(G) = \Phi(c_1(G)) \subseteq \Phi(H_1).
    \end{align*}
Since $\Phi(H_1)$ is also closed in $H_2$ by Proposition \ref{properclosed}, we obtain that $\langle c_2(G) \rangle \subseteq \Phi(H_1)$. As $(q_2, c_2)$ is a group bundle compactification of $p$, it follows that 
$H_2 = \Phi(H_1)$.
\end{proof}

\subsection{The Dual Sheaf and its Well-Supported Subsheaves}

As discussed in the introduction, the category of group compactifications of a topological abelian group is equivalent to the opposite category of all subgroups of its dual group. In our situation, we consider the following replacement for the dual group of a topological abelian group. Recall here that for a topological bundle $q \colon \Omega \rightarrow M$ and an open subset $U \subseteq M$ we denote by $q_U$ the restriction of $q$ to the subspace $q^{-1}(U)$ (see Definition \ref{constructiontopbund} (i)).

\begin{definition}\label{defdualsheaf}
    Let $q \colon H \rightarrow M$ be an open topological abelian group bundle. For an open subset $U \subset \T$ a morphism $\chi \colon q_U \rightarrow \mathrm{triv}_U^{\mathbb{T}}$ to the trivial bundle $\mathrm{triv}_U^{\mathbb{T}}$ with fiber $\T$ is called a \textbf{local character} of $q$ over $U$. We define the \textbf{dual sheaf} $\mathcal{D}_q$ of $q$ by considering 
    \begin{align*}
        \mathcal{D}_q(U) \coloneqq \{ \chi \colon q_U \rightarrow \mathrm{triv}_U^{\mathbb{T}} \mid \chi \text{ is a local character} \} 
    \end{align*}
    with the natural multiplication for each open subset $U \subseteq M$, and taking the restriction morphisms
    \begin{align*}
        \mathcal{D}_q(U) \rightarrow \mathcal{D}_q(V), \quad \chi \mapsto \chi|_{q^{-1}(V)}
    \end{align*}
for all open subsets $U, V \subseteq M$ with $V \subseteq U$. 
\end{definition}

It is easy to check that the dual sheaf is indeed a sheaf (of abelian groups). The following basic result -- concerned with the notion of support (see Definition \ref{definitionsupport}) in the framework of dual sheaves -- is also straightforward to prove (cf. Example \ref{examplesupportsheaves}). Here and in the following, given a trivial bundle $\mathrm{triv}_L^\T$ with fiber $\T$ over any topological space $L$, we write $\mathrm{pr}_\T \colon L \times \T \rightarrow \T$ for the projection onto the second component.

\begin{lemma}\label{lemdescsupp}
    Let $q \colon H \rightarrow M$ be an open topological abelian group bundle. For each local character $\chi \in \mathcal{D}_q(U)$ over an open subset $U \subseteq M$ we have
        \begin{align*}
            \mathrm{supp}(\chi) =  \overline{\{ m \in U \mid (\mathrm{pr}_\T \circ \chi)|_{H_m} \neq \1_{H_m}\}}
        \end{align*}
    where the closure is taken in $U$.
\end{lemma}

We use this description to compute the support in a simple example.

\begin{example}\label{examplenonwellsup}
    Consider the discrete group $\mathbb{Z}$ and assume that $p =  \mathrm{triv}_{[-1, 1]}^{\mathbb{Z}}$ is the corresponding trivial bundle over $M = [-1,1]$. Then we obtain a local character 
            \begin{align*}
                   \chi \colon [-1, 1] \times \mathbb{Z} \rightarrow [-1, 1] \times \mathbb{T}, \quad   (m,k)  \mapsto             \begin{cases}
                (m, 1) \quad &\text{ if } m < 0, \\
                (m, \mathrm{e}^{2 \pi \mathrm{i} m k}) \quad &\text{ if } m \geq 0
                    \end{cases}
            \end{align*}
        of $p$ over $U = [-1,1]$. Moreover, $\mathrm{supp}(\chi) = \overline{(0, 1)} = [0, 1] \subseteq [-1, 1]$. Since the support is not open in $[-1,1]$, we obtain that $\chi$ is not well-supported.
\end{example}

In particular, we see that -- even for trivial bundles -- the dual sheaf is generally not a well-supported sheaf. However, for open and proper Hausdorff bundles over a locally compact base space the situation is different. In fact, this will follow from the subsequent alternative description of the dual sheaf via local sections of the dual bundle (see Definition \ref{defdualbundle}) for  open locally compact abelian group bundles. The easy proof is omitted.

\begin{proposition}\label{dualshvseciso}
   Consider an open locally compact abelian group bundle $q \colon H \rightarrow M$. Then the maps
        \begin{align*}
            &\mathcal{D}_q(U) \rightarrow \Gamma_{q^*}(U), \quad \chi \mapsto
        \tau_\chi\\
          &\Gamma_{q^*}(U) \rightarrow \mathcal{D}_q(U), \quad \tau \mapsto \chi_{\tau}
        \end{align*}
    for each open subset $U \subseteq M$ given by 
        \begin{align*}
            (\tau_\chi(m))(x) &\coloneqq \mathrm{pr}_{\T}(\chi(x)) \quad \textrm{ for }\, x \in H_m, \, m \in M \,\textrm{ and }\, \chi \in \mathcal{D}_q(U) \textrm{ as well as}\\
            \chi_\tau(x) &\coloneqq (q(x),(\tau(q(x)))(x)) \quad \textrm{ for } \, x \in H \, \textrm{ and } \, \tau \in  \Gamma_{q^*}(U) 
        \end{align*}
    define mutually inverse sheaf isomorphisms.
\end{proposition}

\begin{corollary}\label{properwellsupp}
    Let $q \colon H \rightarrow M$ be an open proper Hausdorff abelian group bundle over $M$. Then the dual sheaf $\mathcal{D}_q$ is well-supported. Moreover, for each local character $\chi \in \mathcal{D}_q(U)$ over an open subset $U \subseteq M$ we have
        \begin{align*}
            \mathrm{supp}(\chi) = \{m \in U \mid (\mathrm{pr}_{\T} \circ \chi)|_{H_m} \neq \1_{H_m} \},
        \end{align*}
    i.e., the set $\{ m \in U \mid (\mathrm{pr}_\T \circ \chi)|_{H_m} \neq \1_{H_m} \}$ is already closed in $U$.
\end{corollary}
\begin{proof}
    Given an étale topological bundle $r \colon \Omega \rightarrow M$ over $M$, for any two local sections $\tau_1,\tau_2 \in \Gamma_p(U)$ on some open subset $U \subseteq M$ the set 
        \begin{align*}
            \{m \in U \mid \tau_1(m) = \tau_2(m)\}
        \end{align*}
    is clopen. Since $q^*$ is étale by Proposition \ref{dualpropetal}, we can apply this observation to obtain that $\{m \in U \mid \tau(m) \neq \1_{H_m}\}$ is clopen in $U$ for each local section $\tau \in \Gamma_{q^*}(U)$ defined on some open subset $U \subseteq M$. Via Proposition \ref{dualshvseciso} this yields the claim.
\end{proof}

In the classical case, i.e., $M = \{\mathrm{pt}\}$ is a singleton and $G$ is a topological abelian group, each compactification $(H,c)$ of $G$ (see Example \ref{examplescomp} (ii)) gives rise to an injective group homomorphism 
    \begin{align*}
        c^* \colon H^* \rightarrow G^*, \quad \chi \mapsto \chi \circ c
    \end{align*}
and, in particular, a subgroup  $c^* H^* = \{\chi \circ c \mid \chi \in H^* \}$ of the dual $G^*$. We now adjust this construction for group bundle compactifications.\medskip

Recall first that a morphism $\eta \colon \mathcal{F}_1 \rightarrow \mathcal{F}_2$ of sheaves over a topological space is a \textbf{monomorphism} if $\eta_U \colon \mathcal{F}_1(U) \rightarrow \mathcal{F}_2(U)$ is injective for each open subset $U \subseteq M$ (this is equivalent to the usual category theoretical definition, see \cite[Problem 3.7]{Wedh2016}). 

\begin{proposition}\label{pullback}
    Let $(q, c)$ be a group bundle compactification of $p$. Then the maps
        \begin{align*}
            c_U^* \colon \mathcal{D}_q(U) \rightarrow \mathcal{D}_p(U), \quad \chi \mapsto \chi \circ c|_{p^{-1}(U)},
        \end{align*}
    for $U \subseteq M$ open, define a monomorphism of sheaves $c^* \colon \mathcal{D}_q \rightarrow \mathcal{D}_p$.
\end{proposition}

\begin{proof}
    We write $q \colon H \rightarrow M$. It is straightforward to check that the maps $c_U^*$ for open subsets $U \subseteq M$ define a morphism of sheaves $c^* \colon \mathcal{D}_q \rightarrow \mathcal{D}_p$. To show that this is a monomorphism,  take an open subset $U \subseteq M$ and prove that the group homomorphism $c_U^*$ has trivial kernel. However, if $\chi \in \mathcal{D}_q(U)$ with $\mathrm{pr}_{\T}(\chi(c(t))) = 1$ for each $t \in p^{-1}(U)$, then $\mathrm{pr}_{\T}(\chi(x)) = 1$ for every $x \in q^{-1}(U)$ by Proposition \ref{dualshvseciso} and Corollary \ref{chargenerating}. This yields the claim.
\end{proof}

It is straightforward to verify that, given any monomorphism $\eta \colon \mathcal{F}_1 \rightarrow \mathcal{F}_2$ of sheaves over a topological space $M$, the images $(\eta \mathcal{F}_1)(U) \coloneqq \eta(\mathcal{F}_1(U))$ for open subsets $U \subseteq M$ define a subsheaf $\eta(\mathcal{F}_1)$ of $\mathcal{F}_2$, and that the corestriction $\eta \colon \mathcal{F}_1 \rightarrow \eta(\mathcal{F}_1)$ is a sheaf isomorphism. In view of Corollary \ref{properwellsupp}, we therefore obtain the following consequence of Proposition \ref{pullback}.

\begin{corollary}
    Let $(q, c)$ be a group bundle compactification of $p$. Then the sets
        \begin{align*}
            (c^*\mathcal{D}_q)(U) = \{\chi \circ c|_{p^{-1}(U)} \mid \chi \in \mathcal{D}_q(U)\}  \quad \textrm{ for open subsets } U \subseteq M 
        \end{align*}
    define a well-supported subsheaf  $c^*\mathcal{D}_q$ of $\mathcal{D}_p$.
\end{corollary}

Each compactification therefore gives rise to a well-supported subsheaf of the dual sheaf. We collect these subsheaves in a category:

\begin{definition}
    We order the set of all well-supported subsheaves of the dual sheaf $\mathcal{D}_p$ by setting $\mathcal{F}_1 \leq \mathcal{F}_2$ for such subsheaves $\mathcal{F}_1$ and $\mathcal{F}_2$ precisely when $\mathcal{F}_1$ is a subsheaf of $\mathcal{F}_2$. The induced category $\mathbf{WSub}(\mathcal{D}_p)$ (see Example \ref{examplescat}) is the \textbf{category of well-supported subsheaves of $\mathcal{D}_p$}.   
\end{definition}

With the following basic observation we can define a functor from the category $\mathbf{Comp}(p)$ of group bundle compatifications of $p$ to the opposite category $\mathbf{WSub}(\mathcal{D}_p)^{\mathrm{op}}$ of well-supported subsheaves of the dual sheaf $\mathcal{D}_p$.

\begin{lemma}
     If $\Phi \colon (q_1, c_1) \rightarrow (q_2, c_2)$ is a morphism of group bundle compactifications of $p$, then $c_2^*\mathcal{D}_{q_2}$ is a subsheaf of $c_1^*\mathcal{D}_{q_1}$.
\end{lemma}

\begin{definition}\label{defdiscertespectrum}
    Define a functor\footnote{The notation is borrowed from \cite{kreidler-hermle}.} $\mathrm{DDual} \colon \mathbf{Comp}(p) \rightarrow \mathbf{WSub}(\mathcal{D}_p)^{\mathrm{op}}$ by setting
        \begin{enumerate}[(i)]
            \item $\mathrm{DDual}(q, c) \coloneqq c^*\mathcal{D}_q$ for every group bundle compactification $(q, c)$ of $p$, and
            \item $\mathrm{DDual}(\Phi) \coloneqq (c_2^*\mathcal{D}_{q_2}, c_1^*\mathcal{D}_{q_1})$ for each morphism $\Phi \colon (q_1, c_1) \rightarrow (q_2, c_2)$ of group bundle compactifications of $p$.
        \end{enumerate}
\end{definition}

\subsection{Classifying Group Bundle Compactifications}\label{subsectclasscomp}

We now construct an essential inverse to the functor $\mathrm{DDual}$ and thereby establish that the categories $\mathbf{Comp}(p)$ and $\mathbf{WSub}(\mathcal{D}_p)^{\mathrm{op}}$ are equivalent. In the classical situation, i.e., $M = \{\mathrm{pt}\}$ is a singleton space and $G$ is a topological abelian group, one can construct a compactification $(H,c)$ of $G$ out of a subgroup $\sigma \subseteq G^*$ as follows (see \cite[Lemma 4.21]{kreidler-hermle}): Consider $\sigma$ as a discrete group and then form the Pontryagin dual $H \coloneqq \sigma^*$. Then $H$ is a compact group, and one can check that the map 
    \begin{align*}
        c \colon G \rightarrow H, \quad t \mapsto [\chi \mapsto \chi(t)]
    \end{align*}
is a continuous group homomorphism with dense range.\medskip

In our situation we start from a well-supported subsheaf $\mathcal{F} \subseteq \mathcal{D}_p$. Via the equivalence between well-supported sheaves and Hausdorff étale bundles (see Theorem \ref{wsuphausdorff}) we then obtain a Hausdorff étale abelian group bundle $\mathrm{Bun}(\mathcal{F}) = \pi_{\mathcal{F}}$ over $M$. By Proposition \ref{dualpropetal} its dual bundle $\pi_{\mathcal{F}}^*$ is an open and proper Hausdorff abelian group bundle over $M$. The subsequent result turns this into a group bundle compactification of $p$.

\begin{proposition}\label{constructioncomp}
    Let $\mathcal{F}$ be a well-supported subsheaf of the dual sheaf $\mathcal{D}_p$. By setting $c_{\mathcal{F}}(t)([\chi]_{p(t)}) \coloneqq \mathrm{pr}_{\T}(\chi(t))$ for every $t \in G$ and each $\chi \in \mathcal{F}(U)$, where $U \subseteq M$ is an open neighborhood of $p(t)$, we obtain a morphism
    \begin{align*}
        c_{\mathcal{F}} \colon p \rightarrow \pi_{\mathcal{F}}^*, \quad t \mapsto c_{\mathcal{F}}(t)
    \end{align*}
    of topological group bundles.  Moreover, the pair $(\pi_{\mathcal{F}}^*,c_{\mathcal{F}})$ is a group bundle compactification of $p$.
\end{proposition}

\begin{proof}
    Firstly, we show that $c_{\mathcal{F}}$ is well-defined. Let $t \in G$, $U_1,U_2 \subseteq M$ be open neighborhoods of $p(t)$ in $M$ and $\chi_i \in \mathcal{F}(U_i)$ for $i \in \{1,2\}$. Assume further that $[\chi_1]_{p(t)} = [\chi_2]_{p(t)}$. We then find an open neighborhood  $V \subseteq U_1 \cap U_2$ of $p(t)$ with $(\chi_1)|_{p^{-1}(V)} = (\chi_2)|_{p^{-1}(V)}$, and in particular $\chi_1(t) = \chi_2(t)$. \medskip

    We now prove that $c_{\mathcal{F}} \colon G \rightarrow H_{\mathcal{F}}^*$ is continuous (where $H_{\mathcal{F}}$ is the total space of the bundle $\pi_{\mathcal{F}}$, see Proposition \ref{prepbun}). It is then easy to check that $c_{\mathcal{F}}$ is a morphism of topological group bundles.  So let $(t_\alpha)_{\alpha}$ be a net in $G$ converging to some $t \in G$. Since $p$ is continuous, we have
        \begin{align*}
            \lim_{\alpha} \pi_{\mathcal{F}}^* (c_{\mathcal{F}}(t_\alpha)) = \lim_{\alpha} p(t_\alpha) = p(t) = \pi_{\mathcal{F}}^*(c(t)).
        \end{align*}
    Moreover, if $\chi \in \mathcal{F}(U)$ for some open neighborhood $U \subseteq M$ of $p(t)$, then
        \begin{align*}
            \lim_{\alpha} c_{\mathcal{F}}(t_\alpha)([\chi]_{\pi_{\mathcal{F}}(t_\alpha)})) = \mathrm{pr}_{\T}(\chi(t_{\alpha})) = \mathrm{pr}_\T(\chi(t)) = c_{\mathcal{F}}(t)(\chi_{\pi_{\mathcal{F}}(t)}).
        \end{align*}
    Continuity of $c_{\mathcal{F}}$ therefore follows from Lemma \ref{compopenetale} and Theorem \ref{bunseciso} (i).\medskip
    
    Lastly, we prove that the image of $c_{\mathcal{F}}$ generates $\pi_{\mathcal{F}}^*$ by making use of Corollary \ref{chargenerating}. So assume that $\tau \in \Gamma_{\pi_{\mathcal{F}}^{**}}(U)$ is a local section of $\pi_{\mathcal{F}}^{**}$ on some open subset $U \subseteq M$ such that $\tau(p(t))(c_{\mathcal{F}}(t)) = 1$ for all $t \in p^{-1}(U)$. We then have to show that $\tau(m) = \1_{\mathcal{F}_m^*}$ for each $m \in U$ (recall that $\mathcal{F}_m$ denotes the stalk of $\mathcal{F}$ at $m$, see Definition \ref{defstalk}).\medskip

    Pontryagin duality (see Theorem \ref{pontryagin}) yields an isomorphism
        \begin{align*}
            \iota^{\pi_{\mathcal{F}}} \colon \pi_{\mathcal{F}} \rightarrow \pi_{\mathcal{F}}^{**}, \quad
            &[\tau]_m \mapsto \iota^{\pi_{\mathcal{F}}}_{[\tau]_m}
        \end{align*}
    of topological group bundles over $M$. Thus $\tau' \coloneqq (\iota^{\pi_{\mathcal{F}}})^{-1} \circ \tau$ is a local section of $\pi_{\mathcal{F}}$ over $U$. By Theorem \ref{bunseciso} (i) there exists a unique $\tau'' \in \mathcal{F}(U)$ with $\tau'(m) = [\tau'']_m$ for all $m \in M$. But then
        \begin{align*}
            \tau(m)(x) = \iota(\tau'(m))(x) = x(\tau'(m)) = x([\tau'']_m) \quad \text{for every } m \in U \text{ and } x \in \mathcal{F}_m^*.
        \end{align*}
    By our assumption we thus know that 
        \begin{align*}
            \mathrm{pr}_{\T}(\tau''(t)) = c_{\mathcal{F}}(t)([\tau'']_{p(t)}) = \tau(p(t))(c(t)) = 1
        \end{align*}
    for all $t \in p^{-1}(U) \subseteq G$. 
Therefore, $\tau'' \in \mathcal{F}(U)$ is the trivial local character 
    \begin{align*}
        p_U \rightarrow \mathrm{triv}_U^{\mathbb{T}}, \quad t \mapsto (p(t), 1)
    \end{align*}
over $U$. We conclude that $\tau(m)(x) = x([\tau'']_m) = 1$ for every $x \in H$ and consequently $\tau(m) = \1_{\mathcal{F}_m^*}$ for each $m \in U$.
\end{proof}

We complement Proposition \ref{constructioncomp} with the following basic observation to define our desired functor. Here we use the notation from Propositions \ref{dualmorph} and \ref{prepbun} (ii).

\begin{lemma}
    Let $\mathcal{F}_2, \mathcal{F}_1$ be two well-supported subsheaves of $\mathcal{D}_p$ with $\mathcal{F}_2 \leq \mathcal{F}_1$, and let $i_{\mathcal{F}_2}^{\mathcal{F}_1} \colon \mathcal{F}_2 \rightarrow \mathcal{F}_1$ be the canonical inclusion morphism. The induced morphism of topological bundles
        \begin{align*}
            \Psi_{\mathcal{F}_2}^{\mathcal{F}_1}\coloneqq (\Phi_{i_{\mathcal{F}_2}^{\mathcal{F}_1}})^* \colon \pi_{\mathcal{F}_1}^* \rightarrow \pi_{\mathcal{F}_2}^*,
        \end{align*}
   given by $\Psi_{\mathcal{F}_2}^{\mathcal{F}_1}(x)([\chi]_m) = x([\chi]_m)$ for each $m \in M$, every $x \in (\mathcal{F}_1)_m^*$ and each $\chi \in \mathcal{F}_2(U)$ on an open neighborhood $U$ of $m$, defines a morphism 
        \begin{align*}
            \Psi_{\mathcal{F}_2}^{\mathcal{F}_1} \colon (\pi_{\mathcal{F}_1}^*,c_{\mathcal{F}_1}) \rightarrow (\pi_{\mathcal{F}_2}^*,c_{\mathcal{F}_2})
        \end{align*}
    of group compatifications of $p$.
\end{lemma}

\begin{definition}\label{defcompactdualfunctor}
    Define a functor $\mathrm{CDual} \colon \mathbf{WSub}(\mathcal{D}_p)^{\mathrm{op}} \rightarrow \mathbf{Comp}(p)$ by setting
        \begin{enumerate}[(i)]
            \item $\mathrm{CDual}(\mathcal{F}) \coloneqq (\pi_{\mathcal{F}}, c_{\mathcal{F}})$ for every well-supported subsheaf $\mathcal{F}$ of $\mathcal{D}_p$, and
            \item $\mathrm{CDual}((\mathcal{F}_2, \mathcal{F}_1)) \coloneqq  (\Psi_{\mathcal{F}_2}^{\mathcal{F}_1} \colon (\pi_{\mathcal{F}_1}^*,c_{\mathcal{F}_1}) \rightarrow (\pi_{\mathcal{F}_2}^*,c_{\mathcal{F}_2}))$ for two well-supported subsheaves $\mathcal{F}_1, \mathcal{F}_2$ of $\mathcal{D}_p$ with $\mathcal{F}_2 \leq \mathcal{F}_1$.
        \end{enumerate}
\end{definition}

To show that the functors $\mathrm{CDual}$ and $\mathrm{DDual}$ are essentially inverse to each other, we construct two natural isomorphisms. The first one is introduced in the following result.

\begin{proposition}\label{natiso1}
    For a group bundle compactification $(q, c)$ of $p$ given by the bundle $q \colon H \rightarrow M$ we obtain an isomorphism
        \begin{align*}
            \gamma_{(q, c)} \colon (q, c) \rightarrow \mathrm{CDual}(\mathrm{DDual}(q, c))
        \end{align*}
        by setting $\gamma_{(q, c)}(x)([\chi \circ c|_{p^{-1}(U)}]_{q(x)}) = \mathrm{pr}_{\T}(\chi(x))$ for every $x \in H$ and every local character $\chi \in \mathcal{D}_q(U)$ defined on an open neighborhood $U \subseteq M$ of $q(x)$. Furthermore, these isomorphisms define a natural isomorphism
            \begin{align*}
                \gamma \colon \mathrm{Id}_{\mathbf{Comp}(p)} \rightarrow \mathrm{CDual} \circ \mathrm{DDual}.
            \end{align*}
\end{proposition}

\begin{proof}
    Let $(q, c)$ be a group bundle compactification of $p$ and write $q \colon H \rightarrow M$. We first take $x \in H$ and check that $\gamma_{q, c}(x)$ is
well-defined. Recall from Proposition \ref{pullback} that the maps
         \begin{align*}
            c_U^* \colon \mathcal{D}_q(U) \rightarrow \mathcal{D}_p(U), \quad \chi \mapsto \chi \circ c|_{p^{-1}(U)},
        \end{align*}
    for $U \subseteq M$ open, define a monomorphism of sheaves $c^* \colon \mathcal{D}_q \rightarrow \mathcal{D}_p$. But, by \cite[Proposition 3.18]{Wedh2016}, this implies that the induced group homorphism $(\mathcal{D}_q)_{q(x)} \rightarrow (\mathcal{D}_p)_{q(x)}$ between stalks is injective. Thus, if $[\chi_1 \circ c|_{p^{-1}(U_1)}]_{q(x)} = [\chi_2 \circ c|_{p^{-1}(U_2)}]_{q(x)}$ for some local characters $\chi_i \in \mathcal{D}_q(U_i)$ defined on open neighborhoods $U_i \subseteq M$ of $q(x)$ for $i \in \{1,2\}$, then $[\chi_1]_{q(x)} = [\chi_2]_{q(x)}$ and, in particular, $\chi_1(x) = \chi_2(x)$.\medskip

    Via Lemma \ref{compopenetale} and Theorem \ref{bunseciso} (i) we obtain that the map $\gamma_{(q, c)}$ is continuous, and is then straightforward to check that it is a morphism of topological group bundles. To see that it even defines a morphism of group bundle compactifications of $p$, let $t \in G$. Then
        \begin{align*}
            \gamma_{(q, c)}(c(t))([\chi \circ c|_{p^{-1}(U)}]_{p(t)}) &=  \mathrm{pr}_{\T}(\chi(c(t))) =  \mathrm{pr}_{\T}((\chi \circ  c|_{p^{-1}(U)})(t))) \\
            &= c_{c^*\mathcal{D}_q}(t)([\chi \circ  c|_{p^{-1}(U)}]_{p(t)})
        \end{align*}
    for each $t \in G$ and each $\chi \in \mathcal{D}_q(U)$ on an open neighborhood $U \subseteq M$ of $p(t)$. Therefore, $\gamma_{(q, c)} \circ c = c_{c^*\mathcal{D}_p}$ as desired. \medskip
   
    It follows from Lemma \ref{morgbcsurj} that $\gamma_{(q, c)}$ is surjective. We show that it is also injective, and hence an isomorphism of group bundle compactifications by Corollary \ref{automaticcontprop}. Let $x, y \in H$ with $x \neq y$. We may assume that $m := q(x) = q(y)$. Since the elements of $H_m^*$ separate the points of $H_m$ and $q^*$ is étale, we find by Lemma \ref{etaleneighborhoodbase} and Proposition \ref{dualshvseciso} some local character $\chi \in \mathcal{D}_q(U)$ on some open neighborhood $U$ of $m$ such that $\chi(x) \neq \chi(y)$. But this means 
        \begin{align*}
            \gamma_{(q, c)}(x)([\chi \circ c|_{p^{-1}(U)}]_{q(x)}) \neq \gamma_{(q, c)}(y)([\chi \circ c|_{p^{-1}(U)}]_{q(x)}),
        \end{align*}
    and hence $\gamma_{(q, c)}(x) \neq \gamma_{(q, c)}(y)$.\medskip 
    
    Finally, it is an easy exercise to verify that $\gamma$ defines a natural isomorphism.
\end{proof}

We now discuss the second natural isomorphism.

\begin{proposition}\label{natiso2}
    For each well supported subsheaf $\mathcal{F}$ of $\mathcal{D}_p$ we have 
        \begin{align*}
            \mathcal{F} = \mathrm{DDual}(\mathrm{CDual}(\mathcal{F})).
        \end{align*}
    In particular, the identity morphisms on these subsheaves define a natural isomorphism
        \begin{align*}
            \mathrm{Id}_{(\mathbf{WSub}(\mathcal{D}_p))^{\mathrm{op}}} \rightarrow \mathrm{DDual} \circ \mathrm{CDual}.
        \end{align*}
\end{proposition}

\begin{proof}
    Pick a well-supported subsheaf $\mathcal{F}$ of $\mathcal{D}_p$. We have to show that $\mathcal{F} = c_{\mathcal{F}}^*(\mathcal{D}_{\pi_{\mathcal{F}}})$. Let $U \subseteq M$ be an open subset. By Theorem \ref{bunseciso} (i) we have a group isomorphism
         \begin{align*}
            \Psi_1 \colon \mathcal{F}(U) \rightarrow \Gamma_{\pi_{\mathcal{F}}}(U), \quad \chi \mapsto [m \mapsto [\chi]_m].  
        \end{align*}
    Pontryagin duality (see Theorem \ref{pontryagin}) yields a group isomorphism
        \begin{align*}
            \Psi_2 \colon \Gamma_{\pi_{\mathcal{F}}}(U) \rightarrow \Gamma_{\pi_{\mathcal{F}}^{**}}(U), \quad \tau \mapsto [\iota^{\pi_{\mathcal{F}}} \circ \tau].
        \end{align*}
    Moreover, by Proposition \ref{dualshvseciso} the map
         \begin{align*}
            \Psi_3 \colon \Gamma_{\pi_{\mathcal{F}}^{**}}(U) \rightarrow \mathcal{D}_{\pi_{\mathcal{F}}^*}(U), \quad \tau \mapsto \chi_{\tau}
        \end{align*}
    where $\chi_\tau(x) \coloneqq (\pi_{\mathcal{F}}^*(x),[\tau(\pi_{\mathcal{F}}^*(x))](x))$ for  $x \in H_{\mathcal{F}}^*$ with $\pi_{\mathcal{F}^*}(x) \in U$ and $\tau \in  \Gamma_{\pi_{\mathcal{F}}^{**}}(U)$ is a group isomorphism as well. Thus, the composition 
        \begin{align*}
            \Psi \coloneqq \Psi_3 \circ \Psi_2 \circ \Psi_1 \colon  \mathcal{F}(U) \rightarrow  \mathcal{D}_{\pi_{\mathcal{F}}^*}(U)
        \end{align*}
    is also a group isomorphism. It is given by $\Psi(\chi)(x) = (\pi_{\mathcal{F}}^*(x), x([\chi]_{\pi_{\mathcal{F}}^*(x)}))$ for all $\chi \in \mathcal{F}(U)$ and $x \in H_{\mathcal{F}}^*$ with $\pi_{\mathcal{F}^*}(x) \in U$. If we can show that $((c_{\mathcal{F}})_U^* \circ \Psi)(\chi) = \chi$ for every $\chi \in \mathcal{F}(U)$, then the claim will follow. But for $\chi \in \mathcal{F}(U)$ and $t \in p^{-1}(U) \subseteq G$ we obtain that
        \begin{align*}
            (((c_{\mathcal{F}})_U^* \circ \Psi)(\chi))(t) &= \Psi(\chi)(c_{\mathcal{F}}(t)) =  (\pi_{\mathcal{F}}^*(c_{\mathcal{F}}(t)), c_{\mathcal{F}}(t)([\chi]_{\pi_{\mathcal{F}}^*(c_{\mathcal{F}}(t))})) \\
            &= (p(t), \mathrm{pr}_{\T}(\chi(t))) = \chi(t).
        \end{align*}
\end{proof}

With Propositions \ref{natiso1} and \ref{natiso2} we have shown the following equivalence of categories.
\begin{theorem}\label{maintheoremcomp}
    For the open topological abelian group bundle $p \colon G \rightarrow M$ over the locally compact space $M$ the functors 
        \begin{align*}
           \mathrm{DDual} \colon \mathbf{Comp}(p) \rightarrow \mathbf{WSub}(\mathcal{D}_p)^{\mathrm{op}} \quad \textrm{and} \quad \mathrm{CDual} \colon \mathbf{WSub}(\mathcal{D}_p)^{\mathrm{op}} \rightarrow \mathbf{Comp}(p)
        \end{align*}
    establish an equivalence between the category $\mathbf{Comp}(p)$ of group bundle compactifications of $p$ and the opposite category $\mathbf{WSub}(\mathcal{D}_p)^{\mathrm{op}}$ of well-supported subsheaves of the dual sheaf $\mathcal{D}_p$.
\end{theorem}

Let us illustrate the duality result in a particular case.

\begin{example}\label{examplemainthm}
    Assume that $p =\mathrm{triv}_{M}^{\Z}$ is the trivial bundle over the locally compact space $M$ with fiber $\Z$. Since $p$ is an open locally compact group bundle, the dual sheaf $\mathcal{D}_p$ can be identified with the sheaf of local sections of the dual bundle $(\mathrm{triv}_{M}^{\Z})^*$ by Proposition \ref{dualshvseciso}. One can readily check that
        \begin{align*}
            \mathrm{triv}_{M}^{\T} \rightarrow (\mathrm{triv}_{M}^{\Z})^*, \quad (m,z) \mapsto [(m,k) \mapsto (m,z^k)]
        \end{align*}
    is an isomorphism of topological group bundles. Moreover, the local sections of the trivial bundle $\mathrm{triv}_{M}^{\T}$ can be identified with continuous maps $f \colon U \rightarrow \T$ for open subsets $U \subseteq M$. Combining these observations, we conclude that the maps
        \begin{align*}
            \mathrm{C}(U,\T) \rightarrow \mathcal{D}_p(U), \quad f \mapsto [(m,k) \mapsto (m,f(m)^k)]
        \end{align*}
    for open subsets $U \subseteq M$ define an isomorphism between the sheaf of continuous functions $\mathrm{C}(\, \cdot \, , \T)$ from Example \ref{examplessheaves} and the dual sheaf $\mathcal{D}_p$ of the trivial bundle $p =\mathrm{triv}_{M}^{\Z}$. By Theorem \ref{maintheoremcomp} the group compactifications of the trivial bundle $\mathrm{triv}_{M}^{\Z}$ are therefore classified by the well-supported subsheaves of $\mathrm{C}(\, \cdot \, , \T)$.\medskip

    For a concrete example consider the Bohr compactification $(\mathrm{b}\Z,c_{\mathrm{b}\Z})$ of $\Z$, i.e., the group compactification of $\Z$ defined by taking the dual $\mathrm{b}\Z \coloneqq ((\Z^*)_{\mathrm{d}})^*$ of the dual $\Z^*$ equipped with the discrete topology, and the evaluation map 
        \begin{align*}
            c_{\mathrm{b}\Z} \colon \Z \rightarrow \mathrm{b}\Z, \quad k \mapsto [\varrho \mapsto \varrho(k)]
        \end{align*}
    (see, e.g., \cite[Section 4.7]{Foll2016}). We identify $(\Z^*)_{\mathrm{d}}$ with the discrete torus $\T_{\mathrm{d}}$, and then have $c_{\mathrm{b}\Z}(k)\colon \T_{\mathrm{d}} \rightarrow \T, \, z \mapsto z^k$ for each $k \in \Z$. Form the induced group bundle compactification $(\mathrm{triv}_{M}^{\mathrm{b}\Z},c)$ of $p$ given by $c(m,k) = (m,c_{\mathrm{b}\Z}(k))$ for $(m,k) \in M \times \Z$ from Example \ref{examplescomp} (ii). Via Pontryagin duality we can identify the dual of $\mathrm{b}\Z$ with $\mathrm{\T}_{\mathrm{d}}$, and then the dual bundle of $\mathrm{triv}_{M}^{\mathrm{b}\Z}$ is given by $\mathrm{triv}_{M}^{\mathrm{\T}_{\mathrm{d}}}$. This implies that the maps
         \begin{align*}
            \mathrm{C}(U,\T_{\mathrm{d}}) \rightarrow \mathcal{D}_{\mathrm{triv}_{M}^{\mathrm{b}\Z}}(U), \quad f \mapsto [(m,x) \mapsto (m,x(f(m)))]
        \end{align*}
    for open subsets $U \subseteq M$ define an isomorphism between the sheaf  $\mathrm{C}(\, \cdot \, , \T_{\mathrm{d}})$ of locally constant functions to the torus and the dual sheaf $\mathcal{D}_{\mathrm{triv}_{M}^{\mathrm{b}\Z}}$. Thus, the well-supported subsheaf $c^*\mathcal{D}_{\mathrm{triv}_{M}^{\mathrm{b}\Z}}$ of $\mathcal{D}_p$ defined by the compactification $(\mathrm{triv}_{M}^{\mathrm{b}\Z},c)$ of $p$ is given on an open subset $U \subseteq M$ by
        \begin{align*}
            c^*\mathcal{D}_{\mathrm{triv}_{M}^{\mathrm{b}\Z}}(U) = \{[(m,k) \mapsto (m,f(m)^k)] \mid f\colon U \rightarrow \T \textrm{ locally constant}\}.
        \end{align*}
     With respect to the identification discussed in the first part of this example, the compactification $(\mathrm{triv}_{M}^{\mathrm{b}\Z},c)$ therefore corresponds precisely to the well-supported subsheaf of  $\mathrm{C}(\, \cdot \, , \T)$ defined by all locally constant functions to $\T$.
\end{example}

In the classical situation of group compactifications, i.e., if $M = \{\mathrm{pt}\}$ is a singleton space, one can always construct the Bohr compactification $(\mathrm{b}G,c_{\mathrm{b}G})$ as in Example \ref{examplemainthm}, and this compactification is maximal in the following sense\footnote{In category theoretical terms this means that the Bohr compactification $(\mathrm{b}G,c_{\mathrm{b}G})$ in an \emph{initial object} in the category of all compactifications of $G$.}: For any group compactification $(H,c)$ of $G$ there is a unique morphism of group compactifications $\Phi \colon (\mathrm{b}G,c_{\mathrm{b}G}) \rightarrow (H,c)$. In fact, this is a direct consequence of the duality between group compactifications of $G$ and subgroups of the dual group $G^*$ (see \cite[Theorem 4.25]{kreidler-hermle}) as well as the trivial observation that $G^*$ is a subgroup of itself (and hence gives rise to a compactification).\medskip

Does there also exist a Bohr compactification in our framework of group bundles? We have already seen in Example \ref{examplenonwellsup} that, even for a trivial bundle $p$, the dual sheaf $\mathcal{D}_p$ is not well-supported in general, and hence does not yield a group bundle compactification of $p$. The following example now even shows that there is generally no largest well-supported subsheaf of $\mathcal{D}_p$. As a consequence of Theorem \ref{maintheoremcomp} we can then conclude that in general there is no group bundle compactification $(\mathrm{b}p,c_{\mathrm{b}p})$ of $p$ with the maximality property that for any other group bundle compactification $(q,c)$ of $p$ there exists a unique morphism $\Phi \colon (\mathrm{b}p,c_{\mathrm{b}p}) \rightarrow (q,c)$.

\begin{example}
    Assume that $p= \mathrm{triv}_{[-1, 1]}^{\Z}$ is the trivial bundle over $M= [-1,1]$ with fiber $\Z$ as in Example \ref{examplenonwellsup}. By Example \ref{examplemainthm} the dual sheaf $\mathcal{D}_p$ of $p$ is isomorphic to the sheaf of continuous functions  $\mathrm{C}(\, \cdot \, ,\T)$ on $M$. Consider the two continuous functions
        \begin{align*}
            f_1 &\colon [-1,1] \rightarrow \T, \quad m \mapsto \mathrm{e}^{2 \pi \mathrm{i} m},\\
            f_2 &\colon [-1,1] \rightarrow \T, \quad m \mapsto \mathrm{e}^{2 \pi \mathrm{i} |m|},
        \end{align*}
    and the subsheaves $\mathcal{F}_1$ and $\mathcal{F}_2$ of $\mathrm{C}(\, \cdot \, ,\T)$ given by
        \begin{align*}
            \mathcal{F}_1(U) &\coloneqq \{f \in \mathrm{C}(U,\T) \mid \exists \, (U_k)_{k \in \Z} \textrm{ open cover of } U \textrm{ with } f|_{U_k} = f_1^k|_{U_k} \textrm{ for all } k \in \Z\},\\
            \mathcal{F}_2(U) &\coloneqq \{f \in \mathrm{C}(U,\T) \mid \exists \, (U_k)_{k \in \Z} \textrm{ open cover of } U \textrm{ with } f|_{U_k} = f_2^k|_{U_k} \textrm{ for all } k \in \Z\}
        \end{align*}
    for open subsets $U \subseteq M$. If $f \in \mathcal{F}_1(U)$ for $U \subseteq M$ open and $(U_k)_{k \in \Z}$ is the corresponding open cover of $U$, then a moment's thought reveals that $U_k$ for $k \in \Z$ must be pairwise disjoint and hence define a covering of clopen subsets in $U$. One can further check that $\supp(f) = U \setminus U_0$ since for $k \neq 0$ the equality $f_1^k(m) = 1$ can only hold on rational points $m \in [-1,1]$. This shows that $f$, and consequently also the sheaf $\mathcal{F}_1$, is well-supported. Similarly, we obtain that $\mathcal{F}_2$ is well-supported.\medskip

    Now if $\mathcal{F}_1$ and $\mathcal{F}_2$ are both subsheaves of some subsheaf $\mathcal{F}$ of $\mathrm{C}(\,\cdot \, ,\T)$, then we must have $f \coloneqq f_1f_2 \in \mathcal{F}([-1,1])$. But $\supp(f) = [0,1]$, and hence $f$ is not well-supported. We conclude that there cannot be any well-supported subsheaf $\mathrm{C}(\,\cdot \, ,\T)$ having both $\mathcal{F}_1$ and $\mathcal{F}_2$ as a subsheaf. In particular, there is no largest well-supported subsheaf of $\mathrm{C}(\,\cdot \, ,\T)$.
\end{example}

\begin{remark}
    Note that if the locally compact base space $M$ is extremally disconnected (i.e., the closure of an open subset is again open), then Lemma \ref{lemdescsupp} implies that $\mathcal{D}_p$ is well-supported. In particular, in this situation there is an analogue of the Bohr compactification for $p$.
\end{remark}
\section{Classification of Systems with Relative Discrete Spectrum}\label{discspsec}
\subsection{Systems with Relative Discrete Spectrum}

In the classical case, we define topological dynamical systems as group actions on a compact space. For actions of open topological abelian group bundles the following is a natural framework (see also, e.g., \cite[Subsection 2.1.b]{DeRe2000} in the more general situation of groupoids). We again fix an open topological abelian group bundle $p \colon G \rightarrow M$ over a locally compact space $M$ for the entire section.

\begin{definition}
    A \textbf{topological dynamical system  (over $\boldsymbol{p}$)} is a pair $(q, \varphi)$, where $q \colon K \rightarrow M$ is an open proper Hausdorff topological bundle over $M$, and
        \begin{align*}
            \varphi \colon G \times_{M} K \rightarrow K, \quad (t, x) \mapsto \varphi_t(x)
        \end{align*}
    is a morphism of topological bundles over $M$ such that
        \begin{enumerate}[(i)]
            \item $\varphi_{st}(x) = \varphi_s(\varphi_t(x))$ for all $s,t \in G$ and $x \in K$ with $(s,x), (t,x) \in G \times_{p, q} K$, and 
            \item $\varphi_{1_{G_m}}(x) = x$ for every $m \in M$ and $x \in K_m$.
        \end{enumerate}

    A \textbf{morphism} $\Phi \colon (q, \varphi) \rightarrow (r, \psi)$ between such topological dynamical systems is a morphism $\Phi \colon q \rightarrow r$ of topological bundles such that
    $\Phi_{p(t)} \circ \varphi_t = \psi_t \circ \Phi_{p(t)}$ for each $t \in G$.
\end{definition}

Here are some simple, but important examples.

\begin{examples}\label{examplestds}
    \begin{enumerate}[(i)]
        \item Assume $G$ to be a topological group, i.e.,  $M = \{ \mathrm{pt} \}$ is a singleton. Every continuous group action $\varphi \colon G \times K \rightarrow K$ on a compact non-empty space $K$ gives 
rise to a topological dynamical system $(G \rightarrow \{ \mathrm{pt} \}, \varphi)$.
        \item If $q \colon K \rightarrow M$ is a proper, open and Hausdorff topological bundle over a locally compact space $M$, then the trivial action
            \begin{align*}
                \mathrm{id} \colon G \times_{M} K \rightarrow K, \quad (t, x) \mapsto x
            \end{align*}
        yields a topological dynamical system $(q, \mathrm{id})$.
        \item Let $(q, c)$ be a group bundle compactification of $p$ and write $q \colon H \rightarrow M$. Then the pair $(q, \varphi_c)$, where
        \begin{align*}
            \varphi_c \colon G \times_{M} H \rightarrow H, \quad (t, x) \mapsto c(t)x,
        \end{align*}
        is a topological dynamical system. We call this the \textbf{rotation system} induced by the group bundle compactification $(q, c)$.
        %\item As a concrete example of (iii) take $p = $
\end{enumerate}
\end{examples}

In the Halmos--von Neumann theorem for abelian groups we consider topological dynamical systems with \emph{discrete spectrum}. The notion of discrete spectrum however is purely operator theoretic and can be defined in the general framework of group representations on Banach spaces, see, e.g., \cite[Section 1.1]{kreidler-hermle}. Conretely, a representation $T \colon H \rightarrow \mathscr{L}(E)$ of a topological abelian group $H$ on a Banach space $E$ which is bounded (i.e., $\sup_{t \in G} \|T_t\| < \infty$) and strongly continuous (i.e., for each vector $v \in E$ the map $G \rightarrow E, \, x \mapsto T_xv$ is continuous) has discrete spectrum if 
    \begin{align*}  
        E = \overline{\lin} \{v \in E \mid \exists \chi \in H^* \textrm{ such that } T_xv = \chi(x)v \textrm{ for every } x \in H\},
    \end{align*}
see \cite[Proposition 1.3]{kreidler-hermle}. Thus, for such a representation the Banach space $E$ is generated by the eigenvectors of the representation $T$. \medskip

To formulate a generalization of this notion in the situation of open abelian group bundles over locally compact spaces, we first need to replace the notion of a Banach space by that of a Banach bundle (see, e.g., \cite[Definition 1.1]{DuGi1983}, \cite[Section 1 and Theorem 3.2]{Gierz1982} or \cite[Subsection 2.4]{EJK2023}).

\begin{definition}
    An open topological bundle $s \colon E \rightarrow M$ is called a \textbf{(continuous) Banach bundle} over the locally compact space $M$ if
        \begin{enumerate}[(i)]
            \item every fiber $E_m$ for $m \in M$ is a Banach space,
            \item the maps
                \begin{align*}
                   &E \times_{M} E \rightarrow E, \quad (v,w) \mapsto v + w, \\
                    &\mathbb{C} \times E \rightarrow E, \quad (\lambda, v) \mapsto \lambda \cdot v, \\
                    &E \rightarrow [0, \infty), \quad v \mapsto \|v\|
                \end{align*}
                are continuous, and
            \item the sets
                \begin{align*}
                    \mathrm{V}(U, \varepsilon) \coloneqq \{ v \in s^{-1}(U) \mid \lVert v \rVert < \varepsilon \},
                \end{align*}
                where $U \subseteq M$ is an open neighborhood of $m$ and where $\varepsilon > 0$, form a neighborhood base of the zero element $0_m \in E_m$ for every $m \in M$. 
        \end{enumerate}
\end{definition}

We note that part (iii) of this definition in particular yields that the zero elements of a Banach bundle define a (continuous) section.\medskip

Now, given a Banach bundle $s \colon E \rightarrow M$ over a locally compact space $M$ write
    \begin{align*}
        \Gamma_s^{\mathrm{b}}(U) \coloneqq \biggl\{\tau \in \Gamma_s(U) \mid \sup_{m \in U} \|\tau(m)\| < \infty\biggr\}
    \end{align*}
for its space of bounded local sections on an open subset $U\subseteq M$. This is a Banach space with respect to the supremum norm given by $\|\tau\| \coloneqq \sup_{m \in U} \|\tau(m)\|$ for each $\tau \in  \Gamma_s^{\mathrm{b}}(U)$. The following result (see, e.g., \cite[Theorem 3.2]{Gierz1982}) shows that every Banach bundle has \enquote{many bounded sections}:

\begin{theorem}\label{existencesec}
      Let $s \colon E \rightarrow M$ be a Banach bundle. For every vector $v \in E$ there is some $\tau \in \Gamma_s^{\mathrm{b}}(M)$ with $\tau(s(v)) = v$.
\end{theorem}

Moreover, as shown by the subsequent lemma  (see \cite[Consequence 1.6 (vii) and Theorem 3.2]{Gierz1982}), the sections of a Banach bundle determine its topology.

\begin{lemma}\label{topbanachbun}
     Let $s \colon E \rightarrow M$ be a Banach bundle. For every bounded local section $\tau \in \Gamma_{s}^{\mathrm{b}}(U)$ on an open subset $U \subseteq M$ and $m \in M$ the sets  
        \begin{align*}
            W(\tau,V,\varepsilon) = \{v \in s^{-1}(V)\mid \|v - \tau(s(v))\| < \varepsilon\},
        \end{align*}
    for open neighborhoods $V \subseteq U$ of $m$ and $\varepsilon > 0$, define an open neighborhood base of the element $\tau(m) \in E$.
\end{lemma}

\begin{example}\label{bgloseccorr}
    Let $q \colon K \rightarrow M$ be an open proper Hausdorff topological bundle. A moment's thought reveals that the source map $\mathrm{s} \colon \mathrm{C}_q(K) \rightarrow M$ defines a Banach bundle over $M$ if we equip the fibers $\mathrm{C}(K_m)$ with the supremum norm for $m \in M$, cf. \cite[Example 4.5]{EdKr2021}. Moreover, the space of bounded global sections $\Gamma_{\mathrm{s}}^{\mathrm{b}}(U)$ over some open subset $U \subseteq M$ can be identified with the space $\mathrm{C}_{\mathrm{b}}(q^{-1}(U))$ of bounded complex-valued continuous functions on $q^{-1}(U)$ via the mutually inverse bijections
        \begin{align*}
            &\mathrm{C}_{\mathrm{b}}(q^{-1}(U)) \rightarrow \Gamma_{\mathrm{s}}^{\mathrm{b}}(U), \quad f \mapsto [m \mapsto f|_{K_m}],\\
            &\Gamma_{\mathrm{s}}^{\mathrm{b}}(U) \rightarrow \mathrm{C}_{\mathrm{b}}(q^{-1}(U)), \quad \tau \mapsto [x \mapsto \tau(q(x))(x)].
        \end{align*}
\end{example}

We now turn to representations of group bundles (see \cite[Definition 3.1]{Bos2011}, \cite[Definition 4.6]{EdKr2021}, and \cite[Section 4.3]{Herm2026} in the more general framework of groupoids). 

\begin{definition}
     Let $q \colon H \rightarrow M$ be an open topological group bundle over $M$. A \textbf{bounded strongly continuous representation} of $q$ on a Banach bundle $s \colon E \rightarrow M$ is a map $T \colon H \rightarrow \bigsqcup_{m \in M} \mathscr{L}(E_m),\, x \mapsto T_x$ with the following properties.
        \begin{enumerate}[(i)]
            \item $T_x \in \mathscr{L}(E_{q(x)})$ for each $x \in H$.
            \item $T_{xy} = T_x T_y$ for all $(x,y) \in H \times_{M} H$.
            \item $\sup_{x \in H} \lVert T_x \rVert < \infty$.
            \item For each $\tau \in \Gamma_s^{\mathrm{b}}(M)$ the map $H \rightarrow E, \, x \mapsto T_x \tau(q(x))$ is continuous.
         \end{enumerate}
\end{definition}

Topological dynamical systems introduced in the previous subsection naturally give rise to bounded strongly continuous representations:

\begin{example}
    Let $(q, \varphi)$ be a topological dynamical system with $q \colon K \rightarrow M$. Then
        \begin{align*}
            T^{\varphi} \colon G \rightarrow \bigsqcup_{m \in M} \mathscr{L}(\mathrm{C}(K_m)), \quad t \mapsto T_t^{\varphi},
        \end{align*}
    where $T_t^{\varphi}f \coloneqq f \circ \varphi_t^{-1}$ for $m \in M$, $f \in \mathrm{C}(K_m)$ and $t \in G$, is a bounded strongly continuous representation of $p$ on the Banach bundle $\mathrm{s} \colon \mathrm{C}_q(K) \rightarrow M$. We call this the induced \textbf{Koopman representation} of the topological dynamical system $(q, \varphi)$.
\end{example}

We now extend the concepts of eigenvectors, eigenvalues and eigenspaces to the case of group bundle representations.

\begin{definition}\label{spectralnotions}
    Let $q \colon H \rightarrow M$ be an open topological group bundle over $M$, and $T$ a bounded strongly continuous representation of $q$ on a Banach bundle $s \colon E \rightarrow M$. For a local character $\chi \in \mathcal{D}_q(U)$ defined on an open subset $U \subseteq M$ we introduce the following notions:
        \begin{enumerate}[(i)]
            \item The set
                \begin{align*}
                    \, \,\quad \ker(\chi -T) \coloneqq \{\tau \in \Gamma_s^{\mathrm{b}}(U)\mid T_x \tau(p(x)) = \mathrm{pr}_{\T}(\chi(x)) \cdot \tau(p(x)) \textrm{ for all } x \in q^{-1}(U)\}
                \end{align*}
            is called the \textbf{eigenspace} of $T$ with respect to $\chi$.
            \item An element $\tau \in  \ker(\chi -T)$ with $\|\tau(m)\| = 1$ for every $m \in U$ is a \textbf{(normalized) eigensection} with respect to $\chi$.
            \item The local character $\chi$ is an \textbf{eigencharacter} of $T$ if there exists an eigensection with respect to $\chi$.
        \end{enumerate}
    The family $\sigma_{\mathrm{p}}^T = (\sigma_{\mathrm{p}}^T(U))_{U}$ defined by
        \begin{align*}
            \sigma_{\mathrm{p}}^T(U) \coloneqq \{\chi \in \mathcal{D}_q(U) \mid \chi \textrm{ is an eigencharacter of } T\}
        \end{align*}
    for each open subset $U \subseteq M$ is the \textbf{point spectrum} of $T$.
\end{definition}

\begin{remarks}
    \begin{enumerate}[(i)]
        \item If $M = \{\mathrm{pt}\}$ we recover the usual definitions of the notions of eigenspaces, normalized eigenvectors and eigencharacters for group representations (see, e.g., \cite[Definition 4.1]{kreidler-hermle}). In particular, for $G= \Z$ one obtains the classical spectral notions for (invertible and doubly power-bounded) operators on Banach spaces up to the identification $\Z^* \cong \T$ (see \cite[Remark 4.2]{kreidler-hermle}).
        \item Even for invertible unitary matrices it is evidently not the case that the point spectrum is a subgroup of $\T$ in general. Thus, the point spectrum of a bounded strongly continuous representation $T$ of an open topological group bundle $q$ over $M$ does \emph{not} have to be a subpresheaf (or even a subsheaf) of the dual sheaf $\mathcal{D}_q$.
    \end{enumerate}
\end{remarks}

\begin{example}\label{exampleeigensect}
    Consider a group bundle compactification $(q,c)$ of $p$ and the induced rotation system $(q,\varphi_c)$ from Example \ref{examplestds} (iii). If $\chi \in \mathcal{D}_q(U)$ is a local character of $q$ for some open subset $U \subseteq M$, then 
        \begin{align*}
            (T^{\varphi_c}_t (\mathrm{pr}_{\T} \circ \overline{\chi}))(x) = \mathrm{pr}_{\T}(\overline{\chi}(c(t)^{-1}x)) =  (\mathrm{pr}_\T \circ \chi \circ c|_{p^{-1}(U)})(t) \cdot   (\mathrm{pr}_\T \circ \overline{\chi})(x)
        \end{align*}
    for all $t \in G$ and $x \in H_{p(t)}$. Thus, the local section defined by the bounded continuous function $\mathrm{pr}_\T \circ \overline{\chi} \in \mathrm{C}_{\mathrm{b}}(q^{-1}(U))$ (see Example \ref{bgloseccorr}) is an eigensection of the Koopman representation $T^{\varphi_c}$ with respect to the eigenvalue $\chi \circ c|_{p^{-1}(U)} \in (c^*\mathcal{D}_q)(U) \subseteq \mathcal{D}_p(U)$.
\end{example}

Using the above concept of eigenspaces we finally introduce a relative version of discrete spectrum for group bundle representations.

\begin{definition}\label{defdiscretespectrum}
    Let $q \colon H \rightarrow M$ be an open topological abelian group bundle over $M$. Then  a bounded strongly continuous representation $T$ of $q$ on a Banach bundle $s\colon E \rightarrow M$ has \textbf{relative discrete spectrum} if 
        \begin{align*}
            E_m = \overline{\lin} \{\tau(m)\mid \tau \in \ker(\chi - T) \textrm{ for } \chi \in \mathcal{D}_q(U), \, U \subseteq M \textrm{ open neighborhood of } m\}
        \end{align*}
    for each $m \in M$.
\end{definition}

Thus, loosely speaking, a representation has discrete spectrum if the eigenspaces generate the underlying Banach bundle. We establish equivalent definitions in terms of local approximation properties for bounded sections with respect to the supremum norm.

\begin{proposition}
     Consider an open topological abelian group bundle $q \colon H \rightarrow M$. For a bounded strongly continuous representation $T$ of $q$ on a Banach bundle $s\colon E \rightarrow M$ the following assertions are equivalent.
        \begin{enumerate}[(a)]
            \item $T$ has relative discrete spectrum.
            \item For each bounded local section $\tau \in \Gamma_s^{\mathrm{b}}(U)$ on an open subset $U \subseteq M$ and all $\varepsilon > 0$ there is an open covering $(U_i)_{i \in I}$ of $U$ such that for each $i \in I$ we have
                \begin{align*}
                    \mathrm{dist}\biggl(\tau|_{U_i}, \lin \bigcup_{\chi \in \mathcal{D}_q(U_i)} \mathrm{ker}(\chi - T)\biggr) \leq \varepsilon.
                \end{align*}
            %with respect to the supremum norm on $\Gamma_s^{\mathrm{b}}(U_i)$.
            \item For each bounded global section $\tau \in \Gamma_s^{\mathrm{b}}(M)$ and all $\varepsilon > 0$ there is an open covering $(U_i)_{i \in I}$ of $M$ such that for each $i \in I$ we have
                \begin{align*}
                    \mathrm{dist}\biggl(\tau|_{U_i}, \lin \bigcup_{\chi \in \mathcal{D}_q(U_i)} \mathrm{ker}(\chi - T)\biggr) \leq \varepsilon.
                \end{align*}
            %with respect to the supremum norm on $\Gamma_s^{\mathrm{b}}(U_i)$.
        \end{enumerate}
\end{proposition}

The result is a direct consequence of the following basic lemma. Note here, that for a Banach bundle $s \colon E \rightarrow M$ its local bounded sections -- with respect to addition of sections and with the natural restriction maps -- form a presheaf $\Gamma_s^{\mathrm{b}}$ over $M$ (in fact, they even define a so-called \emph{approximation sheaf}, see \cite[Subsection 3.5]{HoKe1977}).

\begin{lemma}
    Let $s \colon E \rightarrow M$ be a Banach bundle over the locally compact space $M$ and $\Delta$ a subpresheaf of  $\Gamma_s^{\mathrm{b}}$. Then the following assertions are equivalent.
        \begin{enumerate}[(a)]
            \item The set $\{\tau(m)\mid \tau \in \Delta(U), \, U \subseteq M \textrm{ open neighborhood of } m\}$ is dense in $E_m$ for each $m \in M$.
            \item For each bounded local section $\tau \in \Gamma_s^{\mathrm{b}}(U)$ on an open subset $U \subseteq M$ and every $\varepsilon > 0$ there exists an open covering $(U_i)_{i \in I}$ of $U$ such that for each $i \in I$ we have $\mathrm{dist}(\tau|_{U_i}, \Delta(U_i)) \leq \varepsilon$.
            \item For each bounded global section $\tau \in \Gamma_s^{\mathrm{b}}(M)$ and every $\varepsilon > 0$ there exists an open covering $(U_i)_{i \in I}$ of $M$ such that for each $i \in I$ we have $\mathrm{dist}(\tau|_{U_i}, \Delta(U_i)) \leq \varepsilon$.
        \end{enumerate}
\end{lemma}

\begin{proof}
    Assume that (a) holds and deduce (b). So let $\tau \in \Gamma_s^{\mathrm{b}}(U)$ for an open subset $U \subseteq M$. Then for all $m \in U$ there exists some open neighborhood $U_m'\subseteq U$ of $m$ and some $\tau_m \in \Delta(U_m')$ such that $\|\tau(m) - \tau_m(m)\|< \varepsilon$. Now, we set
        \begin{align*}
            U_m \coloneqq \{ m' \in U_m' \mid \| \tau(m') - \tau_m (m')
        \| < \varepsilon\}
    \end{align*}
    for each $m \in U$. Since $\tau$, $\tau_m$ and the norm are continuous maps, we obtain that $U_m$ is open for all $m \in U$. Furthermore, $U = \bigcup_{m \in U} U_m$. Since $\tau_m|_{U_m} \in \Delta(U_m)$ we obtain that $\mathrm{dist}(\tau|_{U_m}, \Delta(U_m)) \leq \varepsilon$ for each $m \in M$. Thus, we have shown (b).\medskip

    The implication \enquote{(b) $\Rightarrow$ (c)} is trivial. So assume now that (c) holds and show (a). Take $m \in M$, $v \in E_m$ and $\varepsilon > 0$.  By Theorem \ref{existencesec} there is a bounded global section $\tau \in \Gamma_s^{\mathrm{b}}(M)$ with $\tau(s(v)) = v$. By (c) (applied with $\frac{\varepsilon}{2}$) there exists an open covering $(U_i)_{i \in I}$ of $M$ and for each $i \in I$ some $\tau_i \in \Delta(U_i)$ such that $\|\tau|_{U_i} - \tau_i\| < \varepsilon$. In particular, we obtain
        \begin{align*}
            \|v - \tau_i(m)\| = \|\tau(m) - \tau_i(m)\|\leq \varepsilon
        \end{align*}
    as desired.
\end{proof}

Strongly continuous representations of compact abelian groups always have discrete spectrum, see, e.g., \cite[Corollary 15.18]{EFHN2015}. We obtain the subsequent analogue for group bundles.

\begin{proposition}\label{comprepdiscsp}
    Let $T$ be a bounded strongly continuous representation of an open proper Hausdorff abelian group bundle $q \colon H \rightarrow M$  on a Banach bundle $s \colon E \rightarrow M$. Then $T$ has relative discrete spectrum. 
\end{proposition}

The following auxiliary result is needed in the proof. Its first assertion can be seen as a relative version of the Banach--Alaoglu theorem. 

\begin{lemma}\label{compact}
    Let $s \colon E \rightarrow L$ be a Banach bundle over a compact space $L$. Then the set
        \begin{align*}
            B_E \coloneqq \bigsqcup_{l \in L} \{\gamma \in \mathscr{L}(E_l,\C)\mid \|\gamma\| \leq 1\}
        \end{align*}
    together with the canonical map $q_E \colon B_E \rightarrow L$ defines a compact bundle over $L$ with respect to the initial topology induced by the maps $B_E \rightarrow \C, \, \gamma \mapsto \gamma(\tau(q_{E}(l)))$ where $\tau \in \Gamma_s(L)$. Moreover, the map
        \begin{align*}
            B_E \times_L E \rightarrow \C, \quad (\gamma,x) \mapsto \gamma(x)
        \end{align*}
    is continuous.
\end{lemma}
\begin{proof}
    The first part is \cite[Proposition 15.3]{Gierz1982}. For the second, let $((\gamma_{\alpha},v_{\alpha}))_{\alpha}$ be a net in $B_E \times_L E$ converging to some $(\gamma,v) \in B_E \times_L E$. By Theorem \ref{existencesec} we find a section $\tau \in \Gamma_s(L)$ that satisfies $\tau(s(v)) = v$. Then $\lim_{\alpha} \|v_\alpha - \tau(s(v_{\alpha}))\| = 0$, and hence also
        \begin{align*}
            \lim_{\alpha} |\gamma_{\alpha}(v_\alpha) - \gamma_{\alpha}(\tau(s(v_{\alpha})))| =  \lim_{\alpha} |\gamma_{\alpha}(v_\alpha - \tau(s(v_{\alpha})))|= 0
        \end{align*}
    since $\|\gamma_{\alpha}\|\leq 1$ for each $\alpha$. But $\lim_{\alpha} \gamma_{\alpha}(\tau(s(v_{\alpha})))) = \gamma(\tau(s(v))) = \gamma(v)$ by the definition of the topology on $B_E$, and thus $\lim_{\alpha} \gamma_{\alpha}(v_\alpha) = \gamma(v)$ as well.
\end{proof}

\begin{proof}[Proof of Proposition \ref{comprepdiscsp}]
    Let $m_0 \in M$. Then the strongly continuous representation 
        \begin{align*}
            T^{m_0} \colon H_{m_0} \rightarrow \mathscr{L}(E_{m_0}), \quad x \mapsto T_x
        \end{align*}
    of the compact abelian group $H_{m_0}$ has discrete spectrum. Thus, the union of the eigenspaces $\ker(\varrho - T^{m_0})$ for $\varrho \in  H_{m_0}^*$ is total in $E_{m_0}$. We can therefore show the claim by establishing that for each $\varrho \in  H_{m_0}^*$ there is some local character $\chi \in \mathcal{D}_q(U)$ on an open neighborhood $U$ of $m_0$ with 
        \begin{align*}
            \ker(\varrho - T^{m_0}) \subseteq \{\tau(m_0) \mid \tau \in \ker(\chi - T)\}.
        \end{align*}
    To do so, we first use local compactness of $M$ to find an open neighborhood $V$ of $m_0$ with compact closure. By Proposition \ref{existenceconthaar} the restriction of $q|_{\overline{V}}$ has a continuous Haar system $(\lambda_{m})_{m \in \overline{V}}$. By rescaling we can assume that $\lambda_{m}$ is normalized, i.e., a probability measure on $H_{m}$ for each $m \in \overline{V}$.\medskip

    Since the dual bundle $q^*$ is \'{e}tale by Proposition \ref{dualpropetal}, we can apply Lemma \ref{etaleneighborhoodbase} and Proposition \ref{dualshvseciso} to find a local character $\chi \in \mathcal{D}_q(U)$ on an open neighborhood $U \subseteq V$ of $m$ such that $\mathrm{pr}_{\T} \circ \chi|_{H_{m_0}} = \varrho$. \medskip
    
    Now take $v \in \ker(\varrho - T^{m_0})$ and choose, using Theorem \ref{existencesec}, a bounded global section $\tau \in \Gamma_s^{\mathrm{b}}(M)$ with $\tau(m_0) = v$. We consider 
        \begin{align*}
            (P\tau)(m) \coloneqq  \int_{H_{m}} \overline{\mathrm{pr}_{\T}(\chi(x))} T_x\tau(m) \,\mathrm{d}\lambda_{m}(x) \quad \textrm{ for } m \in U,
        \end{align*}
    where the integral is understood in the weak sense (see, e.g., \cite[Definition 3.26]{rudinfa}). Then 
        \begin{align*}
            T_y (P\tau)(m) &= \int_{H_{m}} \overline{\mathrm{pr}_{\T}(\chi(x))} T_{yx}\tau(m) \,\mathrm{d}\lambda_{m}(x) = \int_{H_{m}} \overline{\mathrm{pr}_{\T}(\chi(y^{-1}x))} T_{x}\tau(m) \,\mathrm{d}\lambda_{m}(x)\\
            &= \mathrm{pr}_{\T}(\chi(y)) (P\tau)(m)
        \end{align*}
    for each $m \in U$ and $y \in H_m$ by rotation invariance of the Haar measure. Moreover, $(P\tau)(m_0) = \tau(m_0) = v$ since $v \in \ker(\varrho - T^{m_0})$. Thus, the proof is complete if we can show that $P\tau \colon U \rightarrow E$ is continuous, and hence defines a bounded local section of $s$.\medskip
    
    To do so, take a net $(m_{\alpha})_{\alpha}$ in $U$ converging to some $m \in U$. We use Lemma \ref{topbanachbun} to show that $\lim_{\alpha} (P\tau)(m_\alpha) = (P\tau)(m)$ in $E$. So take $\eta \in \Gamma_s^{\mathrm{b}}(M)$ with $\eta(m) = (P\tau)(m)$ and prove that $ \lim_{\alpha} \|P\tau(m_\alpha) - \eta(m_{\alpha})\| = 0$. By Hahn--Banach we find for each $\alpha$ some bounded linear functional $\gamma_{\alpha} \colon E_{m_\alpha} \rightarrow \C$ with $\|\gamma_{\alpha}\|\leq 1$ such that
        \begin{align*}
            \|(P\tau)(m_\alpha) - \eta(m_{\alpha})\| &= \gamma_{\alpha}(P\tau(m_\alpha) - \eta(m_{\alpha})) \\
            &= \int_{H_{m}} \mathrm{pr}_{\T}(\chi(x)) \gamma_{\alpha}(T_x\tau(m_{\alpha})) \,\mathrm{d}\lambda_{m_{\alpha}}(x) - \gamma_{\alpha}(\eta(m_{\alpha})).
        \end{align*}
    Assume that $\limsup_{\alpha}  \|P\tau(m_\alpha) - \eta(m_{\alpha})\| > 0$. We then find a subnet $(m_{\beta})_{\beta}$ of $(m_{\alpha})_{\alpha}$ such that the limit
        \begin{align*}
            c \coloneqq \lim_{\beta} \biggl|\int_{H_{m_{\beta}}} \mathrm{pr}_{\T} (\chi(x)) \gamma_{\beta}(T_x\tau(m_{\beta})) \,\mathrm{d}\lambda_{m_{\beta}}(x) - \gamma_{\beta}(\eta(m_{\beta}))\biggr| 
        \end{align*}
    exists and satisfies $c > 0$.\medskip
    
    Now let $W$ be a compact neighborhood of $m$ within $U$. By the first part of Lemma \ref{compact} we can pass to a subnet once again to assume that there is a bounded linear functional $\gamma \colon E_{m} \rightarrow \C$ such that $\gamma(\omega(m)) = \lim_{\beta} \gamma_{\beta}(\omega(m_\beta))$ for each $\omega \in \Gamma_{s_W}(s^{-1}(W))$. But then, as a consequence of the second part of Lemma \ref{compact}, the continuous functions
        \begin{align*}
            f_{\beta} \colon H_{m_{\beta}} \rightarrow \C, \quad x \mapsto \mathrm{pr}_{\T}(\chi(x)) \gamma_{\beta}(T_x\tau(m_{\beta}))
        \end{align*}
    converge to
        \begin{align*}
            f \colon  H_{m} \rightarrow \C, \quad x \mapsto \mathrm{pr}_{\T}(\chi(x)) \gamma(T_x\tau(m)).
        \end{align*}
    in $\mathrm{C}_q(H)$. If we now apply the second part of Lemma \ref{compact} to the Banach bundle $\mathrm{s} \colon \mathrm{C}_q(H) \rightarrow M$ (using that integration with respect to the Haar measure defines a bounded linear functional)\footnote{One can also directly apply \cite[Lemma 5.5]{EdKr2021} instead.}, we obtain that
        \begin{align*}
            \lim_{\beta} \gamma_{\beta}(P\tau(m)) =   \lim_{\beta} \int_{H_{m_{\beta}}} f_{\beta}(x)\, \mathrm{d}\lambda_{m_{\beta}}(x) =  \int_{H_{m}} f(x)\, \mathrm{d}\lambda_{m}(x) = \gamma(P\tau(m)).
        \end{align*}
    But we also obtain $\lim_{\beta}  \gamma_{\beta}(\eta(m_{\beta})) = \gamma(\eta(m)) = \gamma(P\tau(m))$, contradicting $c > 0$.
\end{proof}

\begin{example}\label{regdiscr}
    Let $q \colon H \rightarrow M$ open proper Hausdorff abelian group bundle. By Proposition \ref{comprepdiscsp} the \textbf{regular representation}
        \begin{align*}
            R \colon H \rightarrow \bigsqcup_{m \in M} \mathscr{L}(\mathrm{C}(H_m)), \quad x \mapsto R_x
        \end{align*}
    with $R_xf(y) \coloneqq f(x^{-1}y)$ for $(x,y) \in H \times_M H$ and $f \in \mathrm{C}_q(H)$ on the Banach bundle $\mathrm{s}\colon \mathrm{C}_q(H) \rightarrow M$ has relative discrete spectrum. 
\end{example}

We now return to the setting of topological dynamical systems.
\begin{definition}\label{defdiscretespectop}
    A topological dynamical system $(q, \varphi)$ with respect to $p$ has \textbf{relative discrete spectrum} if the induced Koopman representation $T^{\varphi}$ on $\mathrm{s}\colon \mathrm{C}_q(K) \rightarrow M$ has relative discrete spectrum.
\end{definition}

For a topological dynamical system $(q, \varphi)$ defined on an open proper Hausdorff bundle $q \colon K \rightarrow M$ we identify the eigenspaces $\ker(\chi - T^{\varphi})$ for $\chi \in \mathcal{D}_p(U)$ and $U \subseteq M$ open with subspaces of $\mathrm{C}_{\mathrm{b}}(q^{-1}(U))$ via the bijections from Example \ref{bgloseccorr}, i.e.,
    \begin{align*}
        \ker(\chi - T^{\varphi}) = \{f \in \mathrm{C}_{\mathrm{b}}(q^{-1}(U)) \mid T_t^{\varphi_c} f|_{K_{p(t)}} = \mathrm{pr}_{\T}(\chi(t)) f|_{K_{p(t)}} \textrm{ for all } t \in p^{-1}(U)\}.
    \end{align*}
With this convention, we obtain that the system $(q, \varphi)$ has relative discrete spectrum if and only if one of the following equivalent assertions holds.
    \begin{enumerate}[(a)]
        \item The set 
            \begin{align*}
                \{f|_{K_m} \mid f \in \ker(\chi - T^\varphi) \textrm{ for } \chi \in \mathcal{D}_p(U), \, U \subseteq M \textrm{ open neighborhood of } m\}
            \end{align*}
        is total in $\mathrm{C}(K_m)$ for each $m \in M$.
        \item For each open subset $U \subseteq M$, every $f \in \mathrm{C}_{\mathrm{b}}(q^{-1}(U))$ and each $\varepsilon > 0$ there is an open covering $(U_i)_{i \in I}$ of $U$ such that for each $i \in I$ we have 
                \begin{align*}
                    \mathrm{dist}\biggl(f|_{q^{-1}(U_i)}, \lin \bigcup_{\chi \in \mathcal{D}_p(U_i)} \mathrm{ker}(\chi - T^{\varphi})\biggr) \leq \varepsilon
                \end{align*}
        with respect to the supremum norm on $\mathrm{C}_{\mathrm{b}}(q^{-1}(U_i))$.
        \item For each $f \in \mathrm{C}_{\mathrm{b}}(K)$ and each $\varepsilon > 0$ there is an open covering $(U_i)_{i \in I}$ of $M$ such that for each $i \in I$ we have 
                \begin{align*}
                    \mathrm{dist}\biggl(f|_{q^{-1}(U_i)}, \lin \bigcup_{\chi \in \mathcal{D}_p(U_i)} \mathrm{ker}(\chi - T^{\varphi})\biggr) \leq \varepsilon
                \end{align*}
        with respect to the supremum norm on $\mathrm{C}_{\mathrm{b}}(q^{-1}(U_i))$.
    \end{enumerate}
Since for any $m \in M$ one can readily check that the linear hull of the set
    \begin{align*}
         \{f|_{K_m} \mid f \in \ker(\chi - T^\varphi) \textrm{ for } \chi \in \mathcal{D}_p(U), U \subseteq M \textrm{ open neighborhood of } m\}
    \end{align*}
unital *-subalgebra of $\mathrm{C}(K_m)$, the Stone--Weierstraß theorem  yields the following additional characterization of topological dynamical systems with relative discrete spectrum:
    \begin{enumerate}[(d)]
        \item For each $(x,y) \in K \times_M K$ we find a local character $\chi \in \mathcal{D}_p(U)$ on an open neighborhood $U$ of $p(x) = p(y)$ and $f \in \ker(\chi - T^\varphi)$ with $f(x) \neq f(y)$.
    \end{enumerate}

We now establish that the rotation systems introduced in Example \ref{examplestds} (iii) indeed have relative discrete spectrum:

\begin{proposition}\label{compactifdiscr}
    Let $(q, c)$ be a compactification of $p$. Then the induced topological dynamical system $(q, \varphi_c)$ has relative discrete spectrum.
\end{proposition}
\begin{proof}
    Write $q \colon H \rightarrow M$. The left regular representation $R$ of $p$ from Example \ref{regdiscr} has relative discrete spectrum. Since we clearly have $\ker(\chi - R) \subseteq \ker(\chi - T^{\varphi_c})$ for every local character $\chi \in \mathcal{D}_q(U)$ on an open subset $U \subseteq M$, we obtain that $T^{\varphi_c}$ has relative discrete spectrum as well.
\end{proof}

\subsection{Relatively Ergodic Systems}\label{ergodicitysec}

The classical topological Halmos--von Neumann theorem for group actions deals with minimal topological dynamical systems, i.e., transformations for which there are no non-trivial closed invariant subsets. For systems with discrete spectrum this is equivalent to \emph{topological ergodicity}, i.e., the absence of non-trivial fixed vectors of the induced Koopman representation (see  \cite[Proposition 1.12]{kreidler-hermle}). We now define a suitable concept of topological ergodicity in our situation (see also \cite[Subsections 3.2 and 3.3]{EdKr2021} and \cite[Subsections 5.1.2--5.1.4]{Herm2026} for groupoid actions with compact base spaces).

\begin{definition}
    Let $(q,\varphi)$ be a topological dynamical system on $q \colon K \rightarrow M$ and consider the induced Koopman representation $T^{\varphi}$ on the Banach bundle $\mathrm{s} \colon \mathrm{C}_q(K) \rightarrow M$. Then the eigenspace
        \begin{align*}
            \fix(T^{\varphi}) \coloneqq \ker(\1 - T)  = \{f \in \mathrm{C}_{\mathrm{b}}(K)\mid T^{\varphi}_t f|_{K_{p(t)}} = f|_{K_{p(t)}} \textrm{ for every } t \in G\} 
        \end{align*}
     of the trivial character $\1_p \colon p \rightarrow \mathrm{triv}_M^{\T}, \, x \mapsto (p(x),1)$ is called the \textbf{fixed space} of $T^{\varphi}$. We also set $\fix_0(T^{\varphi}) \coloneqq \fix(T^{\varphi}) \cap \mathrm{C}_0(K)$ and $\fix_{\mathrm{c}}(T^{\varphi}) \coloneqq \fix(T^{\varphi}) \cap \mathrm{C}_{\mathrm{c}}(K)$.
\end{definition}

Clearly, if $(q,\varphi)$ is a topological dynamical system and $f \in \mathrm{C}_{\mathrm{b}}(M)$, then the pulledback $f \circ q \colon K \rightarrow \C$ is an element of $\fix(T^{\varphi})$. Relative ergodicity now captures the situation that these are the only elements of $\fix(T^{\varphi})$.

\begin{definition}
    A topological dynamical system $(q,\varphi)$ is called \textbf{relatively (topologically) ergodic}\footnote{It would be less cumbersome to simply speak of \emph{ergodicity} rather than \emph{relative ergodicity} here. However, since an action of a group bundle is a special case of a groupoid action, we want to stay consistent with the concept of relative ergodicity from \cite{EdKr2021} and \cite[Subsection 5.1.4]{Herm2026} for groupoid actions. It then also seems natural to use the term \emph{relative discrete spectrum} (instead of only \emph{discrete spectrum}) for representations and dynamical systems of a group bundle.} if $\fix(T^{\varphi}) = \{f \circ q \mid f \in \mathrm{C}_{\mathrm{b}}(M)\}$.
\end{definition}

We derive some basic characterizations.

\begin{lemma}\label{basicchartoperg}
    For a topological dynamical system $(q,\varphi)$ the following assertions are equivalent.
        \begin{enumerate}[(a)]
            \item If $f \in \fix(T^{\varphi})$, then $f|_{K_m}$ is constant for each $m \in M$.
            \item $(q,\varphi)$ is relatively ergodic.
            \item $\fix_0(T^{\varphi}) =  \{f \circ q \mid f \in \mathrm{C}_{0}(M)\}$.
            \item $\fix_{\mathrm{c}}(T^{\varphi}) =  \{f \circ q \mid f \in \mathrm{C}_{\mathrm{c}}(M)\}$.
        \end{enumerate}
\end{lemma}
\begin{proof}
    Write $q \colon K \rightarrow M$. The implication \enquote{(b) $\Rightarrow$ (a)} is trivial. Conversely, if (a) holds, assume that there exists $g \in \fix(T^{\varphi}) \setminus  \{f \circ q \mid f \in \mathrm{C}_{\mathrm{b}}(M)\}$. If $g|_{K_m}$ is constant for each $m \in M$, then we obtain, using openness of $q$, that $g$ factors continuously through $M$, a contradiction.\medskip
    
    For the implications  \enquote{(b) $\Rightarrow$ (c)} and \enquote{(c) $\Rightarrow$ (d)} it suffices to observe that, since $q$ is proper, for $f \in \mathrm{C}_{\mathrm{b}}(M)$ we have $f \circ q \in \mathrm{C}_0(K)$ precisely when $f \in \mathrm{C}_0(M)$, and $f \circ q \in \mathrm{C}_{\mathrm{c}}(K)$ precisely when $f \in \mathrm{C}_{\mathrm{c}}(M)$.\medskip
    
    For the remaining implication \enquote{(d) $\Rightarrow$ (a)} assume that there is $f \in \fix(T^{\varphi})$ such that there exists $m \in M$ with $f|_{K_m}$ non-constant. Let $h \in \mathrm{C}_{\mathrm{c}}(M)$ with $h(m) = 1$. Then $(h \circ q) \cdot f \in \fix_{\mathrm{c}}(T^{\varphi})$ is still not constant on $K_m$. This proves the result.
\end{proof}

We obtain that rotation systems defined by group bundle compactifications are always ergodic:

\begin{proposition}\label{rotergodic}
     Let $(q, c)$ be a group bundle compactification of $p$. Then the induced topological dynamical system $(q,\varphi_c)$ is relatively ergodic.
\end{proposition}

\begin{proof}
    Write $q \colon H \rightarrow M$ and take $f \in \fix(T^{\varphi_c})$. By Lemma \ref{basicchartoperg} it suffices to check that $f|_{H_m}$ is constant for each $m \in M$.
    If we set
        \begin{align*}
            S \coloneqq \{ x \in H \mid f(xy) = f(y) \textrm{ for all } y \in H_{q(x)} \},
        \end{align*}
    then it is clear that $S$ contains $c(G)$ and that $S$ defines a subsemigroup bundle of $p$. We will show below that $S$ is also closed in $H$. Thus,  $S=H$, i.e., $f(xy) = f(y)$ for all $(x,y) \in H \times_{M} H$. This shows $f(x) = f(1_{H_m})$ for all $m \in M$ and $x \in H_m$, as desired.\medskip

    To check that $S$ is closed in $H$, let $(x_{\alpha})_{\alpha}$ be a net in $S$ converging to some $x \in H$. Now take $y \in H_{q(x)}$. Since $q$ is open, we find a subnet  $(x_{\beta})_{\beta}$ of $(x_{\alpha})_{\alpha}$ and elements $y_{\beta} \in H_{q(x_{\beta})}$ with $\lim_{\beta} y_{\beta} = y$. But then
        \begin{align*}
            f(xy) = \lim_{\beta} f(x_{\beta}y_{\beta}) = \lim_{\beta} f(y_{\beta}) = f(y).
        \end{align*}
    This shows $x \in S$.
\end{proof}

Let us give a different perspective on relative ergodicity. Observe that for a topological dynamical system $(q,\varphi)$ on an open proper Hausdorff bundle $p \colon K \rightarrow M$ the subspace $\fix_0(T^{\varphi}) \subseteq \mathrm{C}_0(M)$ is actually a C*-subalgebra of $\mathrm{C}_0(K)$. Since it contains all pullbacks  $f \circ q$ for $f \in \mathrm{C}_0(M)$, it is \textbf{non-degenerate}, i.e., $\fix_0(T^{\varphi}) \cdot \mathrm{C}_0(K)$ spans a dense subspace of $\mathrm{C}_0(K)$, by \cite[Appendix C.6]{Land2017}.\medskip

Now, recall that by the Gelfand--Naimark theorem, there is an equivalence between the category of locally compact spaces with proper continuous maps as morphisms, and the opposite category of commutative C*-algebras with non-degenerate *-homomorphisms as morphisms, see again \cite[Appendix C.6]{Land2017}. In particular, we obtain that for any locally compact space $\Omega$ there is a correspondence between non-degenerate C*-subalgebras $B \subseteq \mathrm{C}_0(\Omega)$ on one hand, and proper continuous surjections $r \colon \Omega \rightarrow L$ onto locally compact spaces $L$ on the other: If $r \colon \Omega \rightarrow L$ is a continuous and proper surjection onto a locally compact space $L$, then 
    \begin{align*}
        r^*(\mathrm{C}_0(L)) \coloneqq \{f \circ r \mid f \in \mathrm{C}_0(L)\}
    \end{align*}
is a non-degenerate C*-subalgebra of $\mathrm{C}_0(\Omega)$. Conversely, for any  non-degenerate C*-subalgebra $B \subseteq \mathrm{C}_0(\Omega)$ its Gelfand space $\mathbb{X}(B)$ is a locally compact space, and the inclusion map $B \hookrightarrow \mathrm{C}_0(\Omega)$ induces a continuous and proper surjection $\pi_{B} \colon \Omega \rightarrow \mathbb{X}(B)$. Moreover, these constructions are, at least up to an isomorphism, inverse to each other:
    \begin{enumerate}[(i)]
        \item  For each non-degenerate C*-subalgebra $B \subseteq \mathrm{C}_0(\Omega)$ we have $\pi_B^*(\mathrm{C}_0(L)) = B$.
        \item  If $r \colon \Omega \rightarrow L$ is a continuous and proper surjection, then there is a unique homeomorphism $h \colon L \rightarrow \mathbb{X}(B)$ with $h \circ r = \pi_{r^*(\mathrm{C}_0(L))}$.
    \end{enumerate}

\begin{definition}
    Let  $(q,\varphi)$ be a topological dynamical system on an open proper Hausdorff bundle $q \colon K \rightarrow M$. The space $K_{\fix} \coloneqq \mathbb{X}(\fix_0(T^{\varphi}))$ is called the \textbf{fixed factor} of $(q,\varphi)$, and we write $q_{\fix} \coloneqq \pi_{\fix_0(T^{\varphi})} \colon K \rightarrow K_{\fix}$ for the corresponding proper continuous surjection.
\end{definition}

By definition, $\fix_0(T^{\varphi})$ is the largest non-degenerate C*-subalgebra $B \subseteq \mathrm{C}_0(K)$ such that $T^{\varphi}_t f|_{K_{p(t)}} = f|_{K_{p(t)}}$ for every $f \in B$ and each $t \in G$. This readily implies that $q_{\fix}$ is invariant, i.e., $q_{\fix}(\varphi_t(x)) = q_{\fix}(x)$ for all $t \in G$ and $x \in K_{p(t)}$, and that $q_{\fix}$ is maximal with this property in the following sense: Whenever $r \colon K \rightarrow L$ is an invariant proper continuous surjection onto a locally compact space $L$, there is a unique proper continuous surjection $h \colon K_{\fix} \rightarrow L$ such that the diagram
    \[
		\xymatrix{
			K  \ar[rr]^{q_{\fix}} \ar[rd]_{r}  &  &K_{\fix} \ar[ld]^{h} \\
            & L  & \\
		}
	\]	
commutes. This universal property characterizes $q_{\fix}$ up to an isomorphism. We refer to \cite[Subsection 5.1.2]{Herm2026} for this concept in the framework of groupoid actions over a compact base space.  
\medskip

In particular, the open proper continuous surjection $q \colon K \rightarrow M$ factors through $K_{\fix}$, i.e., we can write $q = h_q \circ q_{\fix}$ for a unique proper continuous surjection $h_q \colon K_{\fix} \rightarrow M$.  Using this point of view, the system $(q,\varphi)$ is relatively ergodic precisely when the map $h_q \colon K_{\fix} \rightarrow M$ is a homeomorphism (and hence the fixed factor is canonically isomorphic to $M$).\medskip

For later use we discuss yet another perspective: By \cite[Proposition 10.4.9]{Bour1995} there is, for a given locally compact space $\Omega$, a correspondence between proper continuous surjections $r \colon \Omega \rightarrow L$ to onto locally compact spaces $L$ on one the hand, and  closed equivalence relations $\sim \, \, \subseteq \Omega \times \Omega$ on $\Omega$ with compact equivalence classes on the other. If $r \colon \Omega \rightarrow L$ is a proper continuous surjection to onto a locally compact space $L$, letting $x \sim_r y$ for $x,y \in \Omega$ precisely when $r(x) = r(y)$ defines a closed equivalence relation with compact equivalence classes. Conversely, the factor map defined by such an equivalence relation defines a proper continuous surjection. Once again, these constructions are mutually inverse up to an isomorphism.

\begin{definition}\label{deffixrel}
    Let  $(q,\varphi)$ be a topological dynamical system on an open proper Hausdorff bundle $p \colon K \rightarrow M$. The closed equivalence relation with compact equivalence classes on $K$ defined by $q_{\fix} \colon K \rightarrow K_{\fix}$ is denoted by $\sim_{\fix}$.
\end{definition}

By the universal property of the fixed factor we obtain that $\sim_{\fix}$ is the smallest closed equivalence relation with compact equivalence classes on $K$ containing the \textbf{orbit relation}
    \begin{align*}
        \sim_{\varphi} \, \, \coloneqq \{(x,\varphi_t(x))\mid t \in G, x \in K_{p(t)}\} \subseteq K \times K.
    \end{align*}
In fact, it is even the smallest among all closed equivalence relations on $K$ containing the orbit relation $\sim_{\varphi}$: If $\sim$ denotes the intersection of these, then $\sim$ must be a subset of $K \times_M K$, and hence automatically has compact equivalence classes since $q$ is proper.\medskip

In view of the above considerations we obtain the subsequent characterizations of relative ergodicity.

\begin{proposition}\label{chartoperg}
     For a  topological dynamical system  $(q,\varphi)$ on an open proper Hausdorff bundle $q \colon K \rightarrow M$ the following assertions are equivalent.
        \begin{enumerate}[(a)]
            \item $(q,\varphi)$ is relatively ergodic.
            \item For every invariant proper continuous surjection $r \colon K \rightarrow L$ there is a unique proper continuous surjection $h \colon M \rightarrow L$ with $r = h \circ q$.
            \item The smallest closed  equivalence relation with compact equivalence classes on $K$ containing the orbit relation $\sim_{\varphi}$ is $K \times_M K$.
            \item The smallest closed equivalence relation on $K$ containing the orbit relation $\sim_{\varphi}$ is $K \times_M K$.
        \end{enumerate}
\end{proposition}

\subsection{Uniform Enveloping Semigroup Bundles}\label{univenvsec}

As seen in Proposition \ref{compactifdiscr}, every compactification of our open topological abelian group bundle $p \colon G \rightarrow M$ gives rise to a topological dynamical system with relative discrete spectrum. We now go the other direction and construct a group compactification from a given system $(q,\varphi)$ on an open proper Hausdorff bundle $q \colon K \rightarrow M$ with relative discrete spectrum.\medskip

In the classical case of group actions, i.e., if $M = \{\mathrm{pt}\}$ is a singleton space, one can construct a group compactification by introducing the \textbf{uniform enveloping semigroup} as the closure
    \begin{align*}
        \mathrm{E}_{\mathrm{u}}(K,\varphi) = \overline{\{\varphi_t \mid t \in G\}} \subseteq \mathrm{C}(K,K)
    \end{align*}
with respect to the compact-open topology, equipped with the composition of maps as multiplication. This is a variation of the classical enveloping semigroup in topological dynamics introduced by Ellis in \cite{Ellis1960}. For systems with discrete spectrum it is a compact group (and indeed agrees with the Ellis semigroup in this case), hence gives rise to a group compactification
    \begin{align*}
        G \rightarrow \mathrm{E}_{\mathrm{u}}(K,\varphi), \quad t \mapsto \varphi_t,
    \end{align*}
see \cite[Subsection 1.2]{kreidler-hermle}.\medskip

We shall now introduce a \enquote{bundle version} of this construction. This is (apart from allowing locally compact base spaces instead of compact ones here) a special case of the uniform enveloping \emph{semigroupoid} or the uniform enveloping \emph{poloid} from \cite[Section 3]{EdKr2021} and \cite[Subsection 5.1.5]{Herm2026}, respectively. Recall from Proposition \ref{semigrpbundlefibermaps} that for any locally compact bundle $q \colon \Omega \rightarrow M$ we obtain a topological semigroup bundle by restricting $\mathrm{s}\colon \mathrm{C}_q^q(\Omega,\Omega) \rightarrow M$ to the subspace
    \begin{align*}
        \mathrm{C}(q) = \bigsqcup_{m \in M} \mathrm{C}(K_m, K_m) \subseteq \mathrm{C}_q^q(K, K).
    \end{align*}
The following definition also uses the notation for generated subsemigroup bundles from Definition \ref{generatedbundle}.

\begin{definition}\label{defunifenvsem}
    For a topological dynamical system $(q, \varphi)$ on an open proper Hausdorff bundle $q \colon K \rightarrow M$ the subsemigroup bundle $\mathrm{E}_{\mathrm{u}}(q, \varphi)$ defined by 
        \begin{align*}
            \mathrm{E}_{\mathrm{u}}(q, \varphi) \coloneqq \langle \{ \varphi_t \colon K_{p(t)} \rightarrow K_{p(t)} \mid t \in G \} \rangle \subseteq \mathrm{C}(q)
        \end{align*}
    is called the \textbf{uniform enveloping semigroup bundle} of $(q,\varphi)$.
\end{definition}

For rotation systems defined by group bundle compactifications we can explicitly compute the uniform enveloping semigroup bundle (cf. \cite[Proposition 5.1.29]{Herm2026} for groupoid compactifications of topologically ergodic groupoids).

\begin{proposition}\label{envforcomp}
    Consider a group bundle compactification $(q, c)$ where $q \colon H \rightarrow M$. Then
        \begin{align*}
            \Phi \colon H \rightarrow \mathrm{E}_{\mathrm{u}}(q, \varphi_c), \quad x \mapsto \vartheta_x,
        \end{align*}
    with $\vartheta_x(y) \coloneqq xy$ for all $x \in H$ and $y \in H_{q(x)}$, defines an isomorphism of (semi)group bundles over $M$. 
\end{proposition}

\begin{proof}
    One readily checks that 
        \begin{align*}
            \Phi \colon H \rightarrow \mathrm{C}(q), \quad x \mapsto \vartheta_x,
        \end{align*}
    where $\vartheta_x(y) = xy$ for all $x \in H$ and $y \in H_{q(x)}$, defines a morphism of topological semigroup bundles over $M$. Since $q$ is proper, the image $\Phi(H)$ is a closed subset of $\mathrm{C}(q)$ by Proposition \ref{properclosed}. Moreover, the image $\Phi(H)$ defines a subsemigroup bundle and contains 
        \begin{align*}
            \{ (\varphi_c)_t \colon K_{p(t)} \rightarrow K_{p(t)} \mid t \in G \}
        \end{align*}
    by definition of $\varphi_c$. We conclude that $\mathrm{E}_{\mathrm{u}}(q, \varphi_c) \subseteq \Phi(H)$. On the other hand, the preimage $\Phi^{-1}(\mathrm{E}_{\mathrm{u}}(q, \varphi_c))$ is a closed subset of $H$ containing $c(G)$, and defines a subsemigroup bundle of $q$. Thus, $\Phi^{-1}(\mathrm{E}_{\mathrm{u}}(q, \varphi_c)) = H$, and consequently $\Phi(H) \subseteq \mathrm{E}_{\mathrm{u}}(q, \varphi_c)$. In conclusion, we obtain that the corestriction
        \begin{align*}
            \Phi \colon  H \rightarrow \mathrm{E}_{\mathrm{u}}(q, \varphi_c), \quad x \mapsto \vartheta_x
        \end{align*}
    defines a surjective morphism of topological semigroup bundles over $M$. To prove injectivity, let $x, y \in H$ with $\vartheta_x = \vartheta_y$. Since $\Phi$ is a morphism of topological bundles, we then have $m \coloneqq q(x) = q(y)$. But then
        \begin{align*}
            x = \vartheta_{x}(1_{H_{m}}) = \vartheta_{y}(1_{H_{m}}) = y.
        \end{align*}
    By Corollary \ref{automaticcontprop} we conclude that $\Phi$ defines an isomorphism of topological group bundles over $M$. 
\end{proof}

In particular, we see that for rotation systems the uniform enveloping semigroup bundles is even a proper abelian group bundle. The following result shows that this is true for arbitrary topological dynamical systems with relative discrete spectrum (cf. \cite[Theorems 3.27 and 4.14]{EdKr2021} and \cite[Proposition 5.1.37]{Herm2026} for actions of \enquote{topologically ergodic groupoids}). Recall the definition of the homeomorphism group bundle of a proper locally compact bundle from Proposition \ref{semigrpbundlefibermaps}.

\begin{proposition} \label{unifenvcomp}
    Let $(q, \varphi)$ be a topological dynamical system with relative discrete spectrum. Then $\mathrm{E}_{\mathrm{u}}(q, \varphi)$ defines a proper abelian subgroup bundle of the homeomorphism group bundle $\mathrm{s}|_{\mathrm{Homeo}(q)} \colon \mathrm{Homeo}(q) \rightarrow M$. 
\end{proposition}

We first show two auxiliary statements. The proof of the first one is similar to the one of Proposition \ref{rotergodic}.

\begin{lemma} \label{lemmaunifsemigrp}
    Let $(q, \varphi)$ be a topological dynamical system. Let further $f \in \mathrm{ker}(\chi - T^{\varphi})$ for some local character $\chi \in \mathcal{D}_p(U)$ on an open subset $U \subseteq M$. Then for each element $\vartheta \in \mathrm{E}_{\mathrm{u}}(q, \varphi)$ with $\mathrm{s}(\vartheta) \in U$ there exists an element $z \in \mathbb{T}$ such that
        \begin{align*}
            f|_{K_{\mathrm{s}(\vartheta)}} \circ \vartheta = z f|_{K_{\mathrm{s}(\vartheta)}}.
        \end{align*}
\end{lemma}

\begin{proof}
    We set
    \begin{align*}
        S \coloneqq \{ \vartheta \in \mathrm{E}_{\mathrm{u}}(q, \varphi) \mid \text{if } \mathrm{s}(\vartheta) \in U,
        \text{ then there is a } z \in \mathbb{T} \text{ with }
f|_{K_{\mathrm{s}(\vartheta)}} \circ \vartheta = z f|_{K_{\mathrm{s}(\vartheta)}} \}.
    \end{align*}
    Clearly, $S$ contains the set
        \begin{align*}
            \{\varphi_t \colon K_{p(t)} \rightarrow K_{p(t)} \mid t \in G\} 
        \end{align*}
    and defines a subsemigroup bundle of $\mathrm{s}|_{\mathrm{E}_{\mathrm{u}}(q, \varphi)}$. We show that $S$ is closed in $\mathrm{E}_{\mathrm{u}}(q, \varphi)$, which then implies the claim by definition of $\mathrm{E}_{\mathrm{u}}(q, \varphi)$. So let $(\vartheta_{\alpha})_{\alpha}$ be a net in $S$ converging to some $\vartheta \in \mathrm{E}_{\mathrm{u}}(q, \varphi)$. To show that $\vartheta \in S$ we may assume that $\mathrm{s}(\vartheta) \in U$. Then there exists an $\alpha_0$ such that $\mathrm{s}(\vartheta_\alpha) \in U$ for all $\alpha \geq \alpha_0$. For each $\alpha \geq \alpha_0$ we then find some $z_\alpha \in \mathbb{T}$ with
$f|_{K_{\mathrm{s}(\vartheta_\alpha)}} \circ \vartheta_\alpha = z_\alpha f|_{K_{\mathrm{s}(\vartheta_\alpha)}}$. By compactness of $\T$ we can pass to a subnet to achieve that the limit $z \coloneqq \lim_{\alpha} z_{\alpha}$ exists in $\T$. But then
    \begin{align*}
        z f|_{K_{\mathrm{s}(\vartheta)}} = \lim_{\alpha} z_{\alpha} f|_{K_{\mathrm{s}(\vartheta_{\alpha})}} = \lim_{\alpha}  f|_{K_{\mathrm{s}(\vartheta_{\alpha})}} \circ \vartheta_{\alpha} = f|_{K_{\mathrm{s}(\vartheta)}} \circ \vartheta,
    \end{align*}
    and hence $\vartheta \in S$ as desired.
\end{proof}

The second lemma is a group bundle version of \cite[Proposition 3.11]{EdKr2021} which also covers the case of locally compact base spaces.

\begin{lemma}\label{compactsubgrpbundle}
    Let $(q, \varphi)$ be a topological dynamical system. If  $\mathrm{s}|_{\mathrm{E}_{\mathrm{u}}(q, \varphi)} \colon \mathrm{E}_{\mathrm{u}}(q, \varphi) \rightarrow M$ is proper, then it is a subgroup bundle of $\mathrm{s}|_{\mathrm{Homeo}(q)} \colon \mathrm{Homeo}(q) \rightarrow M$. 
\end{lemma}
\begin{proof}
    Let
        \begin{align*}
            S \coloneqq \{ \vartheta \in \mathrm{E}_{\mathrm{u}}(q, \varphi) \ | \ \vartheta \in \mathrm{Homeo}(q) \text{ and }
\vartheta^{-1} \in \mathrm{E}_{\mathrm{u}}(q, \varphi) \}.
        \end{align*}
    It is clear that $S$ contains the set $\{\varphi_t \colon K_{p(t)} \rightarrow K_{p(t)}\mid t \in G\}$ and defines a subsemigroup bundle of $\mathrm{s}|_{\mathrm{E}_{\mathrm{u}}(q, \varphi)}$. By definition of $\mathrm{E}_{\mathrm{u}}(q, \varphi)$, it therefore suffices to prove that $S$ is closed in $\mathrm{E}_{\mathrm{u}}(q, \varphi)$. So let $(\vartheta_\alpha)_{\alpha}$ be a net in $S$ converging to some $\vartheta \in \mathrm{E}_{\mathrm{u}}(q, \varphi)$. Since $\mathrm{s}|_{\mathrm{E}_{\mathrm{u}}(q, \varphi)}$ is proper we may assume, by passing to a subnet, that $(\vartheta_\alpha^{-1})_{\alpha}$ also converges to some element $\vartheta' \in \mathrm{E}_{\mathrm{u}}(q, \varphi)$ (see Lemma \ref{charproper}). But then, by continuity of the multiplication, we have
        \begin{align*}
            \vartheta' \circ \vartheta = \lim_{\alpha} \vartheta_\alpha^{-1} \circ \vartheta_{\alpha} = \mathrm{id}_{K_{\mathrm{s}(\vartheta)}}  = \lim_{\alpha} \vartheta_{\alpha} \circ \vartheta_\alpha^{-1} = \vartheta \circ \vartheta'.
        \end{align*}
    Thus, $\vartheta \in \mathrm{Homeo}(q)$ with $\vartheta^{-1} \in \mathrm{E}_{\mathrm{u}}(q, \varphi)$, and consequently $\vartheta \in S$.
\end{proof}
\begin{remark}
    Lemma \ref{compactsubgrpbundle} is also a direct consequence of an abstract assertion on  topological \emph{poloids}, see \cite[Theorem 4.1.48]{Herm2026}.
\end{remark}

\begin{proof}[Proof of Proposition \ref{unifenvcomp}]
    We write $q \colon K \rightarrow M$ and first check that $\mathrm{s}|_{\mathrm{E}_{\mathrm{u}}(q, \varphi)}$ is a proper bundle. So let $C \subseteq M$ be compact and show that $\mathrm{s}|_{\mathrm{E}_{\mathrm{u}}(q, \varphi)}^{-1}(C) \subseteq \mathrm{C}_q^q(q^{-1}(C),q^{-1}(C))$ is compact. By the Arzel\'a--Ascoli theorem (see Theorem \ref{arzasc}) it suffices to show that $\mathrm{s}|_{\mathrm{E}_{\mathrm{u}}(q, \varphi)}^{-1}(C)$ is uniformly equicontinuous. Since the uniform structure on $q^{-1}(C)$ is generated by the entourages
        \begin{align*} 
            V_{f,\varepsilon} = \{(x,y) \in q^{-1}(C) \times q^{-1}(C) \mid |f(x) - f(y)| < \varepsilon\}
        \end{align*}
    for $f \in \mathrm{C}(q^{-1}(C))$ and $\varepsilon > 0$, we equivalently have to check that 
        \begin{align*}
            f \circ \mathrm{E}_{\mathrm{u}}(q, \varphi) \coloneqq \{ f|_{K_{\mathrm{s}(\vartheta)}} \circ \vartheta \mid \vartheta \in \mathrm{E}_{\mathrm{u}}(q, \varphi) \textrm{ with } \mathrm{s}(\vartheta) \in C\}
        \end{align*}
    is equicontinuous for each $f \in \mathrm{C}(q^{-1}(C))$. So fix $f \in \mathrm{C}(q^{-1}(C))$. We extend $f$ to a compactly supported continuous function on $K$ which we still denote by $f$. Since the system $(q, \varphi)$ has relative discrete spectrum, we find for a given $\varepsilon > 0$ an open covering $(U_i)_{i \in I}$ of $M$ such that for each $i \in I$ there exists a function $g_i \in \lin \bigcup_{\chi \in \mathcal{D}_p(U_i)} \ker(\chi - T^{\varphi})$ with $\|f|_{q^{-1}(U_i)} - g_i\|_\infty \leq \varepsilon$. As $M$ is locally compact, we find for each $i \in I$ an open cover $(V_j^i)_{j \in J_i}$ of $U_i$ such that the closure $\overline{V_j^i}$ within $M$ is compact and still contained in $U_i$ for every $j \in J_i$. Since $C$ is compact, we now find a finite subset $I' \subseteq I$ and finite subsets $J_i' \subseteq J_i$ for each $i\in I'$ such that $C \subseteq \bigcup_{i \in I'} \bigcup_{j \in J_i'} V_j^i$. Thus, for each $\vartheta \in \mathrm{E}_{\mathrm{u}}(q, \varphi)$ we find some $i \in I'$ and $j \in J_i'$ with $\mathrm{s}(\vartheta) \in V_j^i$, and hence
        \begin{align*}
             \|f|_{K_{\mathrm{s}(\vartheta)}} \circ \vartheta - g_i|_{K_{\mathrm{s}(\vartheta)}} \circ \vartheta \|_{\infty} \leq \varepsilon.
        \end{align*}
    Since $\bigcup_{i \in I'} J_i'$ is finite, it suffices by an easy application of the triangle inequality to check that
        \begin{align*}
            \biggl\{g_i|_{K_{\mathrm{s}(\vartheta)}}\circ \vartheta \mid \vartheta \in \mathrm{E}_{\mathrm{u}}(q, \varphi) \textrm{ with } \mathrm{s}(\vartheta) \in \overline{V_j^i}\biggr\}
        \end{align*}
    is a uniformly equicontinuous subset of $\mathrm{C}_q(q^{-1}(\overline{V_j^i}))$ for each $i \in I'$ and $j \in J_i'$.\medskip
    
    Replacing $f$ by one of these functions $g_i|_{q^{-1}(\overline{V_j^i})}$ for $i \in I'$ and $j \in J_i'$, we may therefore assume without loss of generality that $f \in \mathrm{C}(q^{-1}(C))$ can be extended to an element $f \in \lin \bigcup_{\chi \in \mathcal{D}_p(U)} \ker(\chi - T^{\varphi})$ for some open subset $U \subseteq M$ containing $C$. It is then straightforward to further reduce to the case that we even have $f \in \ker(\chi - T^{\varphi})$ for some local character $\chi \in \mathcal{D}_p(U)$ on $U$.\medskip
    
    Applying Theorem \ref{arzasc} once more, it suffices to prove that $f \circ \mathrm{E}_{\mathrm{u}}(q, \varphi)$ is precompact in $\mathrm{C}_q(q^{-1}(C))$. However, Lemma \ref{lemmaunifsemigrp} yields 
        \begin{align*}
            f \circ \mathrm{E}_{\mathrm{u}}(q, \varphi)  \subseteq \{z \cdot f|_{K_m} \mid z \in \T,\,  m \in C\},
        \end{align*}
    and the set on the right-hand side is compact in $\mathrm{C}_q(q^{-1}(C))$. This establishes that $\mathrm{s}|_{\mathrm{E}_{\mathrm{u}}(q, \varphi)}$ is proper.  It now follows from Lemma \ref{compactsubgrpbundle} that  $\mathrm{s}|_{\mathrm{E}_{\mathrm{u}}(q, \varphi)}$ is a proper subgroup bundle of $\mathrm{s}|_{\mathrm{Homeo}(q)}$.\medskip
        
    Lastly, we show that $\mathrm{E}_{\mathrm{u}}(q, \varphi)$ is abelian. So let $\vartheta, \vartheta' \in \mathrm{E}_{\mathrm{u}}(q, \varphi)$ with
$m \coloneqq\mathrm{s}(\vartheta) = \mathrm{s}(\vartheta')$, and show $\vartheta \circ \vartheta' = \vartheta' \circ \vartheta$. Since the continuous functions on $K_m$ separate points, it suffices to prove
        \begin{align*}
            f \circ \vartheta \circ \vartheta' = f \circ \vartheta' \circ \vartheta \quad \text{for all } \,
            f \in \mathrm{C}(K_{m}).
        \end{align*}
    Since $(q,\varphi)$ has relative discrete spectrum, we may restrict to the case that $f = g|_{K_{m}}$, where $g \in \mathrm{ker}(\chi - T^{\varphi})$ for some local character
$\chi \in \mathcal{D}_p(U)$ on an open neighborhood $U \subseteq M$ of $m$. But then we find, by Lemma \ref{lemmaunifsemigrp}, numbers $w,z \in \T$ with $f \circ \vartheta = w f$ and $f \circ \vartheta' = w' f$. This yields
    \begin{align*}
        f \circ \vartheta \circ \vartheta' &= (w f) \circ \vartheta'
        = w (f \circ \vartheta') = w w' f
        = w'w f = w' (f \circ \vartheta) 
        = (w' f) \circ \vartheta\\
        &= f \circ \vartheta' \circ \vartheta,
    \end{align*}
    as desired.
\end{proof}

Note that we have \emph{not} shown the uniform enveloping semigroup bundle of a topological dynamical system with relative discrete spectrum to be open. This will be done -- under the assumption of topological ergodicity -- in Theorem \ref{unifenvcomp2} below. As a prepration we first show the following result. For compact base spaces this is a consequence of  \cite[Proposition 5.9]{EdKr2021} (up to noticing that the proof there only uses compactness of the uniform enveloping semigroupoid, and not the stated assumption of a pseudoisometric action). 

\begin{proposition}\label{charergtransit}
    Let $(q, \varphi)$ be a topological dynamical system and write $q \colon K \rightarrow M$. If $\mathrm{s}|_{\mathrm{E}_{\mathrm{u}}(q, \varphi)} \colon \mathrm{E}_{\mathrm{u}}(q, \varphi) \rightarrow M$ is proper, then the following assertions are equivalent.
        \begin{enumerate}[(a)]
            \item $(q,\varphi)$ is relatively ergodic.
            \item $\mathrm{E}_{\mathrm{u}}(q, \varphi)_m$ acts transitively on $K_m$ for each $m \in M$.
        \end{enumerate}
\end{proposition}

Proposition \ref{charergtransit} follows directly from the subsequent lemma.

\begin{lemma} \label{lemmafixorb}
    Let $(q, \varphi)$ be a topological dynamical system and write $q \colon K \rightarrow M$. If $\mathrm{s}|_{\mathrm{E}_{\mathrm{u}}(q, \varphi)} \colon \mathrm{E}_{\mathrm{u}}(q, \varphi) \rightarrow M$ is proper, then the relation $\sim_{\fix}$ is the orbit relation defined by the action of the group bundle $\mathrm{s}|_{\mathrm{E}_{\mathrm{u}}(q, \varphi)} \colon \mathrm{E}_{\mathrm{u}}(q, \varphi) \rightarrow M$, i.e., 
        \begin{align*}
            \sim_{\fix} \,\, = \{(x,\vartheta(x))\mid \vartheta \in \mathrm{E}_{\mathrm{u}}(q, \varphi),\, x \in K_{\mathrm{s}(\vartheta)}\}.
        \end{align*}
\end{lemma}
\begin{proof}
    Write $\sim$ for the equivalence relation on the right-hand side. Using that the bundle $\mathrm{s}|_{\mathrm{E}_{\mathrm{u}}(q, \varphi)} \colon \mathrm{E}_{\mathrm{u}}(q, \varphi) \rightarrow M$ is a proper group bundle (see Proposition \ref{compactsubgrpbundle}), it is readily checked that $\sim$ is a closed equivalence relation. Since the orbit relation 
        \begin{align*}
            \{(x,\varphi_t(x))\mid t \in G, \, x \in K_{p(t)}\}
        \end{align*}
    is contained in $\sim$, we conclude that $\sim_{\fix}\, \, \subseteq \, \, \sim$. Now set
        \begin{align*}
            S \coloneqq \{\vartheta \in \mathrm{E}_{\mathrm{u}}(q, \varphi)\mid x \, \sim_{\fix} \, \vartheta(x) \textrm{ for every } x \in K_{\mathrm{s}(\vartheta)}\}.
        \end{align*}
    Then $S$ contains $\{\varphi_t \colon K_{p(t)} \rightarrow K_{p(t)}\mid t \in G\}$ and defines a subsemigroup bundle of $\mathrm{S}(q)$. We show that $S$ is also closed in $\mathrm{E}_{\mathrm{u}}(q, \varphi)$, which will imply $\sim \, \, \subseteq \, \, \sim_{\fix}$, and hence the claim. So let $(\vartheta_\alpha)_{\alpha}$ be a net in $S$ converging to some $\vartheta \in \mathrm{E}_{\mathrm{u}}(q,\varphi)$. Let further $x \in K_{\mathrm{s}(\vartheta)}$. Since $q$ is open we may assume, by passing to a subnet, that there is a net $(x_\alpha)_{\alpha}$ in $K$ with $x_{\alpha} \in K_{\mathrm{s}(\vartheta_{\alpha})}$ for each $\alpha$ which converges to $x$. But then $x_{\alpha} \, \sim_{\fix} \,  \vartheta_{\alpha}(x_{\alpha})$ for each $\alpha$, and hence also $x \, \sim_{\fix} \, \vartheta(x)$ since $\sim_{\fix}$ is closed.
\end{proof}

We are now ready to deduce the following strengthening of Proposition \ref{unifenvcomp} providing a characterization of relative discrete spectrum for relatively ergodic systems.

\begin{theorem}\label{unifenvcomp2}
    For a topological dynamical system $(q, \varphi)$ consider the following assertions.
        \begin{enumerate}[(a)]
            \item The map $\mathrm{s}|_{\mathrm{E}_{\mathrm{u}}(q, \varphi)}$ is an open proper abelian subgroup bundle of the homeomorphism group bundle $\mathrm{s}|_{\mathrm{Homeo}(q)} \colon \mathrm{Homeo}(q) \rightarrow M$. 
            \item The system $(q,\varphi)$ has relative discrete spectrum.
        \end{enumerate}
     Then \enquote{(a) $\Rightarrow$ (b)}. If $(q,\varphi)$ is relatively ergodic, then (a) and (b) are equivalent.
\end{theorem}
\begin{proof}
    Write $q\colon K \rightarrow M$ and assume first that (a) holds. Then
        \begin{align*}
            S \colon \mathrm{E}_{\mathrm{u}}(q,\varphi) \rightarrow \bigsqcup_{m \in M} \mathscr{L}(\mathrm{C}(K_m)), \quad \vartheta \mapsto S_{\vartheta}
        \end{align*}
    with $S_{\vartheta}f\coloneqq f \circ \vartheta^{-1}$ for $\vartheta \in \mathrm{E}_{\mathrm{u}}(q,\varphi)$ and $f \in \mathrm{C}(K_{\mathrm{s}(\vartheta)})$ is a bounded strongly continuous representation of $\mathrm{s}|_{\mathrm{E}_{\mathrm{u}}(q, \varphi)}$. Thus, it has relative discrete spectrum by Proposition \ref{comprepdiscsp}. If we write $c \colon G \rightarrow \mathrm{E}_{\mathrm{u}}(q,\varphi), \, t \mapsto \varphi_t$, then
        \begin{align*}
            \ker(\chi - S) \subseteq \ker(\chi \circ c - T^{\varphi})
        \end{align*}
    for each local character $\chi \in \mathcal{D}_{s|_{\mathrm{E}_{u}(q,\varphi)}}(U)$ on an open subset $U \subseteq M$. Thus, $(q,\varphi)$ has relative discrete spectrum as well.\medskip
    
    Now assume that $(q,\varphi)$ is relatively ergodic and that (b) holds. In view of Proposition \ref{unifenvcomp} and  Lemma \ref{abeliangenerator}, the only assertion left to verify is that the bundle $\mathrm{s}|_{\mathrm{E}_{\mathrm{u}}(q, \varphi)}$ is open. So let $\vartheta \in \mathrm{E}_{\mathrm{u}}(q,\varphi)$ and $(m_{\alpha})_{\alpha}$ be a net in $M$ converging to $m \coloneqq \mathrm{s}(\vartheta)$. Pick $x \in K_m$ and set $y \coloneqq \vartheta(x)$. Since $q$ is open, we find a subnet $(m_{\beta})_{\beta}$ of $(m_{\alpha})_{\alpha}$ and for each $\beta$ some $x_{\beta} \in K_{m_\beta}$ such that $\lim_{\beta} x_{\beta} = x$. Passing to a subnet once again, we may further assume that there are $y_{\beta} \in K_{m_{\beta}}$ for each $\beta$ such that $\lim_{\beta} y_{\beta} = y$. By Proposition \ref{charergtransit} we find for each $\beta$ some $\vartheta_{\beta} \in \mathrm{E}_{\mathrm{u}}(q,\varphi)_{m_{\beta}}$ with $\vartheta_{\beta}(x_{\beta}) = y_{\beta}$. Using that $\mathrm{s}|_{\mathrm{E}_{\mathrm{u}}(q, \varphi)}$ is proper, we may pass to a subnet for a final time to also assume that the limit $\vartheta' \coloneqq \lim_{\beta} \vartheta_{\beta}$ exists in $\mathrm{E}_{\mathrm{u}}(q,\varphi)$ (see Lemma \ref{charproper}). But then
        \begin{align*}
            \vartheta(x) = y = \lim_{\beta} y_{\beta} = \lim_{\beta}  \vartheta_{\beta}(x_{\beta}) = \vartheta'(x).
        \end{align*}
    Since transitive and effective actions of abelian groups are free, we conclude that $\vartheta' = \vartheta$, and hence $\vartheta = \lim_{\beta} \vartheta_{\beta}$. This shows that $\mathrm{s}|_{\mathrm{E}_{\mathrm{u}}(q, \varphi)} \colon \mathrm{E}_{\mathrm{u}}(q, \varphi)  \rightarrow M$ is open.
\end{proof}

Note that the assumption of being abelian in the implication \enquote{(a) $\Rightarrow$ (b)} of Theorem \ref{unifenvcomp2} is actually redundant (see Lemma \ref{abeliangenerator}). However, the condition of being open is not:

\begin{example}\label{examplepseudoisometric}
    Equip the symmetric group $\mathrm{Sym}(4)$ of the four-element set $\{1,2,3,4\}$ with the discrete topology. We further pick two non-commuting $4$-cyclic permutations $\zeta,\zeta' \in \mathrm{Sym}(4)$, e.g., $\zeta = (1\, 2\, 3\, 4)$ and $\zeta' = (1\, 2\, 4\, 3)$. Assume that $p$ is the open abelian subgroup bundle of the trivial group bundle $\mathrm{triv}_{[-1,1]}^{\mathrm{Sym}(4)}$ given by
        \begin{align*}
            G_m = \begin{cases} \{m\} \times  \{(\zeta)^j\mid j \in \{0,1,2,3\}\} & \textrm{ for } m \in [-1,0),\\
            \{m\} \times  \{\mathrm{id}\} & \textrm{ for } m = 0,\\
            \{m\} \times \{(\zeta')^j\mid j \in \{0,1,2,3\}\} & \textrm{ for } m \in (0,1].\\
            \end{cases}
        \end{align*}
    Now consider the trivial bundle $q = \mathrm{triv}_{[-1,1]}^{\{1,2,3,4\}}$ of the discrete space $\{1,2,3,4\}$ over $[-1,1]$.  We then obtain a topological dynamical system $(q,\varphi)$ over $p$ be restricting the natural action of $\mathrm{triv}_{[-1,1]}^{\mathrm{Sym}(4)}$ on $q = \mathrm{triv}_{[-1,1]}^{\{1,2,3,4\}}$ to $p$, i.e.,
        \begin{align*}
            \varphi_{(m,\xi)}(m,x) = (m,\xi(x)) \quad \textrm{ for } \, m \in [-1,1],\, x \in \{1,2,3,4\}, \, \textrm{ and } \,  (m,\xi) \in G_m.
        \end{align*}    
        
    Now let $\langle \{\zeta,\zeta'\}\rangle$ be the (non-abelian) subgroup of $\mathrm{Sym}(4)$ generated by the subset $\{\zeta,\zeta'\}$, and consider the compact subgroup bundle $r \colon H \rightarrow [-1,1]$ of $\mathrm{triv}_{[-1,1]}^{\mathrm{Sym}(4)}$ given by
        \begin{align*}
            H_m = \begin{cases} \{m\} \times  \{(\zeta)^j\mid j \in \{0,1,2,3\}\} & \textrm{ for } m \in [-1,0),\\
            \{m\} \times  \langle \{\zeta,\zeta'\}\rangle & \textrm{ for } m = 0,\\
            \{m\} \times \{(\zeta')^j\mid j \in \{0,1,2,3\}\} & \textrm{ for } m \in (0,1].\\
            \end{cases}
        \end{align*}
    Then  $\Phi \colon H \rightarrow \mathrm{E}_{\mathrm{u}}(q,\varphi)$ with $(\Phi(m,\xi))(m,x) = (m,\xi(x))$ for $m \in [-1,1]$, $x \in \{1,2,3,4\}$ and $(m,\xi) \in H_m$ defines an isomorphism between the topological (semi)group bundles $r$ and $\mathrm{s}|_{\mathrm{E}_{\mathrm{u}}(q,\varphi)}$. In particular, we see that $\mathrm{s}|_{\mathrm{E}_{\mathrm{u}}(q, \varphi)}$ is a proper subgroup bundle of the homeomorphism group bundle $\mathrm{s}|_{\mathrm{Homeo}(q)} \colon \mathrm{Homeo}(q) \rightarrow M$. Moreover, the system $(q,\varphi)$ is relatively ergodic by Proposition \ref{charergtransit}. However, it does not have relative discrete spectrum, since the group bundle $\mathrm{s}|_{\mathrm{E}_{\mathrm{u}}(q, \varphi)}$ is not abelian (and consequently not open either).
\end{example}

\begin{remark}
    Recall that in the group case (i.e., if $M= \{\mathrm{pt}\}$ is a singleton space) discrete spectrum is equivalent to two classical notions of structuredness. In fact, for a topological dynamical system $(K,\varphi)$ over a topological abelian group $G$ the following assertions are equivalent (see \cite[Theorem 1.11]{kreidler-hermle}).
        \begin{enumerate}[(a)]
            \item $(K,\varphi)$ has discrete spectrum.
            \item $(K,\varphi)$ is \textbf{pseudoisometric}, i.e., there exists a family of invariant pseudometrics on $K$ generating the topology.
            \item $(K,\varphi)$ is \textbf{equicontinuous}, i.e.,  $\{\varphi_t\mid t \in G\}$ is a uniformly equicontinuous subset of $\mathrm{C}(K,K)$.
        \end{enumerate}
     With suitable generalizations of these notions to \enquote{groupoid actions with compact unit space}, \cite[Theorems 3.27 and 4.14]{EdKr2021} show that the equivalence of (a) and (b) still holds true in the framework of actions of \enquote{topologically ergodic} groupoids. On the other hand, the analogue of (c) is strictly weaker in this situation (see \cite[Proposition 1.17 and Example 3.15]{EdKr2021}).\medskip
     
     Now consider a relatively ergodic system $(q,\varphi)$ over an open group bundle $p \colon G \rightarrow M$ with compact base space $M$. If $(q,\varphi)$ has relative discrete spectrum, then Theorem \ref{unifenvcomp2} combined with the proof method of \cite[Proposition 3.17]{EdKr2021} yields that $(q,\varphi)$ is pseudoisometric in the sense of \cite[Definition 1.15]{EdKr2021}. On the other hand, one can readily check that Example \ref{examplepseudoisometric} is pseudoisometric (and, by \cite[Proposition 1.17]{EdKr2021}, also equicontinuous), and yet does \emph{not} have relative discrete spectrum. Thus, in our framework, the concept of discrete spectrum (over a compact base space) is stronger than the above classical notions of structuredness from topological dynamics.
\end{remark}

In the last part of this subsection, we note that Theorem \ref{unifenvcomp2} allows us to assign a group compactification to each relatively ergodic system with relative discrete spectrum. In fact, this construction is even functorial (see also \cite[Lemma 3.19]{EdKr2021} for a related result in the groupoid setting):

\begin{proposition}\label{compfunctor}
    \begin{enumerate}[(i)]
        \item If $(q, \varphi)$ is a relatively ergodic system with relative discrete spectrum, then the pair $(\mathrm{s}|_{\mathrm{E}_{\mathrm{u}}(q, \varphi)}, i_{(q, \varphi)})$ with
            \begin{align*}
                i_{(q, \varphi)} \colon G \rightarrow \mathrm{E}_{\mathrm{u}}(q, \varphi), \quad
t \mapsto \varphi_t
            \end{align*}
        is a group bundle compactification of $p$.
        \item Let $\Phi \colon (q, \varphi) \rightarrow (q', \varphi')$ be a morphism between relatively ergodic systems with relative discrete spectrum, and write $q \colon K \rightarrow M$. Then $\Phi$ is surjective and we have a morphism of group bundle compactifications
            \begin{align*}
               \Phi^{\sharp} \colon (\mathrm{s}|_{\mathrm{E}_{\mathrm{u}}(q, \varphi)}, i_{(q, \varphi)}) \rightarrow (\mathrm{s}|_{\mathrm{E}_{\mathrm{u}}(q', \varphi')}, i_{(q', \varphi')}), \quad \vartheta \mapsto \Phi^{\sharp}(\vartheta),
            \end{align*}
        where $\Phi^{\sharp}(\vartheta)(\Phi(x)) \coloneqq \Phi(\vartheta(x))$ for every 
$\vartheta \in \mathrm{E}_{\mathrm{u}}(q, \varphi)$ and $x \in K_{\mathrm{s}(\vartheta)}$.
    \end{enumerate}
\end{proposition}

\begin{proof}
    Part (i) is a consequence of Theorem \ref{unifenvcomp2}. To show part (ii), write $q' \colon K'\rightarrow M$, and first consider the set
        \begin{align*}
            S' \coloneqq \{\vartheta' \in \mathrm{E}_{\mathrm{u}}(q', \varphi')\mid \exists \, \vartheta \in  \mathrm{E}_{\mathrm{u}}(q, \varphi)_{\mathrm{s}(\vartheta')} \textrm{ with } \vartheta'(\Phi(x)) = \Phi(\vartheta(x))  \textrm{ for every } x \in K_{\mathrm{s}(\vartheta')}\}.
        \end{align*}
    It is clear that $S'$ contains the set $\{\varphi_t'\colon K_{p(t)} \rightarrow K_{p(t)} \mid t \in G\}$ and that $\mathrm{s}|_{S'}$ defines a subsemigroup bundle of $\mathrm{s}|_{\mathrm{E}_{\mathrm{u}}(q', \varphi')}$. We show that $S$ is closed in $\mathrm{E}_{\mathrm{u}}(q', \varphi')$. So let $(\vartheta'_{\alpha})_{\alpha}$ be a net in $S$ converging to some $\vartheta' \in \mathrm{E}_{\mathrm{u}}(q', \varphi')$. For each $\alpha$ choose $\vartheta_{\alpha} \in \mathrm{E}_{\mathrm{u}}(q, \varphi)_{\mathrm{s}(\vartheta'_{\alpha})}$ with $\vartheta'_{\alpha}(\Phi(x)) = \Phi(\vartheta_{\alpha}(x))$ for each $x \in K_{\mathrm{s}(\vartheta'_{\alpha})}$. Since $\mathrm{s}|_{\mathrm{E}_{\mathrm{u}}(q, \varphi)}$ is proper by Theorem \ref{unifenvcomp2} we may assume, by passing to a subnet, that the limit $\vartheta = \lim_{\alpha} \vartheta_{\alpha}$ exists in $\mathrm{E}_{\mathrm{u}}(q, \varphi)$ (see Lemma \ref{charproper}), and clearly $\mathrm{s}(\vartheta) = \mathrm{s}(\vartheta')$. Now pick $x \in K_{\mathrm{s}(\vartheta')}$. Since $q$ is open, we find a subnet $(\vartheta'_{\beta})_{\beta}$ of $(\vartheta'_{\alpha})_{\alpha}$ and elements $x_{\beta} \in K_{\mathrm{s}(\vartheta'_{\beta})}$ for each $\beta$ such that $x = \lim_{\beta} x_{\beta}$. But then 
        \begin{align*}
            \vartheta'(\Phi(x)) = \lim_{\beta} \vartheta'_{\beta}(\Phi(x_{\beta})) = \lim_{\beta} \Phi(\vartheta_{\beta}(x_{\beta}))  = \Phi(\vartheta(x)) 
        \end{align*}
    since $\Phi$ is continuous. Hence $S$ is closed, and therefore $S' = \mathrm{E}_{\mathrm{u}}(q', \varphi')$.\medskip

    We now obtain that $\Phi$ is surjective: If $y \in K'$, choose $x \in K_{q'(y)}$. By Proposition \ref{charergtransit} we find $\vartheta' \in \mathrm{E}_{\mathrm{u}}(q', \varphi')_{q'(y)}$ with $\vartheta'(\Phi(x)) = y$. But then, since  $S = \mathrm{E}_{\mathrm{u}}(q', \varphi')$, we find $\vartheta \in \mathrm{E}_{\mathrm{u}}(q, \varphi)_{q'(y)}$ with $\vartheta'(\Phi(x)) = \Phi(\vartheta(x))$, and hence $y = \Phi(\vartheta(x)) \in \mathrm{im}(\Phi)$.\medskip

    Now apply a similar argument as above to see that the set
        \begin{align*}
            S \coloneqq \{\vartheta \in \mathrm{E}_{\mathrm{u}}(q, \varphi)\mid \exists \, \vartheta' \in  \mathrm{E}_{\mathrm{u}}(q', \varphi')_{\mathrm{s}(\vartheta)} \textrm{ with } \vartheta'(\Phi(x)) = \Phi(\vartheta(x))  \textrm{ for every } x \in K_{\mathrm{s}(\vartheta)}\}
        \end{align*}
    is $\mathrm{E}_{\mathrm{u}}(q, \varphi)$. But, since $\Phi$ is surjective, this simply means that for each $\vartheta \in \mathrm{E}_{\mathrm{u}}(q, \varphi)$ the map $\Phi^{\sharp}(\vartheta) \colon K_{\mathrm{s}(\vartheta)}' \rightarrow K_{\mathrm{s}(\vartheta)}'$ with $\vartheta'(\Phi(x)) = \Phi(\vartheta(x))$  for each $x \in K_{\mathrm{s}(\vartheta')}$ is well-defined and an element of $\mathrm{E}_{\mathrm{u}}(q', \varphi')_{\mathrm{s}(\vartheta)}$.\medskip

    To check that $\Phi^{\sharp}$ is continuous, it suffices by Proposition \ref{lemmaclosedgraph} to prove that its graph is a closed subset of $\mathrm{E}_{\mathrm{u}}(q, \varphi) \times_{M} \mathrm{E}_{\mathrm{u}}(q', \varphi')$. Let $(\vartheta_{\alpha})_{\alpha}$ be a net in $\mathrm{E}_{\mathrm{u}}(q, \varphi) \times_{M} \mathrm{E}_{\mathrm{u}}(q', \varphi')$ such that $\lim_{\alpha} \vartheta_{\alpha} = \vartheta \in \mathrm{E}_{\mathrm{u}}(q, \varphi)$ and $\lim_{\alpha} \Phi^{\sharp}(\vartheta_{\alpha}) = \vartheta' \in \mathrm{E}_{\mathrm{u}}(q', \varphi')$. For $x \in K_{\mathrm{s}(\vartheta)}$ take a subnet $(\vartheta_{\beta})_{\beta}$ of $(\vartheta_{\alpha})_{\alpha}$ and elements $x_{\beta} \in K_{\mathrm{s}(\vartheta_\beta)}$ with $\lim_{\beta} x_{\beta} = x$. Then
        \begin{align*}
            \Phi^\sharp(\vartheta)(\Phi(x)) = \Phi(\vartheta(x)) = \lim_{\beta} \Phi(\vartheta_{\beta}{\beta}(x_{\beta})) = \lim_{\beta} \Phi^\sharp(\vartheta_{\beta})(\Phi(x_{\beta})) = \vartheta'(\Phi(x)).
        \end{align*}
    Since  $x \in K_{\mathrm{s}(\vartheta)}$ was arbitrary, this shows $\vartheta' = \Phi^{\sharp}(\vartheta)$ as desired. Thus  $\Phi^{\sharp}$ is continuous, and it is now easy to check that it defines a morphism of group bundle compactifications.
\end{proof}

%Under these assumptions, \cite[Theorem 7.2]{EdKr2021} shows that having the relative analogue of discrete spectrum is equivalent to being a pseudoisometric extension (or, if the system $(K,\varphi)$ is minimal, to being an equicontinuous extension, see \cite[Corollary 5.10]{deVr1993})). Thus, classical and important concepts of structuredness for extensions $q \colon (K,\varphi) \rightarrow (M,\psi)$ in topological dynamics can be encoded in a natural way in terms of the induced Koopman representations $T^{\varphi}$ and $T^{\psi}$ \emph{if the base system $(M,\psi)$ is topologically ergodic}.

%It turns out that a topological dynamical system with discrete spectrum can equivalently be characterized as \emph{equicontinuous} or \emph{pseudoisometric} systems (see, e.g., \cite[Section 1.2]{kreidler-hermle}).

\subsection{A Halmos--von Neumann Theorem for Abelian Group Bundles}\label{hvnsec}

Recall from the introduction that the representation aspect of the classical Halmos--von Neumann theorem for a topological abelian group states that each ergodic (or equivalently: minimal) system with discrete spectrum is isomorphic to a rotation system defined by a group compactification.\medskip

However, the straightforward extension of this assertion to the group bundle framework does not hold. In fact, for any rotation system $(q,\varphi_c)$ defined by a group compactification $(q,\varphi_c)$ of $p$ the the bundle $q \colon H \rightarrow M$ has the global section $M \rightarrow H, \, m \mapsto 1_{H_m}$. But -- as shown by the following simple example (see  \cite[Example 2.10]{edeko}) -- a general relatively ergodic system $(q,\varphi)$ with relative discrete spectrum does \emph{not} have to admit any global section.

\begin{example}
    Assume that $p = \mathrm{triv}_\T^{\Z} \colon \T \times \Z \rightarrow \T$ is the trivial bundle over $M= \T$ with fiber $\Z$. The open compact bundle $q \colon \T \rightarrow \T, \, z \mapsto z^2$ does not have a global section (cf. \cite[Example 2.6]{edeko}). Now consider the topological dynamical system $(q,\varphi)$ defined by the map
        \begin{align*}
            \varphi \colon (\T \times \Z) \times_{p,q} \T \rightarrow \T, \quad ((w,k),z) \mapsto (-1)^kz.
        \end{align*}
    Then
        \begin{align*}
            \T \times \Z/2\Z \rightarrow \mathrm{E}_{\mathrm{u}}(q,\varphi), \quad (w,[k]_{2\Z}) \mapsto \psi_{(w,[k]_{2\Z})},
        \end{align*}
    where $\psi_{(w,[k]_{2\Z})}(z) = (-1)^kz$ for $z \in q^{-1}(\{w\})$ and $(w,[k]_{2\Z}) \in \T \times \Z/2\Z$, defines an isomorphism of topological group bundles $\mathrm{triv}_{\T}^{\Z/2\Z} \rightarrow \mathrm{s}|_{\mathrm{E}_{\mathrm{u}}(q,\varphi)}$. In particular, the bundle $\mathrm{s}|_{\mathrm{E}_{\mathrm{u}}(q,\varphi)}$ is an open proper abelian subgroup bundle of the homeomorphism group bundle $\mathrm{s}|_{\mathrm{Homeo}(q)} \colon \mathrm{Homeo}(q) \rightarrow M$, and its fiber groups act transitively. By Proposition \ref{charergtransit} and Theorem \ref{unifenvcomp2} the system $(q,\varphi)$ is therefore relatively ergodic and has relative discrete spectrum.
\end{example}

In the following, we therefore focus on relatively ergodic systems with relative discrete spectrum which also have a global section. The next definition introduces the relative version of the \enquote{category of pointed minimal system with discrete spectrum} from \cite[Definition 4.10]{kreidler-hermle}.

\begin{definition}\label{defpointed}
    The \textbf{category of relatively ergodic systems with relative discrete spectrum and distinguished section (over $p$)}, denoted by $\mathbf{ErgDisc}_{\mathrm{sec}}(p)$, has  
        \begin{enumerate}[(i)]
            \item all triples $(q,\varphi,\tau)$ as objects, where $(q,\varphi)$ is a relatively ergodic system with relative discrete spectrum and $\tau$ is a global section of $q$, and 
            \item morphisms $\Phi \colon (q_1,\varphi_1) \rightarrow (q_2,\varphi_2)$ of topological dynamical systems with $\Phi \circ \tau_1 = \tau_2$ as morphisms $\Phi \colon  (q_1,\varphi_1,\tau_1) \rightarrow (q_2,\varphi_2,\tau_2)$ between such triples.
        \end{enumerate}
\end{definition}

Group bundle compactifications give rise to relatively ergodic systems with relative discrete spectrum and distinguished section in a natural and functorial way.

\begin{definition} \label{rotfunct}
    Define a functor $\mathrm{Rot}\colon \mathbf{Comp}(p) \rightarrow \mathbf{ErgDisc}_{\mathrm{sec}}(p)$ by setting
        \begin{enumerate}[(i)]
            \item $\mathrm{Rot}(q,c) \coloneqq (q, \varphi_c, 1)$ for every group bundle compactification $(q, c)$ of $p$ (where $1$ denotes the section of neutral elements of $q$), and 
            \item $\mathrm{Rot}(\Phi) \coloneqq \Phi$ for each morphism $\Phi \colon (q, c) \rightarrow (q', c')$ of group bundle compactifications.
        \end{enumerate}
    We call this the \textbf{rotation functor}.
\end{definition}

Proposition \ref{compfunctor} allows us to construct a functor in the converse direction.

\begin{definition}
    Define a functor $\mathrm{Env} \colon \mathbf{ErgDisc}_{\mathrm{sec}}(p) \rightarrow \mathbf{Comp}(p)$ by setting
        \begin{enumerate}[(i)]
            \item $\mathrm{Env}(q, \varphi,\tau) \coloneqq (\mathrm{s}|_{\mathrm{E}_{\mathrm{u}}(q, \varphi)}, i_{(q, \varphi)})$ for each object
$(q, \varphi, \tau)$ in $\mathbf{ErgDisc}_{\mathrm{sec}}(p)$, and
            \item $\mathrm{Env}(\Phi) \coloneqq \Phi^{\sharp}$ for every morphism $\Phi \colon (q, \varphi, \tau) \rightarrow (q', \varphi', \tau')$ in $\mathbf{ErgDisc}_{\mathrm{sec}}(p)$.
        \end{enumerate}
    We call this the \textbf{enveloping group bundle functor}.
\end{definition}

Our goal now is to establish that the above two functors are essentially inverse to each other. We  start with the following result. 

\begin{proposition}\label{natisocomprot1}
    For every relatively ergodic system with relative discrete spectrum and distinguished section $(q, \varphi, \tau)$ we have an isomorphism
        \begin{align*}
            \gamma_{(q, \varphi, \tau)} \colon \mathrm{Rot}(\mathrm{Env}(q, \varphi, \tau)) \rightarrow (q, \varphi, \tau), \quad \vartheta \mapsto \vartheta(\tau(\mathrm{s}(\vartheta))).
        \end{align*}
    Furthermore, these isomorphisms define a natural isomorphism
        \begin{align*}
            \gamma \colon \mathrm{Rot} \circ \mathrm{Env} \rightarrow \mathrm{Id}_{\mathbf{ErgDisc}_{\mathrm{sec}}(p)}.
        \end{align*}
\end{proposition}

\begin{proof}
    Let $(q, \varphi, \tau)$ be a relatively ergodic system with relative discrete spectrum and distinguished section. It is straightforward to show that $\gamma_{(q, \varphi, \tau)}$ defines a morphism in $\mathbf{ErgDisc}_{\mathrm{sec}}(p)$. Moreover, for a given $m \in M$, the fiber group $\mathrm{E}_{\mathrm{u}}(q, \varphi)_m$ acts effectively and, by Proposition \ref{charergtransit}, also transitively on $K_m$ for each $m \in M$. Since $\mathrm{E}_{\mathrm{u}}(q, \varphi)_m$ is abelian, this yields that the action is also free. This implies that $\gamma_{(q, \varphi, \tau)}$ is a bijection, and hence an isomorphism by Corollary \ref{automaticcontprop}. Moreover, one can readily check that $\gamma$ is indeed a natural isomorphism.
\end{proof}

\begin{corollary}\label{withsectionisomorphic}
    Let $(q, \varphi, \tau)$ and $(q',\varphi',\tau')$ be relatively ergodic system with relative discrete spectrum and distinguished section. Then $(q, \varphi, \tau)$ and $(q',\varphi',\tau')$ are isomorphic if and only if the underlying topological dynamical systems $(q,\varphi)$ and $(q',\varphi')$ are isomorphic.
\end{corollary}
\begin{proof}
    If $(q,\varphi)$ and $(q',\varphi')$ are isomorphic topological dynamical systems, then the group bundle compactifications 
        \begin{align*}
            \mathrm{Env}(q, \varphi, \tau) = (\mathrm{s}|_{\mathrm{E}_{\mathrm{u}}(q, \varphi)}, i_{(q, \varphi)}) \quad \textrm{ and } \quad \mathrm{Env}(q', \varphi', \tau') = (\mathrm{s}|_{\mathrm{E}_{\mathrm{u}}(q', \varphi')}, i_{(q', \varphi')})
        \end{align*}
    are isomorphic as well. By Proposition \ref{natisocomprot1} we obtain that $(q, \varphi, \tau)$ and $(q',\varphi',\tau')$ are isomorphic in $\mathbf{ErgDisc}_{\mathrm{sec}}(p)$. The converse implication is trivial.
\end{proof}

The following is the counterpart to Proposition \ref{natisocomprot1}.

\begin{proposition}\label{natisocomprot2}
    For every group bundle compactification $(q, c)$ of $p$ we obtain an isomorphism of group bundle compactifications
        \begin{align*}
            \eta_{(q, c)} \colon (q, c) \rightarrow \mathrm{Env}(\mathrm{Rot}(q, c)), \quad x \mapsto \vartheta_x,
        \end{align*}
    with $\vartheta_x(y) \coloneqq xy$ for all $x \in H$ and $y \in H_{q(x)}$. Additionally, these isomorphisms define a natural isomorphism
        \begin{align*}
            \mathrm{Id}_{\mathbf{Comp}(p)} \rightarrow \mathrm{Env} \circ \mathrm{Rot}.
        \end{align*}
\end{proposition}

\begin{proof}
    The first part follows directly from Proposition \ref{envforcomp}, and the second is an easy exercise.
\end{proof}

With Propositions \ref{natisocomprot1} and \ref{natisocomprot2} we have established the following categorical equivalence.

\begin{theorem}\label{equivrotcomp}
     For the open topological abelian group bundle $p \colon G \rightarrow M$ over the locally compact space $M$ the functors 
        \begin{align*}
            \mathrm{Rot}\colon \mathbf{Comp}(p) \rightarrow \mathbf{ErgDisc}_{\mathrm{sec}}(p) \quad \textrm{and} \quad \mathrm{Env} \colon \mathbf{ErgDisc}_{\mathrm{sec}}(p) \rightarrow \mathbf{Comp}(p)
        \end{align*}
    establish an equivalence between the category $\mathbf{ErgDisc}_{\mathrm{sec}}(p)$ of relatively ergodic systems with relative discrete spectrum and distinguished section and the category of group bundle compactifications $\mathbf{Comp}(p)$.
\end{theorem}

Having established Theorems \ref{maintheoremcomp} and \ref{equivrotcomp}, we are almost in a position to state the main result of our article: a Halmos--von Neumann theorem for open topological abelian group bundles. To address the missing part, recall from the introduction that in the group case the point spectrum of the Koopman representation serves as a complete isomorphism invariant for ergodic systems with discrete spectrum. Recall from Definition \ref{spectralnotions} that in our situation the point spectrum  $\sigma_{\mathrm{p}}^{T^\varphi} = (\sigma_{\mathrm{p}}^{T^\varphi}(U))_U$ of the Koopman representation of a topological dynamical system $(q,\varphi)$ on a bundle $q \colon K \rightarrow M$ is given by 
    \begin{align*}
        \sigma_{\mathrm{p}}^{T^{\varphi}}(U) \coloneqq \{\chi \in \mathcal{D}_p(U) \mid \exists f \in \ker(\chi - T^{\varphi}) \textrm{ with } \|f|_{K_m}\|_{\infty} = 1 \textrm{ for every } m \in U\}
    \end{align*}
for every open subset $U \subseteq M$. We now establish the following generalization of \cite[Proposition 4.20]{kreidler-hermle}.

\begin{proposition}\label{functorcomposition}
   The functor $\mathrm{DDual} \circ \mathrm{Env} \colon \mathbf{ErgDisc}_{\mathrm{sec}}(p) \rightarrow \mathbf{WSub}(\mathcal{D}_p)^{\mathrm{op}}$ satisfies
        \begin{align*}
            \sigma_{\mathrm{p}}^{T^{\varphi}}(U) = ((\mathrm{DDual} \circ \mathrm{Env})(q,\varphi, \tau))(U)
        \end{align*}
     for each relatively ergodic system with relative discrete spectrum and distinguished section $(q, \varphi, \tau)$ and each open subset $U \subseteq M$.  
\end{proposition}
\begin{proof}
    Let  $(q, \varphi, \tau)$ be a relatively ergodic system with relative discrete spectrum and distinguished section, and $U \subseteq M$ be an open subset. We have to show that
        \begin{align*}
            \sigma_{\mathrm{p}}^{T^{\varphi}}(U) = \{\chi' \circ (i_{(q,\varphi)})|_{p^{-1}(U)} \mid \chi' \in \mathcal{D}_{\mathrm{s}|_{\mathrm{E}_{\mathrm{u}}(q, \varphi)}}(U)\}.
        \end{align*}
    
    For the first inclusion \enquote{$\subseteq$} let $\chi \in  \sigma_{\mathrm{p}}^{T^{\varphi}}(U)$. Take $f \in \mathrm{ker}(\chi - T^{\varphi})$ with $\|f|_{K_m}\| =1$ for each $m \in U$. By Lemma \ref{lemmaunifsemigrp} we find for a given $\vartheta \in \mathrm{E}_{\mathrm{u}}(q, \varphi)$ with $\mathrm{s}(\vartheta) \in U$ some element $z \in \mathbb{T}$ with $f|_{K_{\mathrm{s}(\vartheta)}} \circ \vartheta^{-1} = z \cdot f|_{K_{\mathrm{s}(\vartheta)}} $. Since $f|_{K_{\mathrm{s}(\vartheta)}}  \neq 0$, the element $z \in \T$ is uniquely determined, and we write $z_{\vartheta}$ for it. Note that for $t \in G$ with $p(t) \in U$ we have $z_{\varphi_{t}} = \mathrm{pr}_{\T}(\chi(t))$.\medskip
    
    Now, consider the map
        \begin{align*}
            \chi' \colon \mathrm{s}|_{\mathrm{E}_{\mathrm{u}}(q, \varphi)}^{-1}(U) \rightarrow U \times \T, \quad
\vartheta \mapsto (\mathrm{s}(\vartheta), z_{\vartheta}).
        \end{align*}
    We show below that $\chi'$ is continuous. It is then easy to check that $\chi'$ is a local character of $\mathrm{s}|_{\mathrm{E}_{\mathrm{u}}(q, \varphi)}$, i.e., $\chi' \in \mathcal{D}_{\mathrm{s}|_{\mathrm{E}_{\mathrm{u}}(q, \varphi)}}(U)$. But
        \begin{align*}
            (\chi' \circ i_{(q,\varphi)})(t) = (p(t), z_{\varphi_t}) = (p(t),\mathrm{pr}_{\T}(\chi(t))) = \chi(t)
        \end{align*}
    for each $t \in p^{-1}(U)$, hence $\chi = \chi' \circ (i_{(q,\varphi)})|_{p^{-1}(U)}$.\medskip
    
    To prove that $\chi'$ is continuous, it suffices by Proposition \ref{lemmaclosedgraph} to check that the graph of $\chi'$ is closed. So take a net $(\vartheta_\alpha)_{\alpha}$ in $\mathrm{E}_{\mathrm{u}}(q, \varphi)$ with $\mathrm{s}(\vartheta_\alpha) \in U$ for all $\alpha$ converging to some $\vartheta \in \mathrm{E}_{\mathrm{u}}(q, \varphi)$ with $\mathrm{s}(\vartheta) \in U$, and assume further that $\lim_{\alpha} \chi'(\vartheta_{\alpha}) = (\mathrm{s}(\vartheta),w)$ for an element $w \in \T$. Then
        \begin{align*}
            w f|_{K_{\mathrm{s}(\vartheta)}} = \lim_{\alpha} z_{\vartheta_\alpha}f|_{K_{\mathrm{s}(\vartheta_\alpha)}}
= \lim_{\alpha} f|_{K_{\mathrm{s}(\vartheta_\alpha)}} \circ \vartheta_{\alpha}^{-1} = f|_{K_{\mathrm{s}(\vartheta)}} \circ \vartheta^{-1}
= z_{\vartheta} f|_{K_{\mathrm{s}(\vartheta)}}.
        \end{align*}
    This yields $w = z_{\vartheta}$, and hence $(\mathrm{s}(\vartheta),w) = \chi'(\vartheta)$ as desired.\medskip

    For the converse inclusion \enquote{$\supseteq$}, take a local character $\chi \in \mathcal{D}_{\mathrm{s}|_{\mathrm{E}_{\mathrm{u}}(q, \varphi)}}(U)$. By Example \ref{exampleeigensect} we have that $\chi \circ (i_{(q,\varphi)})|_{p^{-1}(U)}$ is an eigenvalue of the rotation system $\mathrm{Rot}(\mathrm{Env}(q, \varphi, y))$ defined by the compactification $(\mathrm{s}|_{\mathrm{E}_{\mathrm{u}}(q, \varphi)}, i_{(q,\varphi)})$. Since, by Proposition \ref{natisocomprot1}, we have an isomorphism 
        \begin{align*}
              \gamma_{(q, \varphi, \tau)} \colon \mathrm{Rot}(\mathrm{Env}(q, \varphi, y)) \rightarrow (q, \varphi, y), \quad \vartheta \mapsto \vartheta(\tau(\mathrm{s}(\vartheta))),
        \end{align*}
    we conclude that $\chi \circ (i_{(q,\varphi)})|_{p^{-1}(U)} \in \sigma_{\mathrm{p}}^{T^{\varphi}}(U)$ as well.
\end{proof}

Motivated by Proposition \ref{functorcomposition} we introduce the following definition.

\begin{definition}
    The composition  $\sigma_{\mathrm{p}} \coloneqq \mathrm{DDual} \circ \mathrm{Env} \colon \mathbf{ErgDisc}_{\mathrm{sec}}(p) \rightarrow \mathbf{WSub}(\mathcal{D}_p)^{\mathrm{op}}$ is called the \textbf{point spectrum functor}. 
\end{definition}

We are finally ready to state our main theorem, which extends \cite[Theorem 4.25]{kreidler-hermle} to the setting of group bundles (cf. \cite[Theorem 5.2.31]{Herm2026} for topologically ergodic groupoids).

\begin{theorem}\label{hvncat}
    For the open topological abelian group bundle $p \colon G \rightarrow M$ over the locally compact space $M$ the following assertions hold.
        \begin{enumerate}[(i)]
            \item The functors 
                \begin{align*}
                    \mathrm{Rot}&\colon \mathbf{Comp}(p) \rightarrow \mathbf{ErgDisc}_{\mathrm{sec}}(p),\\
                    \mathrm{Env}&\colon \mathbf{ErgDisc}_{\mathrm{sec}}(p) \rightarrow \mathbf{Comp}(p)
                \end{align*}
            establish an equivalence between the category of relatively ergodic systems with relative discrete spectrum and distinguished section over $p$ and the category of group bundle compactifications of $p$.
            \item The functors 
                \begin{align*}
                    \mathrm{DDual} &\colon \mathbf{Comp}(p) \rightarrow \mathbf{WSub}(\mathcal{D}_p)^{\mathrm{op}},\\
                    \mathrm{CDual} &\colon \mathbf{WSub}(\mathcal{D}_p)^{\mathrm{op}} \rightarrow \mathbf{Comp}(p)
                \end{align*}
            establish an equivalence between the category of group bundle compactifications of $p$ and the opposite category of well-supported subsheaves of the dual sheaf of $p$.
            \item The functor 
                \begin{align*}
                    \sigma_{\mathrm{p}} \colon \mathrm{DDual} \circ \mathrm{Env} \colon \mathbf{ErgDisc}_{\mathrm{sec}}(p) \rightarrow \mathbf{WSub}(\mathcal{D}_p)^{\mathrm{op}}
                \end{align*}
            defines an equivalence between  the category of relatively ergodic systems with relative discrete spectrum and distinguished section and the opposite category of well-supported subsheaves of the dual sheaf of $p$.
        \end{enumerate}
\end{theorem}

The following \enquote{category theory free} formulation is closer to the original Halmos--von Neumann theorem from the introduction. It is a direct consequence of Theorem \ref{hvncat} and Corollary \ref{withsectionisomorphic}.

\begin{theorem}\label{hvnmain}
    For relatively ergodic systems with relative discrete spectrum admitting a continuous section the following assertions hold.
        \begin{enumerate}[(i)]
            \item Two such systems $(q_1,\varphi_1)$ and $(q_2,\varphi_2)$ are isomorphic if and only if $\sigma_{\mathrm{p}}^{T^{\varphi_1}} = \sigma_{\mathrm{p}}^{T^{\varphi_2}}$.
            \item  The point spectra $\sigma_{\mathrm{p}}^{T^{\varphi}}$ arising from such systems $(q,\varphi)$ are precisely the well-supported subsheaves of the dual sheaf $\mathcal{D}_p$.
            \item Every such system is isomorphic to a rotation system $(q,\varphi_c)$ defined by a group bundle compactification $(q,c)$ of $p$.
        \end{enumerate}
\end{theorem}

\begin{remark}
    Let us conclude this work with some open problems.
    \begin{enumerate}[(i)]
        \item Via topological models the classical Halmos--von Neumann theorem for continuous group actions can be transferred to ergodic theory (see, e.g.,   \cite[Theorems 3.8 and 4.5]{kreidler-hermle}). It seems reasonable to expect that one can also use Theorem \ref{hvnmain} to infer a Halmos--von Neumann theorem for measure-preserving systems similar to \cite[Theorem 5.22]{ADKPL2026}.
        \item As shown in \cite[Section 5]{kreidler-hermle}, there is also a topological Halmos--von Neumann theorem for continuous actions of non-abelian topological groups. It would be interesting to generalize this result to open group bundles over a locally compact base space.
        \item The second author develops a \enquote{Halmos--von Neumann theory} for topologically ergodic groupoids in his upcoming PhD thesis \cite{Herm2026}. It rests on representation theory for compact \emph{transitive} groupoids $H$. These can -- from a groupoid point of view -- be seen as the counterparts to proper group bundles, which are fundamental to this article. It is therefore a natural question whether there is a common generalization of both theories, maybe even for actions of arbitrary topological groupoids.
    \end{enumerate}
    We hope to address some of these problems in future work.
\end{remark}

\parindent 0pt
\parskip 0.5\baselineskip
\setlength{\footskip}{4ex}
\bibliographystyle{alpha}
\bibliography{bibliography} 
\footnotesize

\end{document}